\documentclass[1 [leqno,11pt]{amsart}
\usepackage{amssymb, amsmath,latexsym,amsfonts,amsbsy, amsthm,mathtools,graphicx,CJKutf8,CJKnumb,CJKulem,color}
\usepackage{float}
\usepackage{hyperref}

\numberwithin{equation}{section}

\allowdisplaybreaks

\let\al=\alpha

\let\f=\frac

\let\na=\nabla

\let\pa=\partial

\def\R{\mathbb R}
\def\Z{\mathbb Z}

\def\an{\langle n\rangle}

\def\no{\noindent}

\newcommand{\beq}{\begin{equation}}
\newcommand{\eeq}{\end{equation}}
\newcommand{\ben}{\begin{eqnarray}}
\newcommand{\een}{\end{eqnarray}}
\newcommand{\beno}{\begin{eqnarray*}}
\newcommand{\eeno}{\end{eqnarray*}}

\newcommand{\reff}[1]{(\ref{#1})}

\newtheorem{theorem}{Theorem}[section]

\newtheorem{lemma}[theorem]{Lemma}
\newtheorem{proposition}[theorem]{Proposition}

\newtheorem{remark}[theorem]{Remark}

\begin{document}
\begin{CJK*}{UTF8}{gkai}

\title[$L^\infty$ stability of Prandtl expansions]
{On the $L^\infty$ stability of Prandtl expansions in Gevrey class}

\author{Qi Chen}
\address{School of Mathematical Science, Peking University, 100871, Beijing, P. R. China}
\email{chenqi940224@gmail.com}

\author{Di Wu}
\address{School of Mathematical Science, Peking University, 100871, Beijing, P. R. China}
\email{wudi@math.pku.edu.cn}

\author{Zhifei Zhang}
\address{School of Mathematical Science, Peking University, 100871, Beijing, P. R. China}
\email{zfzhang@math.pku.edu.cn}

\date{\today}

\maketitle

\begin{abstract}
In this paper, we prove the $L^\infty\cap L^2$ stability of  Prandtl expansions of shear flow type as $\big(U(y/\sqrt{\nu}),0\big)$
for the initial perturbation in the Gevrey class, where $U(y)$ is a monotone and concave function and $\nu$ is the viscosity coefficient.
To this end, we develop the direct resolvent estimate method for the linearized Orr-Sommerfeld operator instead of the Rayleigh-Airy iteration method. Our method could be used to the other relevant problems in the hydrodynamic stability.
\end{abstract}

\section{Introductions}

In this paper, we study the incompressible Navier-Stokes equations in $\Omega:=\mathbb{T}\times\mathbb{R}_+$ when the viscosity coefficient $\nu$ tends to zero:
\begin{align}\label{NS}
\left\{
\begin{aligned}
&\partial_t u^\nu+u^\nu\cdot\nabla u^\nu+\nabla p^\nu-\nu\Delta u^\nu=f^\nu\quad \text{in $[0,T]\times\Omega$},\\
&\nabla\cdot u^\nu=0\quad \text{in $[0,T]\times\Omega$},\\
&u^\nu|_{\partial\Omega}=0\quad\text{on $[0,T]\times\partial\Omega$},\\
&u^\nu(0)=u_0\quad\text{in $\Omega$}.
\end{aligned}
\right.
\end{align}
Here $u^\nu=\big(u^\nu_1, u^\nu_2\big)$ is the velocity field, $p^\nu$ is the pressure and $f^\nu$ is the external force.  \smallskip

In the absence of the boundary, the solution $u^\nu$ of the Navier-Stokes equations  converges to the solution $u^e$ of the Euler equations as $\nu \to 0$:
\begin{align*}
\pa_t u^{e}+ u^{e}\cdot\na u^{e}+\na p^{e}=0,\quad  \na \cdot u^{e}=0.
\end{align*}
This limit has been justified in various functional settings \cite{Kato, Swann, BM, Mas, CW, AD, Mar}.

In the presence of the boundary, the inviscid limit problem becomes more complicated due to the appearance of boundary layer.
For the Navier-slip boundary condition, since the boundary layer is weak, the limit from the Navier-Stokes equations to the Euler equations was  justified in {2-D} by Clopeau, Mikeli\'{c} and Robert \cite{CMR}, and in {3-D} by Iftimie and Planas \cite{IP}.
See \cite{MR, IS, XX, WXZ}  for more relevant results. For the nonslip boundary condition, the boundary layer is strong so that  when $\nu \to 0$, the solution of \eqref{NS} formally behaves as
\ben\label{eq:Pran-exp}
 \left\{
 \begin{array}{l}
 u^{\nu}_1(t,x,y)=u^{e}_1(t,x,y)+ u^{BL}\big(t,x,\f{y}{\sqrt{\nu}}\big)+O(\sqrt{\nu}),\\
 u^{\nu}_2(t,x,y)=u^{e}_2(t,x,y)+\sqrt{\nu}v^{BL}\big(t,x,\f{y}{\sqrt{\nu}}\big)+O(\sqrt{\nu}),
 \end{array}\right.
 \een
where $(u^p,v^p)=\big(u^e_1(t,x,0)+u^{BL}(t,x,Y), \pa_yu^e_2(t,x,0)Y+v^{BL}(t,x,Y)\big)$ satisfies the Prandtl equation
 \begin{eqnarray}\label{equ:P}
\left\{\begin{aligned}
&\partial_t u^p+u^p\pa_x u^p+v^p\pa_Y u^p+\pa_xp^e|_{y=0}=\pa^2_{Y}u^p,\\
&\pa_xu^p+\pa_Y v^p=0,\\
&u^p|_{Y=0}=v^p|_{Y=0}=0,\quad\lim_{Y\rightarrow+\infty}u^p(t,x,Y)=u^e_1(t,x,0).
\end{aligned}\right.
\end{eqnarray}

To our knowledge, the justification of the Prandtl expansion \eqref{eq:Pran-exp} is still a challenging problem except some special cases:
 the analytic data  \cite{SC2}(see \cite{WWZ} for a new proof via direct energy method), and
 the initial vorticity vanishing near the boundary \cite{Mae, FTZ}. In addition, the convergence was justified in \cite{LMT, MT}  when the domain and the initial data have a circular symmetry. Initiated by Kato \cite{Kato1}, there are many works devoted to the conditional convergence \cite{TW, Wang, Ke}. Let us mention some recent well-posedness results of the Prandtl equation \cite{SC1, XZ, AW, MW, GD, GM, CWZ, LY, DG}, which are a key step toward the inviscid limit problem. 
 \smallskip

Recently,  the stability of some special boundary layer solutions received a lot of attention.  For example, Grenier studied the Prandtl expansion of shear type flow  as
\ben\label{eq:Prandtl exp-shear}
u^\nu_s=\big(U^e(t,y),0\big)+\Big(U^{BL}\big(t,\f y {\sqrt{\nu}}\big),0\Big).
\een
When the shear flow $U^{BL}(0,Y)$ is linearly unstable for the Euler equations, he proved the instability of the expansion in the $H^1$ space by constructing the solution with the highly oscillating as $e^{i\al x/\sqrt{\nu}}$ and the growth as $e^{ct|n|}$ at the high frequency $n=\frac 1 {\sqrt{\nu}}\gg1 $.  In a recent important work, Grenier and Nguyen \cite{GNg} proved the $L^\infty$ instability of the Prandtl expansion \eqref{eq:Prandtl exp-shear}. In another important work, Guo,  Grenier and Nguyen proved that the shear flows which are linearly stable for the Euler equations could be linearly unstable for the Navier-Stokes equations when $\nu$ is very small, where they constructed the solution with the growth as $e^{c|n|^\frac 23t}$. Their result in particular implies that it is possible to prove the stability of monotone and concave shear flows  in the Gevrey class $\frac 32$. In a remarkable work \cite{GMM}, Gerard-Varet,  Masmoudi  and  Maekawa proved the stability of the Prandtl expansion \eqref{eq:Prandtl exp-shear}
for the perturbations in the Gevrey class when $U^{BL}(t,Y)$ is a monotone and concave function. Roughly speaking, they showed that if the initial perturbation $a(x,y)$ satisfies $\|a\|_{G_\gamma}\le \nu^{\frac 12+\beta}$ for $\beta=\frac {2(1-\gamma)} \gamma$, where $G_\gamma$ is a norm of Gevrey class $\frac 1{\gamma}$ with $\gamma\in (0,1]$ depending on the profile $U^{BL}(t,Y)$,  then
\beno
\sup_{t\in [0,T]}\big(\|v^\nu(t)\|_{L^2}+(\nu t)^\frac 12\|\na v^\nu(t)\|_{L^2}+(\nu t)^\frac 14\|v^\nu(t)\|_{L^\infty}\big)\le C\|a\|_{G_\gamma},
\eeno
where $v^\nu(t,x,y)=u^\nu(t,x,y)-\big(U^e(t,y),0\big)+\big(U^{BL}\big(t,\frac y {\sqrt{\nu}}\big),0\big)$. The $L^\infty$ stability estimate, which is in fact an interpolation result  between $L^2$ estimate and $H^1$ estimate, will blow up  when $t\to 0$ due to the prefactor $t^{\frac 14}$. Their proof relies on the resolvent estimates for the linearized  operator via the Rayleigh-Airy iteration method introduced in \cite{GGN}.

For the steady Navier-Stokes equations,
Gerard-Varet and Maekawa \cite{GMa} proved the stability of  shear flows  $\big(U\big(\frac y {\sqrt{\nu}}\big),0\big)$  for the external force $f^\nu$ in the Sobolev space, and Guo and  Iyer  \cite{GI} proved the stability of Blasius flows. Guo and Nguyen \cite{GN} also considered the Prandtl expansions of steady Navier–Stokes equations over a moving plate.

\smallskip

The goal of this paper is twofold. First of all, we would like to prove the $L^\infty$ stability estimate of  the Prandtl expansion \eqref{eq:Prandtl exp-shear} without the prefactor $t^\frac 14$, i.e., $\nu^\frac 14\|v^\nu(t)\|_{L^\infty}\le C\|a\|_{G_\gamma}$.
Secondly, we would like to develop a direct resolvent estimate method for the linearized  operator instead of the Rayleigh-Airy iteration method.
\smallskip

For the simplicity, we take $U^e(t,y)\equiv 1$ and  $U^{BL}(t,Y)=U^P(Y)-1$ in \eqref{eq:Prandtl exp-shear}, where $U^P=U^P(Y)$ is a scalar function on $\mathbb{R}_+$ satisfying
\[
\lim_{Y\to+\infty} U^P(Y)=1,\quad U^P(Y=0)=0.
\]
Taking the external force
\[
f^\nu=\Big(-\partial_Y^2 U^P\big(\frac{y}{\sqrt{\nu}}\big),0\Big),
\]
we may write the solution of \reff{NS} in the perturbation form
\[
u^\nu(t,x,y)=\Big(1+U^{BL}\big(\frac y {\sqrt{\nu}}\big),0\Big)+u(t,x,y)
\]
with
\begin{align}\label{NSP}
\left\{
\begin{aligned}
&\partial_t u-\nu\Delta u+u\cdot\nabla u+U^P\big(\frac{y}{\sqrt{\nu}}\big)\partial_x u+\nu^{-1/2}u_2\Big(\partial_Y U^P\big(\frac{y}{\sqrt{\nu}}\big),0\Big)+\nabla p=0,\\
&\nabla\cdot u=0,\\
&u|_{\partial\Omega}=0,\quad u(0)=a.
\end{aligned}
\right.
\end{align}
The system \reff{NSP} can be written as
\begin{align}\label{NSA}
\left\{
\begin{aligned}
&\partial_t u+\mathbb{A}_\nu u=-\mathbb{P}(u\cdot\nabla u),\\
&u|_{t=0}=a,
\end{aligned}
\right.
\end{align}
where $\mathbb{P}: L^2(\Omega)^2\to L^2_\sigma(\Omega)$ is the Helmoholtz-Leray projection and
\begin{align*}
\mathbb{A}_\nu u=-\nu\Delta u+\mathbb{P}\Big(U^P\big(\frac{y}{\sqrt{\nu}}\big)\pa_xu+\nu^{-1/2}u_2\Big(\partial_Y U^P\big(\frac{y}{\sqrt{\nu}}\big),0\Big)\Big)
\end{align*}
with the domain
\[
D(\mathbb{A}_\nu)=W^{2,2}(\Omega)^2\cap W^{1,2}_0(\Omega)^2\cap L^2_\sigma(\Omega).
\]

Before stating main result,  we introduce some notations and functional spaces.
 Let
\begin{align*}
(\mathcal{P}_n f)(y)=f_n(y) e^{inx},\quad f_n(y)=\frac{1}{2\pi}\int_0^{2\pi}f(x,y)e^{-inx}dx,
\end{align*}
be the projection on the Fourier mode  $n\in \Z$ in $x$. We inroduce the following Gevrey class. For $\gamma\in(0,1], d\geq0$ and $K>0$,
\begin{align*}
&X_{d,\gamma,K}:=\big\{f\in L^2_\sigma(\Omega):\|f\|_{X_{d,\gamma,K}}=\sup_{n\in\mathbb{Z}}(1+|n|^d)e^{K|n|^\gamma}\|\mathcal{P}_nf\|_{L^2(\Omega)}<+\infty\big\},\\
&X^{(1)}_{d,\gamma,K}:=\big\{f\in L^2_\sigma(\Omega):\|f\|_{X_{d,\gamma,K}^{(1)}}=\sup_{n\in\mathbb{Z}}(1+|n|^d)e^{K|n|^\gamma}\|\mathcal{P}_nf\|_{L^2_xH^1_y(\Omega)}<+\infty\big\}.
\end{align*}
When $\gamma=1$, the functions in $X_{d,\gamma, K}$ are analytic in $x$, where $K$ corresponds to the analytic width.   

Next we introduce the {\bf strongly concave(SC) condition} on shear flows:\smallskip

\begin{enumerate}

\item $U(Y)\in BC^2(\mathbb{R}_+)$ with
\[
\|U\|:=\sum_{k=0,1,2}\sup_{Y\geq0}(1+Y)^k|\partial^k_YU(Y)|<\infty.
\]

\item $U|_{Y=0}=0,\,\,\lim_{Y\to+\infty}U(Y)=1$.

\item There exists $M>0$ such that $-M\partial_Y^2 U\geq(\partial_Y U)^2$ and $|\partial_Y^3U/\partial_Y^2 U|+|\partial_Y^2U/\partial_Y U|\leq M$ for any $Y\geq 0$.
\end{enumerate}

Now we  state our main result as follows. 

\begin{theorem}\label{main}
Assume that $U^P$ satisfies (SC) condition. For $\gamma\in[2/3,1), d>5-3\gamma$ and $K>0$, there exists $C,\epsilon, T$ and $K'\in(0,K)$ such that for any sufficiently small $\nu>0$, if $\|a\|_{X^{(1)}_{d,\gamma,K}}\leq \epsilon\nu^{\frac{1}{2}+\beta}$ with 
$\beta=\max\Big\{\f {7(1-\gamma)} {8\gamma}+\f1 {8\gamma}+,\f{3} {16}+\frac{15(1-\gamma)}{16}\Big\},$
 then the system \reff{NSP} admits a unique solution $u\in C([0,T];L^2_\sigma(\Omega))\cap L^2(0,T;W^{1,2}_0(\Omega))$ satisfying
\begin{align*}
&\sup_{0<t\leq T}\big(\|u(t)\|_{X_{d,\gamma,K'}}+\nu^{\frac{1}{4}}\|u(t)\|_{L^\infty(\Omega)}+(\nu t)^{\frac{1}{2}}\|\nabla u(t)\|_{L^2(\Omega)}\big)\leq C\|a\|_{X^{(1)}_{d,\gamma,K}}.
\end{align*}
\end{theorem}

\begin{remark}
Let us give several remarks on our result.

\begin{enumerate}
\item In Theorem \ref{main}, we obtain the $L^\infty$ stability of Prandtl expansion in Gevrey class, i.e., 
\beno
\|u(t)\|_{L^\infty}\le C\nu^{-\f14}\|a\|_{X^{(1)}_{d,\gamma,K}}\le C\nu^{\f14+\beta}.
\eeno

\item The smallness requirement on the initial perturbation  should be not optimal. It is a very interesting problem to investigate whether  the power $\f 12+\beta$ of $\nu$ can be improved to $\f12$.

\item Our method could be used to the case when $U(Y)$ satisfies weakly concave(WC) condition: $-M_\sigma\partial_Y^2 U\geq(\partial_Y U)^2$ for $Y\ge \sigma>0$ or $U(t, Y)$ is  time dependent and weakly concave.

\item In a recent remarkable work \cite{GMM1} by Gerard-Varet, Maekawa and Masmoudi, they show that even for general concave Prandtl profiles, the Prandtl expansion is stable (including in $L^\infty$ sense) if initial perturbation is smaller than $\nu^\f94$ in Gevrey class $\f32$.  The key difference is that  for the shear flows, one can obtain the semigroup estimates global in time via the resolvent estimates, which should be important to study the instability of the Prandtl expansion in the Sobolev space.  
\end{enumerate}
\end{remark}

\section{Sketch of the proof}

The roadmap of the proof of Theorem \ref{main} is similar to \cite{GMM}.  We first  focus on the linearized system of \reff{NSP} and obtain the corresponding semigroup estimates via the resolvent estimates for the linearized Orr-Sommerfeld operator.  Then using the Duhamel formula of the solution, we prove the nonlinear stability by combining the semigroup estimates with nonlinear estimates.

\subsection{Reformulation of the problem}

As in \cite{GMM}, it is more convenient to introduce the rescaled velocity
\begin{align}\label{eq:rescale}
   & u(t,x,y)=v(\tau,X,Y),\quad (\tau,X,Y)=\Big(\frac{t}{\sqrt{\nu}},\frac{x}{\sqrt{\nu}},\frac{y}{\sqrt{\nu}}\Big).
\end{align}
If $u(t)=e^{-t\mathbb{A}_\nu}a$, then $v$ is the solution to the system
\begin{align}\label{eq:sl-ns}\left\{\begin{aligned}
&\partial_{\tau}v-\sqrt{\nu}\Delta_{X,Y} v+ \big(v_2\partial_YV,0\big) + V\partial_Xv+\nabla_{X,Y} q =0\quad \text{in}\,\,\Omega_{\nu},\\
&\text{div}_{X,Y}\ v=0,\ v|_{Y=0}=0,\ v|_{t=0}=a^{(\nu)},
\end{aligned}\right.
\end{align}
where $\Omega_\nu:=(\nu^{-1/2}\mathbb{T})\times\mathbb{R}_+$ and $a^{(\nu)}:=a(\nu^{1/2}X,\nu^{1/2}Y)$. So, $v$, $\nabla q$ are $\frac{2\pi}{\sqrt{\nu}}$-periodic in $X$.

We introduce the linearized operator 
\begin{eqnarray}\label{rescale-operator}
\mathbb{L}_\nu v=-\sqrt{\nu}\mathbb{P}_\nu\Delta v+\mathbb{P}_\nu\big(V\partial_X v+v_2(\partial_Y V,0)\big)
\end{eqnarray}
with $D(\mathbb{L}_\nu)=W^{2,2}(\Omega_\nu)\cap W^{1,2}_0(\Omega_\nu)\cap L^2_\sigma(\Omega_\nu)$, where $\mathbb{P}_\nu$
is the rescaled Helmoholtz-Leray projection. To establish the resolvent estimates of $\mathbb{L}_\nu$,  we consider the following resolvent problem for $\mu\in \mathbb{C}$:
\begin{align}\label{eq:sc-res}\left\{\begin{aligned}
&\mu v-\sqrt{\nu}\Delta_{X,Y} v + \big(v_2\partial_YV,0\big) + V\partial_Xv+\nabla_{X,Y} q =f,\quad Y\geq 0,\\
&\text{div}_{X,Y}\ v=0,\ v|_{Y=0}=0.
\end{aligned}\right.
\end{align}

 Let $w=\partial_Xv_{2}-\partial_Yv_{1}$ be the vorticity field of $v$. Direct computation gives
\begin{align*}
   &\mu w-\sqrt{\nu}\Delta w-v_2\partial_Y^2V+V\partial_xw=\partial_Xf_2-\partial_Yf_{1}.
\end{align*}
The corresponding stream function denoted by $\phi$ solves
\begin{align*}
   &\Delta\phi=w,\quad \phi|_{Y=0}=0.
\end{align*}
Then $v=(-\partial_Y\phi,\partial_X\phi)$. 

Since the shear flow is independent of $x$, it is natural to study the problem on each Fourier mode with respect to the $x$ variable. So, we introduce 
\begin{align*}
(\mathcal{P}_{\nu,n}f)(Y)=f_n(Y)e^{\mathrm{i}n\sqrt{\nu}X},\quad f_n(Y)=\frac{\sqrt{\nu}}{2\pi}\int_0^{\frac{2\pi}{\sqrt{\nu}}}f(X,Y)e^{-\mathrm{i}n\sqrt{\nu}X}\mathrm{d}X.
\end{align*}
Let $\phi_n(Y)$ be the $n\sqrt{\nu}$ fourier modes of $\phi(X,Y)$. Then  $\phi_n$ solves the following system
\begin{align}\label{eq:sc-resphi}\left\{\begin{aligned}
&\mu (\partial_Y^2-n^2\nu)\phi_n-\sqrt{\nu}(\partial_Y^2-n^2\nu)^2 \phi_n - \mathrm{i}n\sqrt{\nu}\phi_n(\partial_Y^2V) + \mathrm{i}n\sqrt{\nu}V(\partial_Y^2-n^2\nu)\phi_n\\
&\quad=\mathrm{i}n\sqrt{\nu}f_{2,n}-\partial_Yf_{1,n},\\
&\phi_n|_{Y=0}=\partial_Y\phi_n|_{Y=0}=0.
\end{aligned}\right.
\end{align}
Let
\beno
\lambda:= \dfrac{\mathrm{i}\mu}{|n|\sqrt{\nu}},\quad \alpha:=\nu^{1/2}|n|.
\eeno
 Removing the subscript $n$ for $(\phi_n, f_{i,n})$,  we obtain the Orr-Sommerfeld equation
\begin{align}\label{eq:sc-resphi1}\left\{\begin{aligned}
&-\sqrt{\nu}(\partial_Y^2-\alpha^2)^2 \phi + \mathrm{i}\alpha\big((V-\lambda)(\partial_Y^2-\alpha^2)\phi- (\partial_Y^2V)\phi\big)
=\mathrm{i}\alpha f_{2}-\partial_Yf_{1},\\
&\phi|_{Y=0}=\partial_Y\phi|_{Y=0}=0,
\end{aligned}\right.
\end{align}
Thus, the problem is reduced to solve the Orr-Sommerfeld equation.

\subsection{Resolvent estimates}

We denote by $\mathbb{L}_{\nu,n}$ the restriction of  $\mathbb{L}_\nu$ on the subspace $\mathcal{P}_{\nu, n}L^2_\sigma(\Omega_\nu)$.
The resolvent estimates for $(\mu-\mathbb{L}_{\nu,n})^{-1}$ has been reduced to solve the Orr-Sommerfeld equation.
In \cite{GMM}, Gerard-Varet,  Masmoudi  and  Maekawa solve the Orr-Sommerfeld equation by developing the Rayleigh-Airy iteration method introduced in \cite{GGN}. Motivated by our work \cite{CLWZ} on the stability of Couette flow,  we develop a direct energy method to solve 
the Orr-Sommerfeld equation for the general shear flows, which may be of independent interest,  and could be used to the relevant problems in the hydrodynamic stability. Our method are composed of the following two key steps.\smallskip

\no{\bf Step 1.} Solving the Orr-Sommerfeld equation with Navier-slip boundary condition:
\begin{align*}\left\{\begin{aligned}
&-\sqrt{\nu}(\partial_Y^2-\alpha^2)w + \mathrm{i}\alpha\big((V-\lambda)w-(\partial_Y^2V)\phi\big)=F,\\
&(\partial_Y^2-\alpha^2)\phi=w,\ w|_{Y=0}=\phi|_{Y=0}=0.
\end{aligned}\right.
\end{align*}
We first decompose the solution of $w$ as $w=w_1+w_2$, where 
\begin{align*}\left\{
\begin{aligned}
&-\sqrt{\nu}(\partial_Y^2-\alpha^2)w_1+\mathrm{i}\alpha \big(\partial_Y\big((V-\lambda)\phi'_1\big)-(V-\lambda)\alpha^2\phi_1-V''\phi_1\big) =F,\\
&(\partial_Y^2-\alpha^2)\phi_1=w_1,\quad w_1|_{Y=0}=\phi_1|_{Y=0}=0,
\end{aligned}\right.\end{align*}
and
\begin{align*}\left\{
\begin{aligned}
&-\sqrt{\nu}(\partial_Y^2-\alpha^2)w_2+\mathrm{i}\alpha\big((V-\lambda)w_2 -V''\phi_2\big) =V'h,\\
&(\partial_Y^2-\alpha^2)\phi_2=w_2,\quad w_2|_{Y=0}=\phi_2|_{Y=0}=0,
\end{aligned}\right.\end{align*}
with $h=\mathrm{i}\alpha\partial_Y\phi_1$. 

Thanks to good boundary condition and good structure on $w_1$, the estimates for $w_1$ is direct by making a inner $L^2$ product with $\phi_1$ to the equation of $w_1$.  The estimates of $w_2$ are based on two tricks. 
First of all, by making $L^2$ inner product estimate with $w/(V''-\varsigma)$($\varsigma$ small positive constant) to the equation of $w_2$, 
we can prove the following weighted estimate on $w_2$:
\begin{align*}
     & \sqrt{\nu} \left\|\dfrac{(\partial_Yw_2,\alpha w_2)}{|V''-\varsigma|^{\f12}}\right\|_{L^2}^2 +\alpha\lambda_i \left\|\dfrac{w_2}{|V''-\varsigma|^{\f12}}\right\|_{L^2}^2\leq C(\alpha\lambda_i)^{-1}\|h\|_{L^2}^2 +C\varsigma(\alpha\lambda_i)^{-1}\|\alpha\phi_2\|_{L^2}^2.
  \end{align*}
 Secondly, to estimate the stream function $\phi_2$, we use the Rayleigh's trick. To this end, we consider the following inhomogeneous Rayleigh equation
 \beno
\textsl{R}\phi:= (V-\lambda)(\partial_Y^2-\alpha^2)\phi-V''\phi=V'h_1+\partial_Yh_2+\mathrm{i}\alpha h_3,\quad \phi(0)=0. 
\eeno
Using Rayleigh's trick(Lemma \ref{lem:GMMray1}), we can show that 
  \begin{align*}
     & \|(\partial_Y\phi,\alpha\phi)\|_{L^2}\leq C\big(\lambda_i^{-1}\|h_1\|_{L^2} +\lambda_i^{-2}\|(h_2,h_3)\|_{L^2}\big).
  \end{align*}
Then we rewrite the equation of $w_2$ as
\begin{align*}
   &\mathrm{i}\alpha(\textsl{R}\phi_2) =\sqrt{\nu}(\partial_Y^2-\alpha^2)w_2+V'h.
\end{align*}
Applying the above estimate with $h_1=h/(\mathrm{i}\alpha),\ h_2=\sqrt{\nu}\partial_Yw_2/(\mathrm{i}\alpha),\ h_3=\sqrt{\nu} w_2$, we get
\begin{align*}
   \|(\partial_Y\phi_2,\alpha\phi_2)\|_{L^2}  \leq &C(\alpha\lambda_i)^{-1}\|h\|_{L^2}+ C\nu^{\f12}\lambda_i^{-2}\alpha^{-1}\left\|\dfrac{(\partial_Yw_2,\alpha w_2)}{|V''-\varsigma|^{\f12}}\right\|_{L^2},
\end{align*}
which along with the weighted estimate on $w_2$ will show that 
\begin{align*}
  & \alpha\lambda_i\|(\partial_Y\phi_2,\alpha\phi_2)\|_{L^2}\leq C\|h\|_{L^2},\\
&\nu^{\f14}(\alpha\lambda_i)^{\f12}\left\|(\partial_Yw_2,\alpha w_2)/V'\right\|_{L^2}+\alpha\lambda_i\left\|w_2/V'\right\|_{L^2}\leq C\|h\|_{L^2}.
\end{align*}

\no{\bf Step 2.} Boundary layer corrector.\smallskip

To match nonslip boundary condition, we need to introduce the boundary layer corrector:\begin{align}\label{eq:Hombound-OSnon}
\left\{\begin{aligned}\nonumber
&-\sqrt{\nu}(\partial_Y^2-\alpha^2)W_b+\mathrm{i}\alpha\big((V-\lambda)W_b- V''\Phi_b\big)=0,\\
&(\partial_Y^2-\alpha^2)\Phi_b=W_b,\\
&\Phi_b|_{Y=0}=0,\ \partial_Y\Phi_b|_{Y=0}=1.
\end{aligned}\right.
\end{align}
It is not easy to solve the above system directly. Instead,  we solve the following system
\begin{align*}
\left\{\begin{aligned}
&-\sqrt{\nu}(\partial_Y^2-\alpha^2)W+\mathrm{i}\alpha\big((V-\lambda)W- V''\Phi\big)=0,\\
&(\partial_Y^2-\alpha^2)\Phi=W,\\
&\Phi|_{Y=0}=0,\ \partial_Y^2\Phi|_{Y=0}=1.
\end{aligned}\right.
\end{align*}
The advantage is that  the solution $W$ of this system  can be well approximated by the following scaled Airy function: 
  \begin{align*}
     & W_a (Y)= Ai\big(\mathrm{e}^{\mathrm{i}\frac{\pi}{6}}|nV'(0)|^{\f13}(Y+d)\big)/ Ai\big(\mathrm{e}^{\mathrm{i}\frac{\pi}{6}}|nV'(0)|^{\f13}d\big),
  \end{align*}
where  $d=-\lambda_\nu/V'(0)$ with $\lambda_\nu=\lambda+\mathrm{i}\sqrt{\nu}\alpha$. 
Further, we observe that  $\pa_Y\Phi_a(0)\neq 0$.

Next we need to solve the error  equation $W_e=W-W_a$, which is given by 
 \begin{align}\nonumber\left\{\begin{aligned}
     &-\sqrt{\nu}(\partial_Y^2-\alpha^2)W_e+\mathrm{i}\alpha\big((V-\lambda)W_e-V''\Phi_e\big) =-\mathrm{i}\alpha\big(V-V'(0)Y\big)W_a+\mathrm{i}\alpha V''\Phi_a,\\
     &(\partial_Y^2-\alpha^2)\Phi_e=W_e,\quad \Phi_e(0)=W_e(0)=0.
     \end{aligned}\right.
  \end{align}
 Since $W_e$ satisfies the Navier-slip boundary condition, we can use the resolvent estimates established in step 1 to obtain various 
 estimates of $W_e$. In particular,  we have $\pa_Y\Phi(0)\neq 0$. Thus, $W_b=W(Y)/\pa_Y\Phi(0)$ is well-defined.
 
 Finally, the solution $w$ of the Orr-Sommerfeld equation with nonslip boundary condition is given by
 \beno
 w(Y)=w_{Na}(Y)-\pa_Y\Phi_{Na}W_b(Y),
 \eeno
where $(w_{Na}, \phi_{Na})$ is the solution of the Orr-Sommerfeld equation with Navier-slip boundary condition.\smallskip

Let us refer to section 3 for more details.

 \subsection{Semigroup estimates}
 We follow similar ideas in \cite{GMM} from the resolvent estimates to the $L^2-L^2$ and $L^2-H^1$ semigroup estimates, which is based on the formula 
 \begin{align}\nonumber
 e^{-\tau\mathbb{L}_{\nu,n}}=\frac{1}{2\pi\mathrm{i}}\int_{\Gamma}e^{\tau\mu}(\mu+\mathbb{L}_{\nu,n})^{-1}\mathrm{d}\mu
\end{align}
with a suitable contour $\Gamma$ which lies in the resolvent set of $\mathbb{L}_{\nu,n}$. 

To obtain $L^\infty$ stability result, 
we need to establish two kinds of new semigroup estimates: (1) $L^2-L^\infty$ semigroup estimates, which will be used to control the inhomogeneous part of the solution of nonlinear system; (2) $H^1-L^\infty$ semigroup estimates, which will be used to control the homogeneous part of the solution.  Let us point out that our $L^\infty$ semigroup estimates  are not  a simple consequence of  the interpolation between $L^2$ estimate and $H^1$ estimate. New weighted $H^1$ resolvent estimate established in section 3 plays an important role.  \smallskip

Let us refer to section 4 for more details.
\subsection{Nonlinear stability} 

We denote by $\mathbb{A}_{\nu,n}$ the restriction of $\mathbb{A}_\nu$ on the subspace $\mathcal{P}_nL^2_\sigma(\Omega)$.
 By the Duhamel formula, the solution $u(t)$ of \reff{NSA} could be written into the following integral form with respect to each Fourier mode $n$:
\begin{align}
\mathcal{P}_n u(t)=e^{-t\mathbb{A}_{\nu,n}}\mathcal{P}_n a-\int_0^te^{-(t-s)\mathbb{A}_{\nu,n}}\mathcal{P}_n\mathbb{P}(u\cdot\nabla u)(s)ds.
\end{align}
We introduce the functional space of Gevrey type:
\begin{eqnarray*}
\begin{split}
Z_{\gamma,K,T}=&\Big\{f\in(C([0,T];L^2_\sigma(\Omega)):\|f\|_{Z_{\gamma,K,T}}:=\sup_{0<t\leq T}\big(\|f(t)\|_{X_{q,\gamma,K(t)}}\\
&\quad+\nu^{\f14}\|f(t)\|_{Y_{q,\gamma,K(t)}}+(\nu t)^{\f12}\|\nabla f(t)\|_{X_{q,\gamma,K(t)}}\big)<+\infty\Big\},
\end{split}
\end{eqnarray*}
where $K(t)=K-2\delta^{-1}t$ and 
\begin{align*}
Y_{d,\gamma,K(t)}=\Big\{f\in L^2_\sigma(\Omega): \|f\|_{Y_{d,\gamma,K(t)}}=\sup_{n\in\mathbb{Z}}(1+|n|^d)e^{K(t)\an^{\gamma}}\|\mathcal{P}_n f\|_{L^2_xL^\infty_y(\Omega)}<+\infty\Big\}.
\end{align*}
Based on the semigroup estimates of $e^{-(t-s)\mathbb{A}_{\nu,n}}$ and nonlinear estimates, we can show that 
\begin{align*}
\|u\|_{Z_{\gamma,K,T}}\leq C\Big(\|a\|_{X^{(1)}_{d,\gamma,K}}+C\nu^{-\frac{1}{2}-\beta}\|u\|^2_{Z_{\gamma,K,T}}\Big).
\end{align*}
This implies our result.

Let us refer to section 5 for more details.

\section{Resolvent Estimates  in Middle Frequency}

For low frequency $|n|\le O(1)$ and high frequency $|n|\ge O(\nu^{-\f34})$,  the semigroup estimates of $e^{-t\mathbb{L}_{\nu,n}}$ can be proved  by using simple energy method. For the middle range of frequency $\mathcal{O}(1)\leq |n|\leq \mathcal{O}(\nu^{-3/4})$, we need to derive the semigroup estimates via the resolvent estimates of $\mathbb{L}_{\nu,n}$.

\subsection{Key resolvent estimates}
We only consider the case of  middle frequency:
\[
\delta^{-1}_0\leq |n|\leq \delta_0^{-1}\nu^{-3/4},\quad\text{where}\quad \delta_0=\frac{1}{2(1+\|V\|)}.
\]
Recall that $\lambda= \dfrac{\mathrm{i}\mu}{|n|\sqrt{\nu}}$ and $\alpha=\nu^{1/2}|n|$. Therefore,  $\delta^{-1}_0\nu^{1/2}\leq\alpha\leq\delta^{-1}_0\nu^{-1/4}$. Moreover, we consider the case
\begin{equation}\label{Remu}
\mathbf{Re}\mu=\alpha\mathbf{Im}\lambda\geq \frac{\nu^{\f12} |n|^\gamma}{\delta}
\end{equation}
for some $\gamma\in[0,1]$ and for sufficiently small but fixed positive  number $\delta$. We denote  $\lambda_r=\mathbf{Re}\lambda$, $\lambda_i=\mathbf{Im}\lambda$. We introduce the weight function $\rho(Y)$ defined by 
\begin{align}\label{rho-lambda}
     \rho(Y)=\left\{\begin{aligned}
     &\big(|n|^{\gamma-\f23}/\delta\big)^{\f32} Y,\qquad \text{if}\,\,\, 0\leq Y\leq \big(|n|^{\gamma-\f23}/\delta\big)^{-\f32},\\
     &1,\qquad\qquad\quad \text{if}\,\,\,Y\geq \big(|n|^{\gamma-\f23}/\delta\big)^{-\f32}.
     \end{aligned}\right.
  \end{align}

Our key resolvent estimates are stated as follows.

\begin{theorem}\label{main-resolvent}
Assume that $(SC)$ condition holds and $\delta_0^{-1}\leq |n|\leq \delta_0^{-1}\nu^{-\f34}$, and \reff{Remu} holds for some $\delta\in(0,\delta_*]$. Then there exist $\delta_1,\delta_2,\delta_*\in(0,1)$ satisfying $\delta_1,\delta_2\leq \delta_0$ and $\delta_*\leq \min\{\delta_1,\delta_2\}$ such that the following statements hold true. 
\begin{enumerate}
\item Let $n\in\mathbb{Z}$ and $\gamma\in[0,1]$. Then there exists $\theta\in(\frac{\pi}{2},\pi)$ such that the set
\begin{align}\label{S-mun}
S_{\nu,n}(\theta)=\big\{\mu\in\mathbb{C}||\mathbf{Im}\mu|\geq (\tan\theta)\mathbf{Re}\mu+\delta_1^{-1}(\nu^{\f12}|n|+|\tan\theta||n|^\gamma\nu^{\f12}),\, |\mu|\geq \delta^{-1}_1\nu^{\f12}|n|\big\}
\end{align}
is in the resolvent set of $-\mathbb{L}_{\nu,n}$ and
\begin{align}\label{mularge}
&\|(\mu+\mathbb{L}_{\nu,n})^{-1}f\|_{L^2(\Omega_{\nu})}\leq\frac{C}{|\mu|}\|f\|_{L^2(\Omega_\nu)},\\
\label{mularge-nabla}
&\|\nabla(\mu+\mathbb{L}_{\nu,n})^{-1}f\|_{L^2(\Omega_{\nu})}\leq\frac{C}{\nu^{\f14}|\mu|^{\f12}}\|f\|_{L^2(\Omega_\nu)}
\end{align}
for all $\mu\in S_{\nu,n}(\theta)$ and $f\in\mathcal{P}_{\nu,n}L^2_\sigma(\Omega_\nu)$.
\item If $|n|\geq \delta_0^{-1}$ and  $\mathbf{Re}\mu+n^2\nu^{\f32}\geq\delta^{-1}_2$, then $\mu$ belongs to the resolvent set of $-\mathbb{L}_{\nu,n}$ and the following estimates hold: for all $f\in\mathcal{P}_{\nu,n}L^2_\sigma(\Omega_\nu)$
\begin{align}\label{Immularge}
&\|(\mu+\mathbb{L}_{\nu,n})^{-1}f\|_{L^2(\Omega_{\nu})}\leq\frac{C}{\mathbf{Re}\mu}\|f\|_{L^2(\Omega_\nu)},\\
\label{Immularge-na}
&\|\nabla(\mu+\mathbb{L}_{\nu,n})^{-1}f\|_{L^2(\Omega_{\nu})}\leq\frac{C}{\nu^{\f14}(\mathbf{Re}\mu)^{\f12}}\|f\|_{L^2(\Omega_\nu)}.
\end{align}
\item Let $\gamma\in[\f23,1]$. Then the set
\begin{align}
O_{\nu,n}:=\Big\{\mu\in\mathbb{C}||\mu|\leq \delta^{-1}_1|n|\nu^{\f12},\quad\mathbf{Re}\mu\geq\frac{|n|^\gamma\nu^{\f12}}{\delta}\Big\}
\end{align}
is included in the resolvent set of $-\mathbb{L}_{\nu,n}$. Moreover, if $\mu\in O_{\nu,n}$ satisfies $\mathbf{Re}\mu=\frac{|n|^\gamma\nu^{\f12}}{\delta}$ and $\mathbf{Re}\mu+n^2\nu^{\f32}\leq\delta_2^{-1}$, then
\begin{align}\label{musmall}
&\|(\mu+\mathbb{L}_{\nu,n})^{-1}f\|_{L^2(\Omega_{\nu})}\leq\frac{Cn^{1-\gamma}}{\mathbf{Re}\mu}\|f\|_{L^2(\Omega_\nu)},\\
&\|\nabla(\mu+\mathbb{L}_{\nu,n})^{-1}f\|_{L^2(\Omega_{\nu})}\leq\frac{C|n|^{\f12+\f14(1-\gamma)})}{\mathbf{Re}\mu}\|f\|_{L^2(\Omega_\nu)},\\
\label{musmall-wegihted}
&\|\rho^{\f12}(\mathrm{curl}(\mu+\mathbb{L}_{\nu,n})^{-1}f)\|_{L^2(\Omega_{\nu})}\leq \frac{C}{\nu^{\f14}(\mathbf{Re}\mu)^{\f12}}\|f\|_{L^2(\Omega_{\nu})}.
\end{align}

\end{enumerate}
\end{theorem}

\begin{proof}
We point out that we can deduce \reff{Immularge}-\reff{musmall-wegihted} directly from Proposition \ref{prop-Immu-large} and \ref{Pro:resdrsmall} respectively. Hence, we only prove the first statement of Proposition \ref{main-resolvent}. Let $\delta_1\leq \delta_0$ be a small constant in Proposition \ref{prop-lambda-large}.
Since $\mu=-\mathrm{i}\sqrt{\nu}|n|\lambda$, we notice that $\lambda$ satisfies the condition of Proposition \ref{prop-lambda-large}, if $|\mu|\geq \delta^{-1}_1\alpha$ and $\mathbf{Re}\mu=\alpha\mathbf{Im}\lambda\geq \delta_1^{-1}|n|^\gamma\nu^{\f12}$. We obtain that such $\mu$ belongs to the resolvent set of $-\mathbb{L}_{\nu,n}$ in $\mathcal{P}_{\nu,n}L^2_\sigma(\Omega_\nu)$. Moreover, we have for such $\mu$, 
\begin{align*}
\|(\mu+\mathbb{L}_{\nu,n})^{-1}f\|_{L^2(\Omega_\nu)}\leq \frac{C}{|\mu|}\|f\|_{L^2(\Omega_{\nu})},
\end{align*}
which implies that the ball $B_{r_\mu}(\mu):=\{\eta\in\mathbb{C}||\eta-\mu|\leq r_\mu\}$ with $r_\mu=\frac{|\mu|}{2C}$ belongs to the resolvent set of $-\mathbb{L}_{\nu,n}$ and for any $\eta\in B_{r_\mu}(\mu)$,
\begin{align*}
\|(\eta+\mathbb{L}_{\nu,n})^{-1}f\|_{L^2(\Omega_\nu)}\leq\frac{C}{|\mu|}\|f\|_{L^2(\Omega_\nu)}.
\end{align*}
Hence, by taking $\theta=\frac{\pi}{2}+\theta_0$ with $\theta_0=\frac{1}{2C}$, we obtain
\begin{align*}
S_{\nu,n}(\theta)\subset\cup_{\mu\in E_{\nu,n}} B_{r_\mu}(\mu)\subset\rho(-\mathbb{L}_{\nu,n}),
\end{align*}
where
\begin{align*}
E_{\nu,n}=\big\{\mu\in\mathbb{C}|\mathbf{Re}\mu\geq \delta_1^{-1}|n|^\gamma\nu^{\f12}, |\mu|\geq\delta_1^{-1}\alpha\big\}.
\end{align*}
Hence, we complete the proof of \reff{S-mun} and \reff{mularge}. Similarly, we can obtain \reff{mularge-nabla}.
\end{proof}

In the sequel, we always assume $\delta_0^{-1}\leq |n|\leq \delta_{0}^{-1}\nu^{-\f34}$, that is $\delta_0^{-1}\nu^{\f12}\leq |\alpha|\leq \delta_{0}^{-1}\nu^{-\f14}$, and we assume that $n>0$ for convenience.

\subsection{Resolvent estimates when $\lambda$ is far away from origin}
In this part, we deal with the case when $|\lambda|$ is large or $\mathbf{Im}\lambda$ is large. We first notice that \reff{eq:sc-resphi1} can be written as
\begin{align}\label{eq-modify}\left\{\begin{aligned}
&-\sqrt{\nu}(\partial_Y^2-\alpha^2)\partial_Y^2 \phi + \mathrm{i}\alpha\big((V-\lambda_\nu)(\partial_Y^2-\alpha^2)\phi- (\partial_Y^2V)\phi\big)
=\mathrm{i}\alpha f_{2}-\partial_Yf_{1},\\
&\phi|_{Y=0}=\partial_Y\phi|_{Y=0}=0,
\end{aligned}\right.
\end{align}
where $\lambda_\nu=\lambda+i\sqrt{\nu}\alpha$. \smallskip

The following proposition shows the resolvent estimates when $|\lambda|$ is large.

\begin{proposition}\label{prop-lambda-large}
There exists $\delta_1\in(0,\delta_0]$ such that the following statements hold. Let $|\lambda|\geq \delta_1^{-1}$ and $n\in\mathbb{N}$. Suppose that \reff{Remu} holds for some $\gamma\in[0,1]$ and $\delta\in(0,\delta_1]$. Then for any $f=(f_1,f_2)\in L^2(\mathbb{R}_+)^2$, the weak solution $\phi\in H^2_0(\mathbb{R}_+)$ to \reff{eq:sc-resphi1} satisfies
\begin{align}
&\|(\partial_Y\phi,\alpha\phi)\|_{L^2}\leq \frac{C}{|\alpha\lambda|}\|f\|_{L^2},\label{lambda-large-L2}\\
&\|(\partial_Y^2-\alpha^2)\phi\|_{L^2}\leq \frac{C}{\nu^{1/4}|\alpha\lambda|^{1/2}}\|f\|_{L^2},\label{lambda-large-Linfinity}
\end{align}
where $C$ is a constant only depending on $\|V\|$.
\end{proposition}
\begin{proof}
Let $f=(f_1,f_2)\in L^2(\mathbb{R}_+)^2$ and $\phi$ be the unique weak solution to \reff{eq:sc-resphi1}. Then by the previous argument, $\phi$ also satisfies \reff{eq-modify}. Assume that $|\lambda|\geq\delta_1^{-1}$ with $\delta_1:=(32(1+\|U\|))^{-1}<\delta_0$. Since $|\mu|=\alpha|\lambda|$ and $\mathbf{Im}\lambda_\nu=\mathbf{Im}\lambda+\nu n\geq \mathbf{Im}\lambda>0$, we have $|\lambda_\nu|\geq|\lambda|\ge\delta_1^{-1}$, which implies
\begin{align}
\frac{|\lambda|}{2}\leq |V-\lambda_\nu|\leq 2|\lambda|.\label{eq:V-lam-est}
\end{align}
Moreover, we have
\ben
n\mathbf{Im}\lambda_\nu\ge n\mathbf{Im}\lambda\ge \delta^{-1}n^\gamma\gg 1.\label{eq:imlam-est}
\een
Multiplying both sides of the first equation of \reff{eq-modify} by $(V-\lambda_\nu)^{-1}\bar{\phi}$, then integrating by parts, we obtain
\begin{eqnarray}\label{prop-proof-energy}
\begin{split}
\|(\partial_Y\phi,\alpha\phi)\|_{L^2}^2\leq& \mathbf{Re}\int_0^{+\infty}\frac{\mathrm{i}\alpha f_2-\partial_Y f_1}{-i\alpha(V-\lambda_\nu)}\bar{\phi}\mathrm{d}Y+\int_0^{+\infty}\frac{|\partial_Y^2 V}{|V-\lambda_\nu|}|\phi|^2\mathrm{d}Y\\
&+\frac{6}{n\mathbf{Im}\lambda_\nu}\int_0^{+\infty}\frac{|\partial_YV|^2}{|V-\lambda_\nu|^2}|\partial_Y\phi|^2+\frac{|\partial_Y^2V|^2}{|V-\lambda_\nu|^2}|\phi|^2 \mathrm{d}Y\\
&+\frac{6}{n\mathbf{Im}\lambda_\nu}\int_0^{+\infty}\Big(\frac{|\partial_Y V|^4}{|V-\lambda_\nu|^4}+\alpha^2\frac{|\partial_Y V|^2}{|V-\lambda_\nu|^2}\Big)|\phi|^2\mathrm{d}Y.
\end{split}
\end{eqnarray}
We first notice that
\begin{align*}
\Big|\mathbf{Re}\int_0^{+\infty}\frac{\mathrm{i}\alpha f_2-\partial_Y f_1}{-i\alpha(V-\lambda_\nu)}\bar{\phi}\mathrm{d}Y\Big|\leq&\f14\|(\partial_Y\phi,\alpha\phi)\|_{L^2}^2+C\Big\|\frac{f}{\alpha(V-\lambda_\nu)}\Big\|^2_{L^2}\\
&+C\|V\|\Big\|\frac{f}{\alpha(V-\lambda_\nu)^2}\Big\|^2_{L^2}.\nonumber\\
\le& \f14\|(\partial_Y\phi,\alpha\phi)\|_{L^2}^2+\f C {\lambda^2\alpha^2}\|f\|_{L^2}^2,
\end{align*}
where we used \eqref{eq:V-lam-est} in the last step. By \eqref{eq:imlam-est}, we have
\begin{align*}
\frac{1}{n\mathbf{Im}\lambda_\nu}\int_0^\infty\frac{|\partial_Y V|^4}{|V-\lambda_\nu|^4}|\phi|^2\mathrm{d}Y\leq& C\frac{\delta_1^4\|V\|^2}{n\mathbf{Im}\lambda_\nu}\int_0\|V\|^2(1+Y)^{-2}|\phi|^2\mathrm{d}Y\\
\le&\f C{n\mathbf{Im}\lambda_\nu}\big\|\f \phi Y\big\|_{L^2}^2\le \f 1 {32}\|\pa_Y\phi\|_{L^2}^2.
\end{align*}
The estimate of the other terms on the right hand side of \reff{prop-proof-energy} is similar. We finally have
\begin{align}
\|(\partial_Y\phi,\alpha\phi)\|_{L^2}^2\leq \frac{C}{\alpha^2|\lambda|^2}\|f\|_{L^2}^2.\nonumber
\end{align}
This shows \reff{lambda-large-L2}.

Now we turn to prove \reff{lambda-large-Linfinity}. We multiply both sides of \reff{eq:sc-resphi1} by $\bar{\phi}$ and integrate over $(0,+\infty)$. Then we have
\begin{align*}
-\sqrt{\nu}\|(\partial_Y^2-\alpha^2)\phi\|_{L^2}^2+\langle\mathrm{i}\alpha(V-\lambda)(\partial_Y^2-\alpha^2)\phi,\phi\rangle_{L^2}-\langle\mathrm{i}\alpha(\partial_Y^2 V)\phi,\phi\rangle_{L^2}=\langle(\mathrm{i}\alpha f_2-\partial_Y f_1),\phi\rangle_{L^2},
\end{align*}
which implies
\begin{align}\label{proof-mularge-IM}
&\sqrt{\nu}\|(\partial_Y^2-\alpha^2)\phi\|_{L^2}^2+\alpha\lambda_i\|(\partial_Y\phi,\alpha\phi)\|_{L^2}^2\\
&\quad=\alpha\mathbf{Im}\langle \pa_YV\partial_Y\phi,\phi\rangle_{L^2}-\mathbf{Re}\langle\mathrm{i}\alpha(\partial_Y^2 V)\phi,\phi\rangle_{L^2}-\mathbf{Re}\langle(\mathrm{i}\alpha f_2-\partial_Y f_1),\phi\rangle_{L^2}.
\nonumber
\end{align}
On the other hand, we have
\begin{align}\label{proof-mularge-V}
&\big|\langle\mathrm{i}\alpha(\partial_Y^2 V)\phi,\phi\rangle_{L^2}\big|\leq C\|\partial_Y\phi\|_{L^2}\|\alpha\phi/Y\|_{L^2}\le C\al\|\pa_Y\phi\|_{L^2}^2
\leq  \frac{C}{\alpha|\lambda|^2}\|f\|_{L^2}^2,\\
&\big|\alpha\langle \pa_YV\pa_Y\phi,\phi\rangle_{L^2}\big|\leq C\|\partial_Y\phi\|_{L^2}\|\alpha/Y\phi\|_{L^2}\le C\al\|\pa_Y\phi\|_{L^2}^2\leq  \frac{C}{\alpha|\lambda|^2}\|f\|_{L^2}^2,
\end{align}
and
\begin{align}\label{proof-mularge-force}
|\langle(\mathrm{i}\alpha f_2-\partial_Y f_1),\phi\rangle_{L^2}|\leq  C\|(\partial_Y\phi,\alpha\phi)\|_{L^2}\|f\|_{L^2}\leq\frac{C}{\alpha|\lambda|}\|f\|_{L^2}^{2}.
\end{align}
Hence, by collecting \reff{proof-mularge-IM}-\reff{proof-mularge-force}, we obtain \reff{lambda-large-Linfinity}.
\end{proof}

\begin{proposition}\label{prop-Immu-large}
There exists $\delta_2\in(0,1)$ such that if $\alpha\lambda_i+\sqrt{\nu}\alpha^2\geq\delta_2^{-1}$, then for any $f=(f_1,f_2)\in L^2(\mathbb{R}_+)^2$, there exists a unique weak solution $\phi\in H^2_0(\mathbb{R}_+)$ to \reff{eq:sc-resphi1}  satisfying 
\begin{align}
&\|\partial_Y\phi\|_{L^2}+\alpha\|\phi\|_{L^2}\leq\frac{C}{\alpha\lambda_i+\sqrt{\nu}\alpha^2}\|f\|_{L^2}\label{ieq-Immu-large1}\\
&\|(\partial_Y^2-\alpha^2)\phi\|_{L^2}\leq \frac{C}{\nu^{1/4}(\alpha\lambda_i+\sqrt{\nu}\alpha^2)^{1/2}}\|f\|_{L^2}.\label{ieq-Immu-large2}
\end{align}
\end{proposition}

\begin{proof}
Let $\delta_2\in(0,\delta_0]$ be a small constant, which is determined later.
From the assumption of the proposition and the definition of $\lambda_\nu$, we have $\alpha\mathbf{Im}\lambda_\nu\geq\delta_2^{-1}$. By taking $L^2$-inner product on both sides of \reff{eq-modify} with $\phi$, we obtain
\begin{eqnarray}\nonumber
\begin{split}
-\sqrt{\nu}(\|\partial_Y^2\phi\|^2_{L^2}&+\alpha^2\|\partial_Y\phi\|^2_{L^2})+\langle\mathrm{i}\alpha(V-\lambda_\nu)(\partial_Y^2-\alpha^2)\phi,\phi\rangle_{L^2}-\langle\mathrm{i}\alpha(\partial_Y^2V)\phi,\phi\rangle_{L^2}\\
&=\langle(\mathrm{i}\alpha f_2-\partial_Y f_1),\phi\rangle_{L^2}.
\end{split}
\end{eqnarray}
Then by taking the real part of the above equality, we get
\begin{eqnarray}\label{proof-Immu1}
\begin{split}
\sqrt{\nu}\big(\|\partial_Y^2\phi\|^2_{L^2}&+\alpha^2\|\partial_Y\phi\|^2_{L^2}\big)+\alpha\mathbf{Im}\lambda_\nu\|(\partial_Y\phi,\alpha\phi)\|^2_{L^2}\\
&=-\alpha\mathbf{Im}\langle(V-\lambda_r)(\partial_Y^2-\alpha^2)\phi,\phi\rangle_{L^2}-\mathbf{Re}\langle(\mathrm{i}\alpha f_2-\partial_Y f_1),\phi\rangle_{L^2}\\
&=-\alpha\mathbf{Im}\langle V\partial_Y^2\phi,\phi\rangle_{L^2}-\mathbf{Re}\langle(\mathrm{i}\alpha f_2-\partial_Y f_1),\phi\rangle_{L^2}.
\end{split}
\end{eqnarray}
We also notice that
\begin{eqnarray}\label{proof-Immu2}
\begin{split}
\mathbf{Im}\langle V\partial_Y^2\phi,\phi\rangle_{L^2}\leq&\|V\|\|\partial_Y\phi\|_{L^2}\|\phi\|_{L^2}\\
\leq &\frac{\mathbf{Im}\lambda_\nu}{2}\|\partial_Y\phi\|_{L^2}^2+\frac{\|V\|^2}{2\mathbf{Im}\lambda_\nu}\|\phi\|^2_{L^2},
\end{split}
\end{eqnarray}
and
\begin{align}\label{proof-Immu3}
\big|\langle(\mathrm{i}\alpha f_2-\partial_Y f_1),\phi\rangle_{L^2}\big|\leq \|f\|_{L^2}\|(\partial_y\phi,\alpha\phi)\|_{L^2}.
\end{align}
Hence, after taking $\delta_2\leq \frac{1}{4(1+\|V\|)}$ and collecting \reff{proof-Immu1} and \reff{proof-Immu2} and \reff{proof-Immu3}, we obtain
\begin{align}
\sqrt{\nu}\big(\|\partial_Y^2\phi\|^2_{L^2}&+\alpha^2\|\partial_Y\phi\|^2_{L^2}\big)+\alpha\mathbf{Im}\lambda_\nu\|(\partial_Y\phi,\alpha\phi)\|^2_{L^2}\leq \frac{C}{\alpha\mathbf{Im}\lambda_\nu}\|f\|^2_{L^2},
\end{align}
which gives \reff{ieq-Immu-large1} and \reff{ieq-Immu-large2}.
\end{proof}

\subsection{Resolvent estimates when $|\lambda|\leq \delta^{-1}_1$}
The purpose of this part is to give the resolvent estimates when $|\lambda|\leq \delta^{-1}_1$. However, the boundary condition in \reff{eq:sc-resphi1} brings a lot of troubles to obtain an appropriate bound.  Our main idea to overcome the difficulty generated by the boundary is the following:

\begin{enumerate}
\item We first obtain the resolvent estimates under the Navier-slip boundary condition, which allows us to use some special structures of the first equation of \reff{eq:sc-resphi1} by using integration by parts argument.

\item We show the bounds of the boundary corrector. Such corrector is built around the Airy function and perfectly matches the boundary layer.

\item By combining the controls of the above two, we obtain the resolvent estimates for non-slip boundary condition.
\end{enumerate}

Throughout this subsection, we assume that $\gamma \in [\f 23,1], |\lambda|\le \delta_1^{-1}$ and (SC) condition holds.

\subsubsection{Resolvent estimates for Navier-slip boundary condition}
In this part, we replace the non-slip boundary condition of \reff{eq:sc-resphi1} by Navier-slip boundary condition. That is, we consider the following system
\begin{align}\label{eq:reswNa1}\left\{\begin{aligned}
&-\sqrt{\nu}(\partial_Y^2-\alpha^2)w + \mathrm{i}\alpha\big((V-\lambda)w-(\partial_Y^2V)\phi\big)=F,\\
&(\partial_Y^2-\alpha^2)\phi=w,\ w|_{Y=0}=\phi|_{Y=0}=0.
\end{aligned}\right.
\end{align}
where $F= -\partial_YF_1+\mathrm{i}\alpha F_2$. Since the source term $F$ actually belongs to $H^{-1}(\mathbb{R}_+)$ due to $F_1,F_2\in L^2(\mathbb{R}_+)$, we decompose $w=w_1+w_2$ with $w_1$ and $w_2$ satisfying
\begin{align}\label{eq:B1}\left\{
\begin{aligned}
&-\sqrt{\nu}(\partial_Y^2-\alpha^2)w_1+\mathrm{i}\alpha \big(\partial_Y\big((V-\lambda)\phi'_1\big)-(V-\lambda)\alpha^2\phi_1-V''\phi_1\big) =F,\\
&(\partial_Y^2-\alpha^2)\phi_1=w_1,\quad w_1|_{Y=0}=\phi_1|_{Y=0}=0,
\end{aligned}\right.\end{align}
and
\begin{align}\label{eq:B3}\left\{
\begin{aligned}
&-\sqrt{\nu}(\partial_Y^2-\alpha^2)w_2+\mathrm{i}\alpha\big((V-\lambda)w_2 -V''\phi_2\big) =V'h,\\
&(\partial_Y^2-\alpha^2)\phi_2=w_2,\quad w_2|_{Y=0}=\phi_2|_{Y=0}=0,
\end{aligned}\right.\end{align}
with $h=\mathrm{i}\alpha\partial_Y\phi_1$.

\begin{proposition}\label{pro:B3resH-1}
There exists $\delta_*\in(0,\delta_1]$ such that if $\lambda$ satisfies \reff{Remu} for some $\delta\in(0,\delta_*]$, then the unique solution to \eqref{eq:reswNa1} satisfies
  \begin{align*}
    \nu^{\f14}\alpha^\f12\lambda_i^{-\f12}\|w\|_{L^2} +\alpha\lambda_i\|(\partial_Y\phi,\alpha\phi)\|_{L^2}\leq C\lambda_i^{-1}\|(F_1,F_2)\|_{L^2},
\end{align*}
where the constant $C$ only depends on $\|V\|$.
\end{proposition}
\begin{proof}
 Let $w$ be the solution to \reff{eq:reswNa1} and $\phi$ be the corresponding stream function. As we mentioned on above, we decompose $w$ as $w=w_1+w_2$.
 

We first give the estimates for $(w_1,\phi_1)$. By taking inner product with $-\bar{\phi}_1$ to \eqref{eq:B1}, we have
  \begin{align*}
     & \sqrt{\nu}\|w_1\|_{L^2}^2+ \alpha\lambda_i\|(\partial_Y\phi_1,\alpha\phi_1)\|_{L^2}^2 +\mathrm{i}\alpha\int_{0}^{+\infty}\bigg((V-\lambda_r)(|\phi_1'|^2+|\alpha\phi_1|^2) +V''|\phi_1|^2\bigg)\mathrm{d}Y\\
     &=-\int_{0}^{+\infty}(-\partial_YF_1+\mathrm{i}\alpha F_2)\bar{\phi}_1\mathrm{d}Y.
  \end{align*}
  Then considering the real part, we get
  \begin{align*}
     &\nu^{\f12}\|w_1\|_{L^2}^2+ \alpha\lambda_i\|(\partial_Y\phi_1,\alpha\phi_1)\|_{L^2}^2\leq \|(F_1,F_2)\|_{L^2}\|(\partial_Y\phi_1,\alpha\phi_1)\|_{L^2},
    \end{align*}
which implies
  \begin{align*}
     & \nu^{\f14}(\alpha\lambda_i)^{\f12}\|w_1\|_{L^2}+\alpha\lambda_i\|( \partial_Y\phi_1,\alpha\phi_1)\|_{L^2}\leq C\|(F_1,F_2)\|_{L^2}.
  \end{align*}

Now we turn to deal with $(w_2,\phi_2)$. Noticing that $(w_2,\phi_2)$ satisfies the system \eqref{eq:B3}, then by Lemma \ref{lem:resB3}, we get
\begin{align*}
   &\alpha\lambda_i\|w_2\|_{L^2}+ \alpha\lambda_i\|(\partial_Y\phi_2,\alpha\phi_2)\|_{L^2}\\
  &\leq \alpha\lambda_i\|V'\|_{L^\infty}\left\|w_2/V'\right\|_{L^2}+ \alpha\lambda_i\|(\partial_Y\phi_2,\alpha\phi_2)\|_{L^2}\\
   & \leq C\alpha\lambda_i\big(\left\|w_2/V'\right\|_{L^2}+ \|(\partial_Y\phi_2,\alpha\phi_2)\|_{L^2}\big)\leq C\alpha\|\partial_Y \phi_1\|_{L^2}.
  \end{align*}
Then we obtain
\begin{align*}
   \alpha\lambda_i\|(\partial_Y\phi,\alpha\phi)\|_{L^2}\leq &\alpha\lambda_i\big(\|(\partial_Y\phi_1,\alpha\phi_1)\|_{L^2}+ \|(\partial_Y\phi_2,\alpha\phi_2)\|_{L^2}\big)\\
   \leq &C\alpha\|\partial_Y\phi_1\|_{L^2} +\alpha\lambda_i\|(\partial_Y\phi_1,\alpha\phi_1)\|_{L^2}\\
   \leq & C\alpha \|(\partial_Y\phi_1,\alpha\phi_1)\|_{L^2} \leq C\lambda_i^{-1}\|(F_1,F_2)\|_{L^2}.
\end{align*}
For $\|w\|_{L^2}$, we first notice that
\begin{align*}
   \|w\|_{L^2}\leq& \|w_1\|_{L^2}+\|w_2\|_{L^2}\leq C\lambda_i^{-1}\|\partial_Y\phi_1\|_{L^2}+ \|w_1\|_{L^2}
   \leq C\big(\alpha^{-1}\lambda_i^{-2}+ \nu^{-\f14}(\alpha\lambda_i)^{-\f12}\big)\|(F_1,F_2)\|_{L^2}.
\end{align*}
Thanks to \reff{Remu} and $\gamma\ge \f23$, we obtain
\begin{align*}
\alpha^{-1}\lambda_i^{-2}+\nu^{-\f14}(\alpha\lambda_i)^{-\f12}&=\nu^{-\f14}(\alpha\lambda_i)^{-\f12}(\alpha^{-\f12}\lambda_i^{-\f32}\nu^{\f14}+1)\\&\leq \nu^{-\f14}(\alpha\lambda_i)^{-\f12}(n^{1-\f32\gamma}+1)\leq C\nu^{-\f14}(\alpha\lambda_i)^{-\f12}.
\end{align*}
Finally, we have
\begin{align*}
\|w\|_{L^2}\leq C\nu^{-\f14}(\alpha\lambda_i)^{-\f12}\|(F_1,F_2)\|_{L^2}.
\end{align*}

This finishes the proof.
\end{proof}

The following lemma gives  the control of the solution to \reff{eq:B3}.

\begin{lemma}\label{lem:resB3}
Let $h\in L^2(\mathbb{R}_+)$. Then there exists $\delta_*\in(0,\delta_1]$ such that if $\lambda$ satisfies \reff{Remu} for some $\delta\in(0,\delta_*]$, then the unique solution to \eqref{eq:B3} satisfies
\begin{align*}
  & \alpha\lambda_i\|(\partial_Y\phi,\alpha\phi)\|_{L^2}\leq C\|h\|_{L^2},\\
&\nu^{\f14}(\alpha\lambda_i)^{\f12}\left\|(\partial_Yw,\alpha w)/V'\right\|_{L^2}+\alpha\lambda_i\left\|w/V'\right\|_{L^2}\leq C\|h\|_{L^2}.
\end{align*}
\end{lemma}
\begin{proof}
Let $\varsigma$ be an arbitrary fixed small  positive number.

\textbf{Step 1.  Estimate of the vorticity $w$ in a weighted norm}\smallskip

By taking $L^2$-inner product with $-w/(V''-\varsigma)$ on both sides of the first equation of \reff{eq:B3}, we obtain
\begin{align}\label{L2productw}
\mathbf{Re}\langle-\sqrt{\nu}(\partial_Y^2-\alpha^2)w+\mathrm{i}\alpha((V-\lambda)w-V''\phi),-\frac{w}{V''-\varsigma}\rangle\leq |\langle V'h,\frac{w}{V''-\varsigma}\rangle|.
\end{align}

For each terms on the left-hand side of \reff{L2productw}, we first have
\begin{align*}
     \big\langle (\partial_Y^2-\alpha^2)w,w/(V''-\varsigma)\big\rangle= &\left\|\dfrac{(\partial_Yw,\alpha w)}{|V''-\varsigma|^{\f12}}\right\|_{L^2}^2 +\left\langle\partial_Yw,\dfrac{wV'''}{(V''-\varsigma)^2}\right\rangle,
\end{align*}
which gives
\begin{eqnarray}\label{nuww}
\begin{split}
&\mathbf{Re}\langle \sqrt{\nu}(\partial_Y^2-\alpha^2)w,w/(V''-\varsigma)\big\rangle\\
&=\sqrt{\nu}\left\|\dfrac{(\partial_Yw,\alpha w)}{|V''-\varsigma|^{\f12}}\right\|_{L^2}^2 +\sqrt{\nu}\mathbf{Re}\left\langle\partial_Yw,\dfrac{wV'''}{(V''-\varsigma)^2}\right\rangle\\
&\geq\sqrt{\nu}\left\|\dfrac{(\partial_Yw,\alpha w)}{|V''-\varsigma|^{\f12}}\right\|_{L^2}^2-\sqrt{\nu}\left\| \dfrac{\partial_Yw}{|V''-\varsigma|^{\f12}}\right\|_{L^2}\left\| \dfrac{V'''}{V''-\varsigma}\right\|_{L^\infty}\left\| \dfrac{w}{|V''-\varsigma|^{\f12}}\right\|_{L^2}\\
&\geq \frac{\sqrt{\nu}}{2}\left\|\dfrac{(\partial_Yw,\alpha w)}{|V''-\varsigma|^{\f12}}\right\|_{L^2}^2-C\sqrt{\nu}\left\| \dfrac{w}{|V''-\varsigma|^{\f12}}\right\|_{L^2}^2,
\end{split}
\end{eqnarray}
where we used $|V'''/V''|+|V''/V'|\leq C$ in the last inequality.

We also notice that
\begin{align*}
     &\mathbf{Im} \big\langle (V-\lambda)w-V''\phi,w/(V''-\varsigma)\big\rangle\\
     &= \mathbf{Im}\bigg( \big\langle V-\lambda,|w|^2/(V''-\varsigma)\big\rangle+\|(\partial_Y\phi,\alpha \phi)\|_{L^2}^2-\varsigma\big\langle \phi,w/(V''-\varsigma)\big\rangle\bigg)\\
     &=\lambda_i\left\|\dfrac{w}{|V''-\varsigma|^{\f12}}\right\|_{L^2}^2-\varsigma\mathbf{Im} \big\langle \phi,w/(V''-\varsigma)\big\rangle,
  \end{align*}
from which, we deduce that
\begin{eqnarray}
\begin{split}
\mathbf{Re}\left\langle\mathrm{i}\alpha(V-\lambda)w-V''\phi),-\frac{w}{V''-\varsigma}\right\rangle&=\alpha\mathbf{Im} \big\langle (V-\lambda)w-V''\phi,w/(V''-\varsigma)\big\rangle\\
&=\alpha\lambda_i\left\|\dfrac{w}{|V''-\varsigma|^{\f12}}\right\|_{L^2}^2-\varsigma\alpha\mathbf{Im} \big\langle \phi,w/(V''-\varsigma)\big\rangle\\
&\geq\alpha\lambda_i\left\|\dfrac{w}{|V''-\varsigma|^{\f12}}\right\|_{L^2}^2-\varsigma^{\frac{1}{2}}\left\|\dfrac{w}{|V''-\varsigma|^{\f12}}\right\|_{L^2}\|\alpha\phi\|_{L^2}\\
&\geq \frac{\alpha\lambda_i}{2}\left\|\dfrac{w}{|V''-\varsigma|^{\f12}}\right\|_{L^2}^2-C\varsigma(\alpha\lambda_i)^{-1}\|\alpha\phi\|_{L^2}^2.
\end{split}
\end{eqnarray}
According to  the above inequality,  \reff{L2productw} and \reff{nuww}, we obtain
\begin{eqnarray*}
\begin{split}
&\sqrt{\nu}\left\|\frac{(\partial_Y w,\alpha w)}{|V''-\varsigma|^{1/2}}\right\|^{2}_{L^2}+\alpha\lambda_i\left\|\dfrac{w}{|V''-\varsigma|^{\f12}}\right\|_{L^2}^2\\
&\leq 2\|h\|_{L^2}\left\|\frac{V'}{|V''-\varsigma|^{\f12}}\right\|_{L^\infty}\left\|\dfrac{w}{|V''-\varsigma|^{\f12}}\right\|_{L^2}+C\sqrt{\nu}\left\|\dfrac{w}{|V''-\varsigma|^{\f12}}\right\|_{L^2}^2+C\varsigma(\alpha\lambda_i)^{-1}\|\alpha\phi\|^2_{L^2},
\end{split}
\end{eqnarray*}
which gives
\begin{align}\label{est:nablew1}
     & \sqrt{\nu} \left\|\dfrac{(\partial_Yw,\alpha w)}{|V''-\varsigma|^{\f12}}\right\|_{L^2}^2 +\alpha\lambda_i \left\|\dfrac{w}{|V''-\varsigma|^{\f12}}\right\|_{L^2}^2\leq C(\alpha\lambda_i)^{-1}\|h\|_{L^2}^2 +C\varsigma(\alpha\lambda_i)^{-1}\|\alpha\phi\|_{L^2}^2,
  \end{align}
along with the (SC) condition and the fact that
\begin{align*}
\alpha\lambda_i\geq \frac{\delta_0^\gamma}{\delta}\sqrt{\nu},~~\text{with $\delta$ small enough}.
\end{align*}

\textbf{Step 2. Estimates via the Rayleigh equation}\smallskip

Denote $\textsl{R}\phi :=(V-\lambda)(\partial_Y^2-\alpha^2)\phi-V''\phi$, then we can write the equation as
\begin{align*}
   &\mathrm{i}\alpha(\textsl{R}\phi) =\sqrt{\nu}(\partial_Y^2-\alpha^2)w+V'h.
\end{align*}
Applying Lemma \ref{lem:GMMray1} with $h_1=h/(\mathrm{i}\alpha),\ h_2=\sqrt{\nu}\partial_Yw/(\mathrm{i}\alpha),\ h_3=\sqrt{\nu} w$, we get
\begin{align*}
   \|(\partial_Y\phi,\alpha\phi)\|_{L^2}\leq &C\big(\lambda_i^{-1}\|h_1\|_{L^2}+\lambda_i^{-2}\|(h_2,h_3)\|_{L^2}\big)\\
   \leq &C(\alpha\lambda_i)^{-1}\|h\|_{L^2}+ C\nu^{\f12}\lambda_i^{-2}\alpha^{-1}\|(\partial_Yw,\alpha w)\|_{L^2}\\
   \leq &C(\alpha\lambda_i)^{-1}\|h\|_{L^2}+ C\nu^{\f12}\lambda_i^{-2}\alpha^{-1}\|V''-\varsigma\|_{L^\infty}^{\f12} \left\|\dfrac{(\partial_Yw,\alpha w)}{|V''-\varsigma|^{\f12}}\right\|_{L^2}\\
   \leq &C(\alpha\lambda_i)^{-1}\|h\|_{L^2}+ C\nu^{\f12}\lambda_i^{-2}\alpha^{-1}\left\|\dfrac{(\partial_Yw,\alpha w)}{|V''-\varsigma|^{\f12}}\right\|_{L^2}.
\end{align*}
Substituting \eqref{est:nablew1} into the above inequality, we deduce that
\begin{align*}
   & \|(\partial_Y\phi,\alpha\phi)\|_{L^2}\leq C(\alpha\lambda_i)^{-1}\|h\|_{L^2} +C\nu^{\f14}(\alpha\lambda_i)^{-\f32}\lambda_i^{-1}\|h\|_{L^2}+ C\varsigma^{\f12}\nu^{\f14}(\alpha\lambda_i)^{-\f32}\lambda_i^{-1}\|\alpha\phi\|_{L^2}.
\end{align*}
Letting $\varsigma\rightarrow 0^{+},$ we infer that
\begin{align}\label{est:phiB3pori}
   & \|(\partial_Y\phi,\alpha\phi)\|_{L^2}\leq C(\alpha\lambda_i)^{-1}\|h\|_{L^2} +C\nu^{\f14}(\alpha\lambda_i)^{-\f32}\lambda_i^{-1}\|h\|_{L^2}.
\end{align}
On the other hand, we notice that by  \reff{Remu}
\begin{align*}
\nu^{\f14}(\alpha\lambda_i)^{-\f32}\leq \alpha^{-1}n^{1-\f32\gamma} \delta^\f32\leq\alpha^{-1},
\end{align*}
provided that $\gamma\in[2/3,1]$.
Then \eqref{est:phiB3pori} becomes
\begin{align}\label{est:parphiB3}
   \|(\partial_Y\phi,\alpha\phi)\|_{L^2}\leq C(\alpha\lambda_i)^{-1}\|h\|_{L^2}.
\end{align}
Putting this into \eqref{est:nablew1}, we conclude that
\begin{align*}
   & \nu^{\f14}\left\|\dfrac{(\partial_Yw,\alpha w)}{|V''-\varsigma|^{\f12}}\right\|_{L^2}+(\alpha\lambda_i)^{\f12}\left\| \dfrac{w}{|V''-\varsigma|^{\f12}}\right\|_{L^2}\leq C(\alpha\lambda_i)^{-\f12}\big(1+\varsigma^{\f12}(\alpha\lambda_i)^{-1}\big)\|h\|_{L^2}.
\end{align*}
Applying $1/|V''-\varsigma|^{\f12}\geq CM^{\f12}/\big|V'+(M\varsigma)^{\f12}\big|$, and letting $\varsigma\rightarrow 0^{+}$, we deduce that
\begin{align*}
    &\nu^{\f14}(\alpha\lambda_i)^{\f12}\left\|(\partial_Yw,\alpha w)/V'\right\|_{L^2}+\alpha\lambda_i\left\|w/V'\right\|_{L^2}\leq C\|h\|_{L^2},
\end{align*}
which gives the second inequality.
\end{proof}

The following lemma used Rayleigh's trick for  strong concave shear flow.

\begin{lemma}\label{lem:GMMray}
Let $\phi\in H^1_0(\mathbb{R}_+)\cap H^2(\mathbb{R}_+)$ and we denote the Rayleigh operator as $\textsl{R}\phi:= (V-\lambda)(\partial_Y^2-\alpha^2)\phi-V''\phi$. Then we have
  \begin{align}
     & \mathbf{Re} \bigg(\dfrac{1-\lambda}{\mathrm{i}\lambda_i}\int_{0}^{+\infty}(\textsl{R}\phi) \dfrac{\bar{\phi}}{V-\lambda}\mathrm{d}Y\bigg) \geq \|(\partial_Y\phi,\alpha\phi)\|_{L^2}^2+M^{-1} \left\|\dfrac{(1-V)^{\f12}V'\phi}{V-\lambda}\right\|_{L^2}^2.
  \end{align}
 Moreover,  if $\lambda_r\geq 1$, then we have
  \begin{align}
     &-\mathbf{Re} \bigg(\int_{0}^{+\infty}(\textsl{R}\phi) \dfrac{\bar{\phi}}{V-\lambda}\mathrm{d}Y\bigg) \geq \|(\partial_Y\phi,\alpha\phi)\|_{L^2}^2+M^{-1} \left\|\dfrac{(1-V)^{\f12}V'\phi}{V-\lambda}\right\|_{L^2}^2.
  \end{align}
Here $M$ is the constant in third property of the (SC) condition.
\end{lemma}
\begin{proof}
  Taking inner product with $\dfrac{-\bar{\phi}}{V-\lambda}$, we get
  \begin{align*}
     &-\int_{0}^{+\infty}(\textsl{R}\phi) \dfrac{\bar{\phi}}{V-\lambda}\mathrm{d}Y= \|(\partial_Y\phi,\alpha\phi)\|_{L^2}^2+\int_{0}^{+\infty} \dfrac{(\partial_Y^2V)|\phi|^2}{V-\lambda}\mathrm{d}Y.
  \end{align*}
 Considering the real and imaginary part respectively, we obtain
  \begin{align}
    -\mathbf{Re}\left(\int_{0}^{+\infty} \dfrac{(\textsl{R}\phi) \bar{\phi}}{V-\lambda}\mathrm{d}Y\right) &=\|(\partial_Y\phi,\alpha\phi)\|_{L^2}^2+\int_{0}^{+\infty} \dfrac{(V-\lambda_r)(\partial_Y^2V)|\phi|^2}{|V-\lambda|^2}\mathrm{d}Y,\label{proi:real}\\
    \mathbf{Im}\left(\int_{0}^{+\infty} \dfrac{(\textsl{R}\phi) \bar{\phi}}{V-\lambda}\mathrm{d}Y\right)&=\lambda_i\int_{0}^{+\infty} \dfrac{(-\partial_Y^2V)|\phi|^2}{|V-\lambda|^2}\mathrm{d}Y.\label{proi:imega}
  \end{align}
 By (SC) condition: $-\partial_Y^2V\geq M^{-1}(\partial_YV)^2$, we get by \eqref{proi:imega} that 
  \begin{align*}
     &-\int_{0}^{+\infty}\dfrac{(V-\lambda_r)(\partial_Y^2V)|\phi|^2}{|V-\lambda|^2}\mathrm{d}Y\\
     &= -\int_{0}^{+\infty}\dfrac{(1-\lambda_r)(\partial_Y^2V)|\phi|^2}{|V-\lambda|^2}\mathrm{d}Y -\int_{0}^{+\infty}\dfrac{(V-1)(\partial_Y^2V)|\phi|^2}{|V-\lambda|^2}\mathrm{d}Y\\
     &= \dfrac{1-\lambda_r}{\lambda_i}\left(\int_{0}^{+\infty} \dfrac{\lambda_i(-\partial_Y^2V)|\phi|^2}{|V-\lambda|^2}\mathrm{d}Y\right) -\int_{0}^{+\infty}\dfrac{(V-1)(\partial_Y^2V)|\phi|^2}{|V-\lambda|^2}\mathrm{d}Y\\
     &\leq\dfrac{1-\lambda_r}{\lambda_i}\mathbf{Im}\left(\int_{0}^{+\infty} \dfrac{(\textsl{R}\phi) \bar{\phi}}{V-\lambda}\mathrm{d}Y\right) -\int_{0}^{+\infty}\dfrac{(1-V)(\partial_YV)^2|\phi|^2}{M|V-\lambda|^2}\mathrm{d}Y\\
     &\leq \dfrac{1-\lambda_r}{\lambda_i}\mathbf{Im}\left(\int_{0}^{+\infty} \dfrac{(\textsl{R}\phi) \bar{\phi}}{V-\lambda}\mathrm{d}Y\right) -M^{-1}\left\|\dfrac{(1-V)^{\f12}(\partial_YV)\phi}{V-\lambda}\right\|_{L^2}^2.
  \end{align*}
 Putting this into \eqref{proi:real}, we obtain
  \begin{align*}
    &\|(\partial_Y\phi,\alpha\phi)\|_{L^2}^2 =- \int_{0}^{+\infty}\dfrac{(V-\lambda_r)(\partial_Y^2V)|\phi|^2}{|V-\lambda|^2}\mathrm{d}Y -\mathbf{Re}\left(\int_{0}^{+\infty} \dfrac{(\textsl{R}\phi) \bar{\phi}}{V-\lambda}\mathrm{d}Y\right)\\
    &\leq \dfrac{1-\lambda_r}{\lambda_i}\mathbf{Im}\left(\int_{0}^{+\infty} \dfrac{(\textsl{R}\phi) \bar{\phi}}{V-\lambda}\mathrm{d}Y\right) -M^{-1}\left\|\dfrac{(1-V)^{\f12}(\partial_YV)\phi}{V-\lambda}\right\|_{L^2}^2 -\mathbf{Re}\left(\int_{0}^{+\infty} \dfrac{(\textsl{R}\phi) \bar{\phi}}{V-\lambda}\mathrm{d}Y\right)\\
    &=\mathbf{Re}\left(\dfrac{1-\lambda_r}{\mathrm{i}\lambda_i}\int_{0}^{+\infty} \dfrac{(\textsl{R}\phi) \bar{\phi}}{V-\lambda}\mathrm{d}Y\right) -\mathbf{Re}\left(\int_{0}^{+\infty} \dfrac{(\textsl{R}\phi) \bar{\phi}}{V-\lambda}\mathrm{d}Y\right) -M^{-1}\left\|\dfrac{(1-V)^{\f12}(\partial_YV)\phi}{V-\lambda}\right\|_{L^2}^2\\
    &=\mathbf{Re}\left(\dfrac{1-\lambda}{\mathrm{i}\lambda_i}\int_{0}^{+\infty} \dfrac{(\textsl{R}\phi) \bar{\phi}}{V-\lambda}\mathrm{d}Y\right) -M^{-1}\left\|\dfrac{(1-V)^{\f12}(\partial_YV)\phi}{V-\lambda}\right\|_{L^2}^2,
  \end{align*}
  This gives the first inequality.

If $\lambda_r\geq 1$, again by \eqref{proi:real} and (SC) condition: $-\partial_Y^2V\geq M^{-1}(\partial_YV)^2$, we have
  \begin{align*}
     -\mathbf{Re}\left(\int_{0}^{+\infty} \dfrac{(\textsl{R}\phi) \bar{\phi}}{V-\lambda}\mathrm{d}Y\right) &=\|(\partial_Y\phi,\alpha\phi)\|_{L^2}^2+\int_{0}^{+\infty} \dfrac{(V-\lambda_r)(\partial_Y^2V)|\phi|^2}{|V-\lambda|^2}\mathrm{d}Y\\
     &\geq \|(\partial_Y\phi,\alpha\phi)\|_{L^2}^2+\int_{0}^{+\infty} \dfrac{(1-V)(-\partial_Y^2V)|\phi|^2}{|V-\lambda|^2}\mathrm{d}Y\\
     &\geq \|(\partial_Y\phi,\alpha\phi)\|_{L^2}^2+M^{-1} \left\|\dfrac{(1-V)^{\f12}V'\phi}{V-\lambda}\right\|_{L^2}^2.
  \end{align*}
  This gives the second inequality.
\end{proof}

 The following lemma has been used in the proof of  Lemma \ref{lem:resB3}.
\begin{lemma}\label{lem:GMMray1}
   Let $\phi\in H^1_0(\mathbb{R}_+)$ solve $\textsl{R}\phi=\tilde{h}:=V'h_1+\partial_Yh_2+\mathrm{i}\alpha h_3$ with $h_i\in L^2(\mathbb{R}_+)$ for $i=1,2,3$. Then it holds that
  \begin{align*}
     & \|(\partial_Y\phi,\alpha\phi)\|_{L^2}\leq C\big(\lambda_i^{-1}\|h_1\|_{L^2} +\lambda_i^{-2}\|(h_2,h_3)\|_{L^2}\big).
  \end{align*}
\end{lemma}
\begin{proof}
Notice that
 \begin{align*}
     &\left|\int_{0}^{+\infty}\dfrac{\tilde{h}\bar{\phi}}{V-\lambda}\mathrm{d}Y \right|=\left|\int_{0}^{+\infty}\dfrac{\big( V'h_1+\partial_Yh_2+\mathrm{i}\alpha h_3\big)\bar{\phi}}{V-\lambda}\mathrm{d}Y\right|\\
    &\leq \left|\int_{0}^{+\infty} \bigg(V'h_1+\dfrac{V'h_2}{V-\lambda}\bigg)\dfrac{\bar{\phi}}{V-\lambda}\mathrm{d}Y\right| +\left| \int_{0}^{+\infty}\dfrac{h_2\partial_Y\bar{\phi}}{V-\lambda}\mathrm{d}Y\right| +\left| \int_{0}^{+\infty}\dfrac{ \alpha h_3\bar{\phi}}{V-\lambda}\mathrm{d}Y\right|\\
    &\leq \left\|\dfrac{\sqrt{-V''}\phi}{V-\lambda}\right\|_{L^2}\left( \left\|\dfrac{V'h_1}{\sqrt{-V''}}\right\|_{L^2} +\left\|\dfrac{V'h_2}{\sqrt{-V''}(V-\lambda)}\right\|_{L^2}\right) +\|\partial_Y\phi\|_{L^2}\left\|\dfrac{h_2}{V-\lambda}\right\|_{L^2} \\&\quad+\|\alpha\phi\|_{L^2}\left\|\dfrac{h_3}{V-\lambda}\right\|_{L^2}\\
    &\leq M^{\f12}\left\|\dfrac{\sqrt{-V''}\phi}{V-\lambda}\right\|_{L^2}\left( \|h_1\|_{L^2} +\left\|\dfrac{h_2}{(V-\lambda)}\right\|_{L^2}\right) +C\lambda_i^{-1}\|(\partial_Y\phi,\alpha\phi)\|_{L^2}\|(h_2,h_3)\|_{L^2}.
  \end{align*}
 Here we used the strong concave condition. And \eqref{proi:imega} gives
  \begin{align*}
     &\left|\int_{0}^{+\infty}\dfrac{\tilde{h}\bar{\phi}}{V-\lambda}\mathrm{d}Y \right|\geq \mathbf{Im}\left(\int_{0}^{+\infty}\dfrac{\tilde{h}\bar{\phi}}{V-\lambda}\mathrm{d}Y \right)= \lambda_i\left\|\dfrac{\sqrt{-V''}\phi}{V-\lambda}\right\|_{L^2}^2.
  \end{align*}
  Then we obtain
  \begin{align*}
   \left|\int_{0}^{+\infty}\dfrac{\tilde{h}\bar{\phi}}{V-\lambda}\mathrm{d}Y\right|&\leq C\lambda_i^{-1}M\bigg( \|h_1\|_{L^2} +\left\|\dfrac{h_2}{(V-\lambda)}\right\|_{L^2}\bigg)^2 +C\lambda_i^{-1}\|(\partial_Y\phi,\alpha\phi)\|_{L^2}\|(h_2,h_3)\|_{L^2}\\
   &\leq C\lambda_i^{-1}\|h_1\|_{L^2}^2+ C\lambda_i^{-3}\|h_2\|_{L^2}^2+C\lambda_i^{-1}\|(\partial_Y\phi,\alpha\phi)\|_{L^2} \|(h_2,h_3)\|_{L^2}.
  \end{align*}
  By lemma \ref{lem:GMMray} and $|\lambda|\leq \delta_1^{-1}$, we get
 \begin{align*}
    \|(\partial_Y\phi,\alpha\phi)\|_{L^2}^2\leq& C\lambda_i^{-1} \left|\int_{0}^{+\infty}\dfrac{\tilde{h}\bar{\phi}}{V-\lambda}\mathrm{d}Y\right|\\
    \leq &C\lambda_i^{-1}\big(\lambda_i^{-1}\|h_1\|_{L^2}^2+ \lambda_i^{-4}|\lambda|\|h_2\|_{L^2}^2+\lambda_i^{-1}\|(\partial_Y\phi,\alpha\phi)\|_{L^2} \|(h_2,h_3)\|_{L^2}\big),
 \end{align*}
 which gives
 \begin{align*}
    &\|(\partial_Y\phi,\alpha\phi)\|_{L^2}\leq C\big(\lambda_i^{-1}\|h_1\|_{L^2} +\lambda_i^{-2}\|(h_2,h_3)\|_{L^2}\big).
 \end{align*}
 This proves the lemma.
 \end{proof}

\subsubsection{ Boundary layer corrector}
Since we have obtained  the resolvent estimate under the Navier-slip boundary condition, the  next step is to re-correct the boundary condition from Navier-slip boundary condition to nonslip one. To this end,  we introduce the boundary layer corrector, which  is the solution to the following homogeneous system
\begin{align}\label{eq:Hombound-OSnon}
\left\{\begin{aligned}
&-\sqrt{\nu}(\partial_Y^2-\alpha^2)W_b+\mathrm{i}\alpha\big((V-\lambda)W_b- V''\Phi_b\big)=0,\\
&(\partial_Y^2-\alpha^2)\Phi_b=W_b,\\
&\Phi_b|_{Y=0}=0,\ \partial_Y\Phi_b|_{Y=0}=1.
\end{aligned}\right.
\end{align}

Instead of considering the above system directly, we first pay our attention to study the system as follows
\begin{align}\label{eq:Hombound-OS}
\left\{\begin{aligned}
&-\sqrt{\nu}(\partial_Y^2-\alpha^2)W+\mathrm{i}\alpha\big((V-\lambda)W- V''\Phi\big)=0,\\
&(\partial_Y^2-\alpha^2)\Phi=W,\\
&\Phi|_{Y=0}=0,\ \partial_Y^2\Phi|_{Y=0}=1.
\end{aligned}\right.
\end{align}

The reason we consider this system is that
\begin{itemize}
\item The solution $W$ to  \reff{eq:Hombound-OS} actually is a small perturbation around the Airy function, which is known very well by us. Moreover, the corresponding perturbation satisfies the Navier-slip condition, on which we can apply the result in the previous part.

\item We observe that $\partial_Y\Phi$ on the boundary $Y=0$ is also positive. Hence, to drive the estimates from \reff{eq:Hombound-OS} to \reff{eq:Hombound-OSnon}, we just need to normalized the value of $\partial_Y\Phi$ on the boundary.
\end{itemize}

For convenience, we introduce the following notation
\begin{align}\label{notation:A,A1}
   &A=|n|^{\f13}(1+|n|^{\f13}|\lambda_{\nu}|)^{\f12}.
\end{align}
We first present the following lemma, which will be used frequently.

\begin{lemma}\label{lem:L-index}
Let $\delta_0^{-1}\nu^{\f12}\leq \alpha\leq \delta_0^{-1}\nu^{-\f14}$ with $\alpha=\sqrt{\nu}n$, $\lambda_i\geq \dfrac{n^{\gamma-1}}{\delta}$ for some $\gamma\in[\f23,1]$. Then it holds that
  \begin{align*}
     & \max\big(\delta_0^{\f23}\alpha,\delta_0^{-\f13}\big) \leq |n|^{\f13}\quad \mathrm{and}\quad  |n|^{-\f13}\leq C(1+\alpha)^{-1}.
  \end{align*}
  Moreover, we have
  \begin{align*}
    &|n|^{\f13}\lambda_i\geq \delta_0^{-(\gamma-2/3)}\delta^{-1}\geq \delta^{-1}.
  \end{align*}
\end{lemma}

\begin{proof}
Due to $\alpha \leq \delta_0^{-1}\nu^{-\f14}$, $\alpha = |n|^{\f13}\big|\sqrt{\nu}\alpha^2\big|^{\f13}\leq |n|^\f13\delta_0^{-\f23}$. Due to $\delta_0^{-1}\nu^{\f12}\leq \alpha$,  $|n|^{\f13}\geq \delta_0^{-\f13}$. This deduces the first inequality. We also have
\begin{align*}
   &|n|^{\f13}\lambda_i\geq \dfrac{\alpha^{\f13}}{\nu^{\f16}}\dfrac{\nu^{(1-\gamma)/2} \alpha^{\gamma-1}}{\delta}= \nu^{(2/3-\gamma)/2}\alpha^{\gamma-2/3}\delta^{-1} = \big( \alpha/\sqrt{\nu}\big)^{\gamma-2/3}\delta^{-1}\geq \delta_0^{-(\gamma-2/3)}\delta^{-1},
\end{align*}
which gives the third inequality.
\end{proof}

Now we construct the solution $W$ to \reff{eq:Hombound-OS} via the Airy function.
We denote $d=-\lambda_\nu/V'(0)$, where $\lambda_\nu=\lambda+\mathrm{i}\sqrt{\nu}\alpha$. We introduce
  \begin{align*}
     & W_a (Y)= Ai\big(\mathrm{e}^{\mathrm{i}\frac{\pi}{6}}|nV'(0)|^{\f13}(Y+d)\big)/ Ai\big(\mathrm{e}^{\mathrm{i}\frac{\pi}{6}}|nV'(0)|^{\f13}d\big).
  \end{align*}
  here $Ai(y)$ is the Airy function defined in the appendix, which satisfies $Ai''(y)-yAi(y)=0$.
  Then we have
  \begin{align*}
     \partial_Y^2W_a&=\mathrm{e}^{\mathrm{i}\frac{\pi}{3}} |nV'(0)|^{\f23}(\partial_Y^2Ai)\big(\mathrm{e}^{\mathrm{i}\frac{\pi}{6}}|nV'(0)|^{\f13}(Y+d)\big)/Ai\big(\mathrm{e}^{\mathrm{i}\frac{\pi}{6}}|nV'(0)|^{\f13}d\big) \\
     &=\mathrm{i} |nV'(0)|^{\f23}(|nV'(0)|^{\f13}(Y+d))Ai\big(\mathrm{e}^{\mathrm{i}\frac{\pi}{6}}|nV'(0)|^{\f13}(Y+d)\big)/Ai\big(\mathrm{e}^{\mathrm{i}\frac{\pi}{6}}|nV'(0)|^{\f13}\tilde{d}\big)\\
     &=\mathrm{i}(\alpha/\sqrt{\nu}) \big(V'(0)(Y+d)\big)W_a,
  \end{align*}
  which gives
  \begin{align}\left\{\begin{aligned}
     &-\sqrt{\nu}(\partial_Y^2-\alpha^2)W_a+\mathrm{i}\alpha \big(V'(0)Y-\lambda\big)W_a=0,\\
     &W_a(0)=1.
     \end{aligned}\right.
  \end{align}
We denote the perturbation  $W_e=W-W_a$, which satisfies
  \begin{align}\label{eq:We}\left\{\begin{aligned}
     &-\sqrt{\nu}(\partial_Y^2-\alpha^2)W_e+\mathrm{i}\alpha\big((V-\lambda)W_e-V''\Phi_e\big) =-\mathrm{i}\alpha\big(V-V'(0)Y\big)W_a+\mathrm{i}\alpha V''\Phi_a,\\
     &(\partial_Y^2-\alpha^2)\Phi_e=W_e,\ (\partial_Y^2-\alpha^2)\Phi_a=W_a,\\
     &\Phi_a(0)=\Phi_e(0)=0,\ W_e(0)=0.
     \end{aligned}\right.
  \end{align}
We point out that $W_e$ satisfies the Navier-slip boundary condition. As a consequence, we have the following lemma.

\begin{lemma}\label{lem:Webounds}
  Let $W_e$ solve \eqref{eq:We}. Then it holds that
  \begin{align*}
     & \nu^{\f14}\alpha^{\f12}\lambda_i^{-\f12}\|W_{e}\|_{L^2} +\alpha\lambda_i\|(\partial_Y\Phi_{e},\alpha\Phi_{e})\|_{L^2} \leq C\alpha \lambda_i^{-1}A^{-\f72}.
  \end{align*}
 Moreover,  we have
  \begin{align*}
     & \|\partial_Y\Phi_e\|_{L^\infty}\leq C|n|^{-\f13}(A\lambda_i)^{-\f74}.
  \end{align*}
\end{lemma}
\begin{proof}
We first notice that $W_a$ satisfies
\begin{align*}
   & \big(V-V'(0)Y\big)W_a\\
   &=\partial_Y\big[(V-V'(0)Y)\partial_Y\Phi_a\big] -\partial_Y\big[(V'-V'(0))\Phi_a\big]+V''\Phi_a-\alpha^2(V-V'(0)Y)\Phi_a.
\end{align*}
Then we have
\begin{align*}
   & -\mathrm{i}\alpha\big(V-V'(0)Y\big)W_a+\mathrm{i}\alpha V''\Phi_a\\
   &=-\mathrm{i}\alpha \partial_Y\left[(V-V'(0)Y)\partial_Y\Phi_a -(V'-V'(0))\Phi_a\right]+\mathrm{i}\alpha^3(V-V'(0)Y)\Phi_a\\
   &=-\partial_YF_{1,1}+\mathrm{i}\alpha F_{1,2},
\end{align*}
where
\begin{align*}
   &F_{1,1}=\mathrm{i}\alpha\left[(V-V'(0)Y)\partial_Y\Phi_a -(V'-V'(0))\Phi_a\right],\\
   &F_{1,2}=\alpha^2(V-V'(0)Y)\Phi_a.
\end{align*}
  Since we have
  \begin{align*}
     & |V(Y)-V'(0)Y|=\left|\int_{0}^{Y}\int_{0}^{Z}V''(Z_1) \mathrm{d}Z_1\mathrm{d}Z\right|\leq Y^2\|V''\|_{L^\infty}/2\leq CY^2,\\
     & |V'(Y)-V'(0)|=\left|\int_{0}^{Y}V''(Z)\mathrm{d}Z\right|\leq Y\|V''\|_{L^\infty}\leq CY,
  \end{align*}
  we infer that
  \begin{align*}
     & |F_{1,1}(Y)|\leq C\alpha\big( \big|Y^2\partial_Y\Phi_a\big|+ \big|Y\Phi_a\big|\big),\qquad |F_{1,2}(Y)|\leq C\alpha^2\big|Y^2\Phi_a\big|.
  \end{align*}
 Applying Lemma \ref{lem:Airy-w} with $\kappa=|nV'(0)|^{\f13},\ \eta=-\lambda_{\nu}/V'(0)$(so, $\kappa(1+|\kappa\eta|)^{\f12}\sim A$), we get by  Lemma \ref{lem:L-index} that 
  \begin{align*}
     \|(F_{1,1},F_{1,2})\|_{L^2}\leq& C\alpha \left(\big\|Y^2\partial_Y\Phi_a\big\|_{L^2}+ \big\|Y\Phi_a\big\|_{L^2} + \alpha\big\|Y^2\Phi_a\big\|_{L^2}\right)\\
     \leq &C\alpha \left(A^{-\f72}+\alpha A^{-\f92}\right)=C\alpha A^{-\f72}\big(1+\alpha A^{-1}\big)\\
     \leq& C\alpha A^{-\f72}\big(1+\alpha |n|^{-\f13}\big)\leq C\alpha A^{-\f72}.
  \end{align*}

    Now we come to estimate $W_{e}$. We know that $(W_{e},\Phi_{e},F_{1,1},F_{1,2})$ fits the structure of \eqref{eq:reswNa1}. Then by Proposition \ref{pro:B3resH-1}, we get
  \begin{align*}
     & \nu^{\f14}\alpha^{\f12}\lambda_i^{-\f12}\|W_{e}\|_{L^2} +\alpha\lambda_i\|(\partial_Y\Phi_{e},\alpha\Phi_{e})\|_{L^2} \leq C\lambda_i^{-1}\big\|(F_{1,1},F_{1,2})\big\|_{L^2}\leq C\alpha A^{-\f72}\lambda_i^{-1}.
  \end{align*}
  This gives the first inequality. And we have
  \begin{align*}
     A\|\partial_Y\Phi_{e}\|_{L^\infty}&\leq CA\|W_{e}\|_{L^2}^{\f12}\|(\partial_Y\Phi_{e},\alpha\Phi_{e})\|_{L^2}^{\f12}\\
    &\leq C \alpha A^{-\f52}\lambda_i^{-1}\nu^{-\f18}\alpha^{-\f34}\lambda_i^{-\f14} =C(|n|^{\f13}/A)^{\f34}\lambda_i^{\f12}(A\lambda_i)^{-\f74}\leq C(A\lambda_i)^{-\f74}.
  \end{align*}
  Here we used $|n|^{\f13}/A\leq 1$, $\lambda_i\leq \delta_1^{-1}$.
 \end{proof}

Combining with Lemma \ref{lem:Webounds}, and Lemma \ref{lem:Airy-w}, we get
\begin{lemma}\label{lem:Wnorms}
  Let $W$ solve \eqref{eq:Hombound-OS}. Then it holds that
   \begin{align*}
      & \|(\partial_Y\Phi,\alpha\Phi)\|_{L^2}\leq CA^{-\f32},\quad 
   \|W\|_{L^2} \leq A^{-\f12},\quad  \|\rho^{\f12}_\lambda W\|_{L^2}\le C|n|^{\f14}\lambda_i^{\f34}A^{-1},
\end{align*}
where $\rho_\lambda$ is defined by 
\begin{align}\nonumber
     \rho_{\lambda}(Y)=\left\{\begin{aligned}
     &(|n|^{\f13}\lambda_i)^{\f32} Y,\qquad \text{if}\,\,\, 0\leq Y\leq (|n|^{\f13}\lambda_i)^{-\f32},\\
     &1,\qquad\qquad\quad \text{if}\,\,\,Y\geq (|n|^{\f13}\lambda_i)^{-\f32}.
     \end{aligned}\right.
  \end{align}
\end{lemma}
\begin{proof}
 From Lemma \ref{lem:Airy-w} and Lemma \ref{lem:Webounds}, we deduce that
  \begin{align*}
     \|(\partial_Y\Phi,\alpha\Phi)\|_{L^2} \leq&\|(\partial_Y\Phi_a,\alpha\Phi_a)\|_{L^2}+ \|(\partial_Y\Phi_e,\alpha\Phi_e)\|_{L^2}\\
     \leq &CA^{-\f32}+C\lambda_i^{-2}A^{-\f72}\leq CA^{-\f32}\big(1+(A\lambda_i)^{-2}\big)\\
     \leq & CA^{-\f32}\big(1+(|n|^{\f13}\lambda_i)^{-2}\big)\leq  CA^{-\f32},
  \end{align*}
  here we used $(A\lambda_i)^{-2}\leq (|n|^{\f13}\lambda_i)^{-2}\leq \delta^{2}\leq 1$.  This gives the first inequality. Applying Lemma \ref{lem:Airy-w} and Lemma \ref{lem:Webounds} again, we infer that
  \begin{align*}
     \|W\|_{L^2}\leq& \|W_a\|_{L^2}+\|W_e\|_{L^2}\leq CA^{-\f12}+C\nu^{-\f14}\alpha^{-\f12}\lambda_i^{\f12}\alpha A^{-\f72}\lambda_i^{-1}\\
     \leq &CA^{-\f12}\big(1+(|n|^\f13/A)^{\f32}(A\lambda_i)^{-\f32}\lambda_i\big)\leq CA^{-\f12},
  \end{align*}
provide that  $(|n|^{\f13}/A)\leq 1$, $(A\lambda_i)^{-\f32}\leq (|n|^{\f13}\lambda_i)^{-\f32}\leq \delta^{\f32}\leq 1$, $\lambda_i\leq \delta_1^{-1} $ .
Since $\rho_\lambda\leq1$, we get by Lemma \ref{lem:Airy-w} and Lemma \ref{lem:Webounds}  that 
\begin{align*}
\|\rho^{\f12}_\lambda W\|_{L^2}\leq&\|\rho^{\f12}_\lambda W_a\|_{L^2}+\|W_e\|_{L^2}\leq C((|n|^\f13\lambda_i)^{\f34}A^{-1}+\nu^{-\f14}\alpha^{\f12}\lambda_i^{\f12}A^{-\f72}\lambda_i^{-1})\\
=&C(|n|^{\f14}\lambda_i^{\f34}A^{-1}+|n|^{\f12}A^{-\f72}\lambda_i^{-\f12})\leq C|n|^{\f14}\lambda_i^{\f34}A^{-1}(1+|n|^\f14A^{-\f34}(A\lambda_i)^{-\f74}\lambda_i^{\f12})\\
\leq&C|n|^{\f14}\lambda_i^{\f34}A^{-1}.
\end{align*}

This proves the lemma.
\end{proof}

Now we denote
\begin{align*}
   &J:=-\int_{0}^{+\infty}W(Y)\mathrm{e}^{-\alpha Y}\mathrm{d}Y.
\end{align*}
Then by Lemma \ref{lem:ham-bound}, we find that $J=\partial_Y\Phi(0)$. Hence, the task of the following lemma is to show that the lower bound of $J$ is strictly positive.

\begin{lemma}\label{lem:lowerJ}
  Let $W$ be the solution of \eqref{eq:Hombound-OS}. Then it holds that
  \begin{align*}
     &|J|\geq C^{-1}A^{-1}.
  \end{align*}
\end{lemma}
\begin{proof}
  Thanks to Lemma \ref{lem:Airy-bound}, we have
  \begin{align*}
     \left|\int_{0}^{+\infty}W_a(Y)\mathrm{e}^{-\alpha Y}\mathrm{d}Y\right| =&|\partial_Y\Phi_a(0)| \geq C^{-1}(1+|n|^{\f13}|\lambda_\nu|)^{-\f12}(|n|^{\f13}+\alpha)^{-1},
  \end{align*}
  which along  with Lemma \ref{lem:L-index} gives
    \begin{align}\label{est:W1Philower}
     & \left|\int_{0}^{+\infty}W_a(Y)\mathrm{e}^{-\alpha Y}\mathrm{d}Y\right|\geq C^{-1}(1+|n|^{\f13}|\lambda_\nu|)^{-\f12}n^{-\f13}=C^{-1}A^{-1}.
  \end{align}
  And by Lemma \ref{lem:Webounds} and Lemma \ref{lem:L-index}, we get
  \begin{align*}
     \left|\int_{0}^{+\infty}W_e(Y)\mathrm{e}^{-\alpha Y}\mathrm{d}Y\right|=&|\partial_Y\Phi_e(0)|\leq \|\partial_Y\Phi_e\|_{L^\infty}\leq C|n|^{-\f13}(A\lambda_i)^{-\f74}\\
     \leq &CA^{-1}(|n|^{\f13}\lambda_i)^{-\f74}\leq CA^{-1}\delta^{\f74}.
  \end{align*}
  Then we deuce that
  \begin{align*}
  |J|&\geq \left|\int_{0}^{+\infty}W_a(Y)\mathrm{e}^{-\alpha Y}\mathrm{d}Y\right|- \left|\int_{0}^{+\infty}W_e(Y)\mathrm{e}^{-\alpha Y}\mathrm{d}Y\right|\geq C^{-1}A^{-1}\big(1-C\delta^{\f74}\big).
  \end{align*}
  Taking $\delta$ sufficiently small so that $C\delta^{\f74}\leq 1/2$, we arrive at
  \begin{align*}
     & |J|\geq C^{-1}A^{-1}.
  \end{align*}
 \end{proof}

From Lemma \ref{lem:lowerJ}, we can define  $W_b(Y)=W(Y)/J$. Then it follows Lemma \ref{lem:lowerJ} and Lemma \ref{lem:Wnorms} that 

\begin{proposition}\label{pro:Wbnorms}
Let  $W_b(Y)=W(Y)/J$. Then $W_b$ solves \eqref{eq:Hombound-OSnon} and satisfies 
  \begin{align*}
     & \|(\partial_Y\Phi_b,\alpha\Phi_b)\|_{L^2}\leq CA^{-\f12}\le C,\quad
   \|W_b\|_{L^2} \leq CA^{\f12},\quad  \big\|\rho^{\f12}_\lambda W_b\big\|_{L^2}\le C|n|^{\f14}\lambda_i^{\f34}. 
   \end{align*}
\end{proposition}

\subsubsection{Resolvent estimates for nonslip boundary condition}
Now we are back to consider the estimate for the solution to the system with nonslip boundary condition 
\begin{align}\label{eq:reswnon}\left\{\begin{aligned}
&-\sqrt{\nu}(\partial_Y^2-\alpha^2)w + \mathrm{i}\alpha\big((V-\lambda)w-(\partial_Y^2V)\phi\big)=F,\\
&(\partial_Y^2-\alpha^2)\phi=w,\ \partial_Y\phi|_{Y=0}=\phi|_{Y=0}=0,
\end{aligned}\right.
\end{align}
where  $F= -\partial_YF_1+\mathrm{i}\alpha F_2$. Let $w_{Na}$ solve
\begin{align*}\left\{\begin{aligned}
&-\sqrt{\nu}(\partial_Y^2-\alpha^2)w_{Na} + \mathrm{i}\alpha\big((V-\lambda)w_{Na}-(\partial_Y^2V)\phi_{Na}\big)=F,\\
&(\partial_Y^2-\alpha^2)\phi_{Na}=w_{Na},\ w_{Na}|_{Y=0}=\phi_{Na}|_{Y=0}=0.
\end{aligned}\right.
\end{align*}
By matching the boundary condition, we find that 
\begin{align*}
   &w(Y)=w_{Na}(Y)-\partial_Y\phi_{Na}(0)W_b(Y).
\end{align*}
Moreover, there hold the following estimates for the solution $(w,\phi)$.

\begin{proposition}\label{Pro:resdrsmall}
There exists $\delta_*\in(0,\delta_1]$ and $\nu_0$ such that the following statements hold. Let $\nu\leq \nu_0$. Suppose that \reff{Remu} holds for some $\gamma\in[\f23,1]$ and $\delta\in(0,\delta_*]$. Then for any $(F_1,F_2)\in L^2(\mathbb{R}_+)^2$, the  solution $\phi\in H^2_0(\mathbb{R}_+)$ to the system  \reff{eq:reswnon} satisfies
  \begin{align*}
     &\|(\partial_Y\phi,\alpha\phi)\|_{L^2}\leq C(\alpha\lambda_i)^{-1}\lambda_i^{-1}\|(F_1,F_2)\|_{L^2},\\
     &\|w\|_{L^2}\leq C\nu^{-\f14}\alpha^{-\f12}\lambda_i^{-\f54}\|(F_1,F_2)\|_{L^2},\\
     &\big\|\rho^{\f12}_\lambda w\big\|_{L^2}\leq C\nu^{-\f14}\alpha^{-\f12}\lambda_i^{-\f12}\|(F_1,F_2)\|_{L^2}.
  \end{align*}
\end{proposition}
\begin{proof}
By Proposition \ref{pro:B3resH-1}, we obtain
  \begin{align}\label{est:drsmallwna}
     &\nu^{\f14}\alpha^{\f12}\lambda_i^{-\f12}\|w_{Na}\|_{L^2}+ \alpha\lambda_i\|(\partial_Y\phi_{Na},\alpha\phi_{Na})\|_{L^2}\leq C\lambda_i^{-1}\|(F_1,F_2)\|_{L^2}.
  \end{align}
  Then by the interpolation, we get
  \begin{align*}
     |\partial_Y\phi_{Na}(0)\leq &\|\partial_Y\phi_{Na}\|_{L^{\infty}}\leq C\|w_{Na}\|_{L^2}^{\f12}\|(\partial_Y\phi_{Na},\alpha\phi_{Na})\|_{L^2}^{\f12}\\
     \leq & C\nu^{-\f18}\alpha^{-\f34}\lambda_i^{-\f54}\|(F_1,F_2)\|_{L^2}= C|n|^{\f14}\alpha^{-1}\lambda_i^{-\f54}\|(F_1,F_2)\|_{L^2}.
  \end{align*}
  This along with Proposition \ref{pro:Wbnorms} gives
  \begin{align*}
     &|\partial_Y\phi_{Na}(0)|\|(\partial_Y\Phi_b,\alpha\Phi_b)\|_{L^2} \leq C|n|^{\f14}\alpha^{-1}\lambda_i^{-\f54}A^{-\f12}\|(F_1,F_2)\|_{L^2},\\
     &|\partial_Y\phi_{Na}(0)|\|W_b\|_{L^2}\leq C|n|^{\f14}\alpha^{-1}\lambda_i^{-\f54}A^{\f12}\|(F_1,F_2)\|_{L^2}.
  \end{align*}
  Since $w(Y)=w_{Na}(Y)-\partial_Y\phi_{Na}(0)W_b(Y)$, the above inequality and \eqref{est:drsmallwna} give
  \begin{align*}
      \|(\partial_Y\phi,\alpha\phi)\|_{L^2}&\leq \|(\partial_Y\phi_{Na},\alpha\phi_{Na})\|_{L^2}+ |\partial_Y\phi_{Na}(0)|\|(\partial_Y\Phi_b,\alpha\Phi_b)\|_{L^2}\\
   &\leq C\big((\alpha\lambda_i)^{-1}\lambda_i^{-1} +|n|^{\f14}\alpha^{-1}\lambda_i^{-\f54}A^{-\f12}\big)\|(F_1,F_2)\|_{L^2}\\
     &= C(\alpha\lambda_i)^{-1}\lambda_i^{-1}\big(1+ (|n|^{\f12}\lambda_i^{\f12}/A)^{\f12}\lambda_i^{\f12}\big)\|(F_1,F_2)\|_{L^2}\\
     &\leq C(\alpha\lambda_i)^{-1}\lambda_i^{-1}\|(F_1,F_2)\|_{L^2},
  \end{align*}
  here we used the fact that  due to $\lambda_i\ge \f {|n|^{\gamma-1}} \delta\ge \f {|n|^{-\f13}} \delta\ge 2\sqrt{\nu}\alpha$, 
 \beno
|n|^{\f12}\lambda_i^{\f12}/A\leq C\f {|n|^\f13(|n|^\f13|\lambda_\nu|)^\f12} A\le C, \quad \lambda_i\leq \delta_1^{-1}.
\eeno
This gives the first inequality.

 Notice that
  \begin{align*}
     \|w\|_{L^2}\leq &\|w_{Na}\|_{L^2}+|\partial_Y\phi_{Na}(0)|\|W_b\|_{L^2}\\
     \leq & C\big(\nu^{-\f14}\alpha^{-\f12}\lambda_i^{-\f12}+ |n|^{\f14}\alpha^{-1}\lambda_i^{-\f54}A^{\f12}\big)\|(F_1,F_2)\|_{L^2}\\
     \leq &C(\alpha\lambda_i)^{-1}\lambda_i^{-1} \big(|n|^{\f12}\lambda_i^{\f32}+ |n|^{\f14}\lambda_i^{\f34}A^{\f12}\big)\|(F_1,F_2)\|_{L^2}\\
     \leq &C(\alpha\lambda_i)^{-1}\lambda_i^{-1} \big(|n|^{\f12}\lambda_i^{\f32}+ |n|^{\frac{5}{12}}\lambda_i^{\f34}+ |n|^{\f12}\lambda_i^{\f34}|\lambda_{\nu}|^{\f14}\big)\|(F_1,F_2)\|_{L^2}\\
     \leq &C(\alpha\lambda_i)^{-1}\lambda_i^{-1}  |n|^{\f12}\lambda_i^{\f34}\|(F_1,F_2)\|_{L^2}.
  \end{align*}
where in the last line, we used $|n|^{-\f13}+\lambda_i+|\lambda_{\nu}|\leq C$. This gives the second inequality. 

By Proposition \ref{pro:Wbnorms}, we also have
  \begin{align*}
     \big\|\rho^{\f12}_\lambda w\big\|_{L^2}&\leq \|w_{Na}\|_{L^2}+|\partial_Y\phi_{Na}(0)|\big\|\rho^{\f12}_\lambda W_b\big\|_{L^2}\\
     &\leq  C\big(\nu^{-\f14}\alpha^{-\f12}\lambda_i^{-\f12}+ |n|^{\f14}\alpha^{-1}\lambda_i^{-\f54}|n|^{\f14}\lambda_i^{\f34}\|(F_1,F_2)\|_{L^2}\big)\\
     &=C\nu^{-\f14}\alpha^{-\f12}\lambda_i^{-\f12}\|(F_1,F_2)\|_{L^2}.
  \end{align*}

This finishes the proof of the proposition.
\end{proof}

\section{Semigroup  estimates of  $e^{-t\mathbb{A}_\nu}$}
This section is devoted to the semigroup estimates of $e^{-t\mathbb{A}_\nu}$. More precisely, we will establish the semigroup estimates for 
$e^{-t\mathbb{A}_{\nu,n}}$, where  $\mathbb{A}_{\nu,n}$ is the restriction of $\mathbb{A}_\nu$ on the subspace $\mathcal{P}_{n}L^2_\sigma(\Omega)$. In the first part of this section, we show the $L^2$ estimates of the semigroup $e^{-t\mathbb{A}_{\nu,n}}$, and in the second part, we give the $L^\infty$ estimates of the semigroup $e^{-t\mathbb{A}_{\nu,n}}$.

\subsection{$L^2$ estimate of the semigroup $e^{-t\mathbb{A}_{\nu,n}}$}

This part is devoted to $L^2-L^2$ estimates of the semigroup $e^{-t\mathbb{A}_{\nu,n}}$. The following is our main result.

\begin{proposition}\label{semigroup-L2}
Assume that $(SC)$ condition holds. Then there exist $\delta_1,\delta_2,\delta_*\in(0,1)$ satisfying $\delta_1,\delta_2\leq \delta_0$ and $\delta_*\leq \min\{\delta_1,\delta_2\}$ such that the following statements hold true. Assume that \reff{Remu} holds for some $\delta\in(0,\delta_*]$ and $\gamma\in[\f23,1]$. Then the following estimates hold for all $f\in \mathcal{P}_nL^2_\sigma(\Omega)$ and $t>0$.
\begin{enumerate}
\item If $|n|\leq \delta_0^{-1}$, then
\begin{align*}
&\|e^{-t\mathbb{A}_{\nu,n}}f\|_{L^2}\leq Ce^{ct}\|f\|_{L^2},\\
&\|\nabla e^{-t\mathbb{A}_{\nu,n}}f\|_{L^2}\leq \frac{C}{\sqrt{\nu t}}(1+te^{ct})\|f\|_{L^2}.
\end{align*}
\item If $\delta_0^{-1}\leq |n|\leq \delta_0^{-1}\nu^{-3/4}$ and $|n|^\gamma\nu^{\frac{1}{2}}<1$, then
\begin{align*}
&\|e^{-t\mathbb{A}_{\nu,n}}f\|_{L^2}\leq C |n|^{2(1-\gamma)}e^{\frac{|n|^\gamma}{\delta}t}\|f\|_{L^2},\\
&\|\nabla e^{-t\mathbb{A}_{\nu,n}}f\|_{L^2}\leq \frac{C}{\nu^{1/2}}\Big(t^{-1/2}+|n|^{\f54(1-\gamma)+\f12}e^{\frac{|n|^\gamma}{\delta}t}\Big)\|f\|_{L^2}.
\end{align*}
\item If $|n|^\gamma\nu^{\frac{1}{2}}\geq 1$ and $|n|\leq\delta_0^{-1}\nu^{-\f34}$,  then
\begin{align*}
&\|e^{-t\mathbb{A}_{\nu,n}}f\|_{L^2}\leq C|n|^{1-\gamma}e^{\frac{|n|^\gamma}{\delta}t}\|f\|_{L^2},\\
&\|\nabla e^{-t\mathbb{A}_{\nu,n}}f\|_{L^2}\leq \frac{C}{\nu^{1/2}}\Big(t^{-1/2}+|n|^{(1-\gamma/2)}e^{\frac{|n|^\gamma}{\delta}t}\Big)\|f\|_{L^2}.
\end{align*}
\item If $|n|\geq \delta_0^{-1}\nu^{-3/4}$, then
\begin{align*}
&\|e^{-t\mathbb{A}_{\nu,n}}f\|_{L^2}\leq e^{-\frac{1}{4}\nu n^2t}\|f\|_{L^2},\\
&\|\nabla e^{-t\mathbb{A}_{\nu,n}}f\|_{L^2}\leq \frac{Ce^{-\frac{1}{4}\nu n^2t}}{\sqrt{\nu t}}(1+|n|t)\|f\|_{L^2}.
\end{align*}
\end{enumerate}
Here $C$ and $c$ are universal constants only depending on $U^P$.
\end{proposition}
We mention that the results of (1) and (4) are obtained via standard energy method. However, for the results of  (2) and (4) which are estimates in Gevrey class for the mid-range frequency, we need to use the corresponding resolvent estimates in Theorem \ref{main-resolvent}. The idea of the proof is similar to \cite{GMM}. For the completeness, we present the details of this proof.

\begin{proof}
Let $n\in \mathbb{Z}$ and $f\in\mathcal{P}_n L^2_{\sigma}(\Omega)$. We denote $u^{(n)}:=e^{-t\mathbb{A}_{\nu,n}}f$. Then by the definition of semigroup $e^{-t\mathbb{A}_{\nu,n}}$, we know that $u^{(n)}$ satisfies
\begin{align*}
\partial_t u^{(n)}-\nu\Delta u^{(n)}+U^{p}\big(\frac{y}{\sqrt{\nu}}\big)\partial_x u^{(n)}+\frac{1}{\sqrt{\nu}}\Big(u_2^{(n)}\partial_YU^p(\frac{y}{\sqrt{\nu}}),0\Big)+\mathcal{P}_n\nabla p=0,\quad u^{(n)}|_{t=0}=f,
\end{align*}
which by introducing $Y=\frac{y}{\sqrt{\nu}}$ can be written as
\begin{align}\label{n-mode-velocity}
\partial_t u^{(n)}-\nu\Delta u^{(n)}+U^{p}(Y)\partial_x u^{(n)}+\big(y^{-1}u_2^{(n)}Y\partial_YU^p(Y),0\big)+\mathcal{P}_n\nabla p=0.
\end{align}
 Recall  that
\begin{align}
\delta_0=\frac{1}{2(1+\|U^P\|)},
\end{align}
where $\|\cdot\|$ is defined in (SC) condition.

Now we start to prove the first statement of Proposition \ref{semigroup-L2}. We first notice that by the boundary condition and divergence free condition
\begin{align}
\|y^{-1}u^{(n)}_2\|_{L^2(\Omega)}\leq2\|\partial_y u^{(n)}_2\|_{L^2(\Omega)}=2\|\partial_x u^{(n)}_1\|_{L^2(\Omega)}\leq 2|n|\|u^{(n)}\|_{L^2(\Omega)}.
\end{align}
By taking $L^2$-inner product with $u^{(n)}$ on  both sides of \reff{n-mode-velocity}, we have
\begin{align*}
\frac{\mathrm{d}}{\mathrm{d}t}\|u^{(n)}\|^2_{L^2(\Omega)}&=-2\nu\|\nabla u^{(n)}\|_{L^2(\Omega)}^2-2\mathbf{Re}\langle y^{-1}u^{(n)}_2(y)Y\partial_Y U^P(Y),u^{(n)}_1\rangle_{L^2}\\
&\leq -2\nu\|\nabla u^{(n)}\|_{L^2(\Omega)}^2+2\|Y\partial_Y U^P\|_{L^\infty}|n|\|u^{(n)}\|^2_{L^2(\Omega)}.
\end{align*}
Hence by Gronwall's inequality, we obtain that for any $n\in\mathbb{Z}$,
\begin{align}\label{semi-L2-n1}
\|u^{(n)}(t)\|_{L^2(\Omega)}\leq e^{\frac{|n|t}{\delta_0}}\|f\|_{L^2(\Omega)}.
\end{align}
To prove the derivative estimates, by the Duhamel formula, we have
\begin{align}
u^{(n)}(t)=e^{\nu t\mathbb{P}\Delta}f-\int_0^t e^{\nu(t-s)\mathbb{P}\Delta}\mathbb{P}\big(U^{p}(Y)\partial_x u^{(n)}(s)+(y^{-1}u_2^{(n)}(s)Y\partial_YU^p(Y),0)\big)\mathrm{d}s,
\end{align}
which along with classical estimates of the Stokes semigroup gives
\begin{align*}
\|\nabla u^{(n)}(t)\|_{L^2(\Omega)}\leq C(\nu t)^{-\f12}\|f\|_{L^2(\Omega)}+C|n|\|U^P\|\sup_{0\leq s\leq t}\|u^{(n)}(s)\|_{L^2(\Omega)}\int_0^t\frac{1}{\nu^{\f12}(t-s)^{\f12}}\mathrm{d}s.
\end{align*}
Combing the above inequality and \reff{semi-L2-n1}, we obtain
\begin{align}
\|\nabla u^{(n)}(t)\|_{L^2}\leq \frac{C}{\sqrt{\nu t}}(1+|n|te^{\frac{|n|t}{\delta_0}})\|f\|_{L^2}.
\end{align}
This shows the first statement in the proposition by taking $|n|\leq \delta_0^{-1}$.\smallskip

Now we turn to prove the second and third statement in Proposition \ref{semigroup-L2}. By rescaling, we have
\begin{align*}
u^{(n)}(t,x,y)=(e^{-t\mathbb{A}_{\nu,n}})f(x,y)=(e^{-\tau\mathbb{L}_{\nu,n}}f_{\nu})(X,Y),
\end{align*}
where $(\tau,X,Y)=(t/\sqrt{\nu},x/\sqrt{\nu},y/\sqrt{\nu})$, $f_\nu(X,Y)=f(\nu^{\f12}X,\nu^{\f12}Y)$. According to \reff{semi-L2-n1}, after rescaling, we have already known that $-\mathbb{L}_{\nu,n}$ generates a $C_0$-semigroup acting on $\mathcal{P}_{\nu,n}L^2_\sigma(\Omega_\nu)$. In details, from \reff{semi-L2-n1}, we infer that for any $g\in\mathcal{P}_{\nu,n}L^2_\sigma(\Omega_\nu)$,
\begin{align}
\|e^{-\tau\mathbb{L}_{\nu,n}}g\|_{L^2_\sigma(\Omega_\nu)}\leq e^{\frac{\sqrt{\nu}|n|\tau}{\delta}}\|g\|_{L^2_\sigma(\Omega_\nu)},
\end{align}
which implies the results of the second and third statement for short time $0\leq\tau\leq \nu^{-\f12}|n|^{-1}$. Hence, we now only need to consider the case of $\tau\geq \nu^{-\f12}|n|^{-1}$. According to Theorem \ref{main-resolvent}, we know that the set
\begin{align}
\Sigma_{\nu,\gamma}:=S_{\nu,n}(\theta)\cup\Big\{\mu\in\mathbb{C}|\mathbf{Re}\mu\geq\frac{n^{\gamma}\nu^{\f12}}{\delta}\Big\}
\end{align}
is included in the resolvent set of $-\mathbb{L}_{\nu,n}$, where $S_{\nu,n}(\theta)$ is the set defined in Theorem \ref{main-resolvent}. Thus, the semigroup $e^{-\tau\mathbb{L}_{\nu,n}}$ can be represented as
\begin{align}\label{semi-repre}
e^{-\tau\mathbb{L}_{\nu,n}}=\frac{1}{2\pi\mathrm{i}}\int_{\Gamma}e^{\tau\mu}(\mu+\mathbb{L}_{\nu,n})^{-1}\mathrm{d}\mu,
\end{align}
where the curve $\Gamma$ is taken as $\Gamma=\Gamma_++\Gamma_-+l_++l_-+l_0$ with
\begin{eqnarray}\label{curve-Gamma}
\begin{split}
\Gamma_{\pm}&:=\big\{\mu\in\mathbb{C}|\pm\mathbf{Im}\mu=(\tan\theta)\mathbf{Re}\mu +\delta^{-1}_1(\sqrt{\nu}n+|\tan\theta||n|^\gamma\nu^{\f12}),\mathbf{Re}\mu\leq0\big\},\\
l_{\pm}&:=\Big\{\mu\in\mathbb{C}|\pm\mathbf{Im}\mu=\delta^{-1}_1(\sqrt{\nu}|n|+|\tan\theta||n|^\gamma\nu^{\f12}),0\leq\mathbf{Re}\mu\leq\frac{|n|^\gamma\nu^{\f12}}{\delta}\Big\},\\
l_0&:=\Big\{\mu\in\mathbb{C}|0\leq|\mathbf{Im}\mu|\leq\delta^{-1}_1(\sqrt{\nu}|n|+|\tan\theta||n|^\gamma\nu^{\f12}),\mathbf{Re}\mu=\frac{|n|^\gamma\nu^{\f12}}{\delta}\Big\}.
\end{split}
\end{eqnarray}
By \reff{mularge}, we have that for any $g\in\mathcal{P}_{\nu,n}L^2(\Omega_\nu)$
\begin{align*}
\Big\|\frac{1}{2\pi\mathrm{i}}\int_{\Gamma_{\pm}}e^{\tau\mu}(\mu+\mathbb{L}_{\nu,n})^{-1}g\mathrm{d}&\mu\Big \|_{L^2(\Omega_\nu)}\leq C\|g\|_{L^2(\Omega_\nu)}\Big|\int_{\Gamma_{\pm}}e^{\tau\mathbf{Re}\mu}|\mu|^{-1}\mathrm{d}\mu\Big|\\
&\leq C\|g\|_{L^2(\Omega_\nu)}\int_0^{+\infty}\frac{e^{-\tau s}}{s+|\tan\theta|s+\delta^{-1}_1(\sqrt{\nu}n+|\tan\theta|n^\gamma\nu^{\f12})}\mathrm{d}s\\
&\leq C(\sqrt{\nu} n \tau)^{-\f12}\|g\|_{L^2(\Omega_\nu)},
\end{align*}
which implies that for any $\tau\geq (\nu^{1/2}n)^{-1}$
\begin{align}\label{Gamma-L2}
\Big\|\frac{1}{2\pi\mathrm{i}}\int_{\Gamma_{\pm}}e^{\tau\mu}(\mu+\mathbb{L}_{\nu,n})^{-1}g\mathrm{d}&\mu\Big\|_{L^2(\Omega_\nu)}\leq C\|g\|_{L^2(\Omega_\nu)}.
\end{align}
Again by \reff{mularge}, we have
\begin{eqnarray}\label{l-L2}
\begin{split}
&\Big\|\frac{1}{2\pi\mathrm{i}}\int_{l_{\pm}}e^{\tau\mu}(\mu+\mathbb{L}_{\nu,n})^{-1}g\mathrm{d}\mu \Big\|_{L^2(\Omega_\nu)}\\
&\leq C\|g\|_{L^2(\Omega_\nu)}\int_0^{\frac{|n|^\gamma\nu^{\f12}}{\delta}}\frac{e^{\tau s}}{s+\delta^{-1}_1(\sqrt{\nu}|n|+|\tan\theta||n|^{\gamma}\nu^{\f12})}\mathrm{d}s\\
&\leq Ce^{\frac{|n|^{\gamma}\nu^{\f12}\tau}{\delta}}\|g\|_{L^2(\Omega_\nu)}.
\end{split}
\end{eqnarray}

On $l_0$, we apply \reff{Immularge} and \reff{musmall} for the case $|n|^\gamma\nu^{\f12}\geq 1$ and $|n|^\gamma\nu^{\f12}\leq 1$ respectively. If $|n|^\gamma\nu^{\f12}\geq 1$, then $|n|^\gamma\nu^{\f12}+\delta|n|^2\nu^{\f32}\geq\delta\delta^{-1}_2$, which allow us to use \reff{Immularge}. Hence,  we have
\begin{eqnarray}\label{l0-L2-large}
\begin{split}
&\Big\|\frac{1}{2\pi\mathrm{i}}\int_{l_{0}}e^{\tau\mu}(\mu+\mathbb{L}_{\nu,n})^{-1}g\mathrm{d}\mu\Big\|_{L^2(\Omega_\nu)}\\
&\leq C\frac{\delta}{|n|^\gamma\nu^{\f12}}e^{\frac{|n|^{\gamma}\nu^{\f12}\tau}{\delta}}\|g\|_{L^2(\Omega_\nu)}\int_0^{\frac{C\sqrt{\nu}|n|}{\delta_1}}\mathrm{d}s\\
&\leq C|n|^{1-\gamma}e^{\frac{|n|^{\gamma}\nu^{\f12}\tau}{\delta}}\|g\|_{L^2(\Omega_\nu)}.
\end{split}
\end{eqnarray}
If $|n|^\gamma\nu^{\f12}\leq 1$, by \reff{musmall} we obtain
\begin{eqnarray}\label{l0-L2-small}
\begin{split}
&\Big\|\frac{1}{2\pi\mathrm{i}}\int_{l_{0}}e^{\tau\mu}(\mu+\mathbb{L}_{\nu,n})^{-1}g\mathrm{d}\mu \Big\|_{L^2(\Omega_\nu)}\\
&\leq C\frac{\delta}{|n|^\gamma\nu^{\f12}}|n|^{1-\gamma}e^{\frac{|n|^{\gamma}\nu^{\f12}\tau}{\delta}}\|g\|_{L^2(\Omega_\nu)}\int_0^{\frac{C\sqrt{\nu}|n|}{\delta_1}}\mathrm{d}s\\
&\leq C|n|^{2(1-\gamma)}e^{\frac{|n|^{\gamma}\nu^{\f12}\tau}{\delta}}\|g\|_{L^2(\Omega_\nu)}.
\end{split}
\end{eqnarray}
Therefore, we prove the $L^2$ estimates in the second and third statement of Proposition \ref{semigroup-L2} by \reff{Gamma-L2}, \reff{l-L2}, \reff{l0-L2-large} and \reff{l0-L2-small}.

Next we consider the derivative estimates. We use the representation \reff{semi-repre} again. For the integral on $\Gamma_{\pm}$,  we have that by \reff{mularge-nabla},
\begin{eqnarray}\label{Gamma-nabla}
\begin{split}
&\Big\|\nabla_{X,Y}\frac{1}{2\pi\mathrm{i}}\int_{\Gamma_{\pm}}e^{\tau\mu}(\mu+\mathbb{L}_{\nu,n})^{-1}g\mathrm{d}\mu\Big\|_{L^2(\Omega_\nu)}\\
&\leq C\nu^{-\f14}\|g\|_{L^2(\Omega_\nu)}\big|\int_{\Gamma_{\pm}}e^{\tau\mathbf{Re}\mu}|\mu|^{-\f12}\mathrm{d}\mu\big|\\
&\leq C\nu^{-\f14}\|g\|_{L^2(\Omega_\nu)}\int_0^{+\infty}\frac{e^{-\tau s}}{(s+|\tan\theta|s+\delta_1^{-1}(\sqrt{\nu}|n|+|\tan\theta||n|^{\gamma}\nu^{\f12}))^{\f12}}\mathrm{d}s\\
&\leq C\frac{C}{\nu^{\f14}\tau^{\f12}}\|g\|_{L^2(\Omega_\nu)}.
\end{split}
\end{eqnarray}
In a similar way, on $l_{\pm}$, we have
\begin{eqnarray}
\begin{split}
&\Big\|\nabla_{X,Y}\frac{1}{2\pi\mathrm{i}}\int_{l_{\pm}}e^{\tau\mu}(\mu+\mathbb{L}_{\nu,n})^{-1}g\mathrm{d}\mu\Big\|_{L^2(\Omega_\nu)}\\
&\leq C\nu^{-\f14}\|g\|_{L^2(\Omega_\nu)}\int_0^{\frac{|n|^\gamma\nu^{\f12}}{\delta}}\frac{e^{\tau s}}{(s+\delta_1^{-1}(\sqrt{\nu}|n|+|\tan\theta||n|^\gamma\nu^{\f12}))^{\f12}}\mathrm{d}s\\
&\leq C|n|^{\gamma-\f12}e^{\frac{|n|^\gamma\nu^{\f12}\tau}{\delta}}\|g\|_{L^2(\Omega_\nu)}.
\end{split}
\end{eqnarray}
We obtain the estimates on $l_0$ in the same way as in \reff{l0-L2-large} and \reff{l0-L2-small}. In details, for $|n|^\gamma\nu^{\f12}\geq1$, we have
\begin{align}
\Big\|\nabla_{X,Y}\frac{1}{2\pi\mathrm{i}}&\int_{l_{0}}e^{\tau\mu}(\mu+\mathbb{L}_{\nu,n})^{-1}g\mathrm{d}\mu\Big\|_{L^2(\Omega_\nu)}\leq C|n|^{1-\frac{\gamma}{2}}e^{\frac{|n|^\gamma\nu^{\f12}\tau}{\delta}}\|g\|_{L^2(\Omega_\nu)},
\end{align}
and for $|n|^\gamma\nu^{\f12}\leq1$, we have
\begin{align}\label{l0-nalbal}
&\Big\|\nabla_{X,Y}\frac{1}{2\pi\mathrm{i}}\int_{l_{0}}e^{\tau\mu}(\mu+\mathbb{L}_{\nu,n})^{-1}g\mathrm{d}\mu\Big\|_{L^2(\Omega_\nu)}\leq Cn^{\f54(1-\gamma)+\f12}e^{\frac{|n|^\gamma\nu^{\f12}\tau}{\delta}}\|g\|_{L^2(\Omega_\nu)}.
\end{align}
Combining \reff{Gamma-nabla}-\reff{l0-nalbal} and scaling back to the original variables, we finish the proof of the second and third statement of Proposition \ref{semigroup-L2}.\smallskip

Finally, we deal with the last statement of the proposition. In this situation, we are back to the system \reff{n-mode-velocity}. Recall that we consider the case of $|n|\geq \delta_0^{-1}\nu^{-\f34}$. By the standard energy method, we obtain
\begin{align*}
\frac{\mathrm{d}}{\mathrm{d}t}\|u^{(n)}\|^2_{L^2(\Omega)}&=-2\nu\|\nabla u^{(n)}\|_{L^2(\Omega)}^2-2\nu^{-\f12}\mathbf{Re}\langle u^{(n)}_2(y)\partial_Y U^P(Y),u^{(n)}_1\rangle_{L^2}\\
&\leq -\nu\|\nabla u^{(n)}\|_{L^2(\Omega)}^2-\nu n^2\|u^{(n)}\|_{L^2}^2+2\nu^{-\f12}\|\partial_Y U^P\|_{L^\infty}\|u\|^2_{L^2}.
\end{align*}
Then the above inequality and $|n|\geq \delta^{-1}_0\nu^{-\f34}$ lead to
\begin{align*}
\frac{\mathrm{d}}{\mathrm{d}t}\|u^{(n)}\|^2_{L^2(\Omega)}\leq -\frac{\nu n^2}{2}\|u^{(n)}\|^2_{L^2(\Omega)}.
\end{align*}
Hence by the Gronwall's inequality, we have
\begin{align*}
\|e^{-t\mathbb{A}_{\nu,n}}f\|_{L^2}\leq e^{-\frac{1}{4}\nu n^2t}\|f\|_{L^2}.
\end{align*}
As in the proof of the first statement, by the Duhamel formula and the above $L^2$-estimates, we can obtain the derivative estimates in  the last statement.
\end{proof}

\subsection{$L^\infty$-estimates of semigroup $e^{-t\mathbb{A}_{\nu,n}}$}
In this part, we establish two kinds of $L^\infty$-estimates of the semigroup $e^{-t\mathbb{A}_{\nu,n}}$. The first one stated as follows is the $L^2-L^\infty$ estimate, which is used to controll  nonlinear part of the full nonlinear perturbation system \eqref{NSA}.

\begin{proposition}\label{semigoup-L2Linf}
Assume that $(SC)$ condition holds. Then there exist $\delta_1,\delta_2,\delta_*\in(0,1)$ satisfying $\delta_1,\delta_2\leq \delta_0$ and $\delta_*\leq \min\{\delta_1,\delta_2\}$ such that the following statements hold true. Assume that \reff{Remu} holds for some $\delta\in(0,\delta_*]$ and $\gamma\in[\f23,1]$.
\begin{enumerate}
\item If $|n|^\gamma\nu^{\f12}\leq1$ and $|n|\le \delta_0^{-1}\nu^{-\f34}$, then
\begin{eqnarray}\nonumber
\begin{split}
&\|e^{-t\mathbb{A}_{\nu,n}}f\|_{L^2_xL^\infty_y}\\
&\leq C\an^{\f34(1-\gamma)}\log^\f12(\an)\frac{1}{\nu^{\f14}\an^{\f14}t^{\f12}}\|f\|_{L^2}+C\frac{\an^{1-\gamma}}{\nu^{\f14}t^{\f14}}e^{\frac{\an^\gamma t}{2\delta}}\|f\|_{L^2}\\
&\qquad+C\nu^{-\f14}\log^{\f12}(\an)\an^{\frac{3}{2}-\f54\gamma}e^{\frac{\an^\gamma t}{\delta}}\|f\|_{L^2},
\end{split}
\end{eqnarray}
here $\an=(1+n^2)^\f12$.

\item If $|n|^\gamma\nu^{\frac{1}{2}}\geq 1$ and $|n|\leq\delta_0^{-1}\nu^{-\f34}$, then
\begin{align*}
&\|e^{-t\mathbb{A}_{\nu,n}}f\|_{L^2_xL^\infty_y}
\leq C\frac{|n|^{(1-\gamma)/2}}{\nu^{\f14}t^{\f14}}e^{\frac{|n|^\gamma t}{2\delta}}\|f\|_{L^2}+C\nu^{-\f14}|n|^{1-\f34\gamma}e^{\frac{|n|^\gamma t}{\delta}}\|f\|_{L^2}.
\end{align*}

\item If $|n|\geq\delta_0^{-1}\nu^{-\f34}$, then
\begin{align*}
&\|e^{-t\mathbb{A}_{\nu,n}}f\|_{L^2_xL^\infty_y}\leq C\frac{1}{\nu^{\f14}t^{\f14}}(1+|n|^{\f12}t^{\f12})e^{-\frac{1}{4}\nu|n|^2t}\|f\|_{L^2}.
\end{align*}

\end{enumerate}
\end{proposition}

\begin{remark}
We point out that the second and third statement of the above proposition follow from Proposition \ref{semigroup-L2} by the interpolation. However, for the case of $|n|^\gamma\nu^{\f12}\leq 1$, the $L^\infty$ estimate of the semigroup $e^{-t\mathbb{A}_\nu}$ will loss many derivatives(the power of $|n|$), if we apply the interpolation directly. This will in turn cause more smallness requirement on the initial perturbation when we apply it to the full nonlinear system. To obtain a sharper bounds for the case of $|n|^\gamma\nu^{\f12}\le 1$, we introduce a weight function $\rho(Y)$ matching boundary layer quite well and obtain a sharper $L^\infty$ bound by a $H^1$ resolvent estimate associated to the weight function $\rho(Y)$ defined by \eqref{rho-lambda}.
\end{remark}

\begin{proof}
Let $n\in\mathbb{Z}$ and $f\in\mathcal{P}_nL^2_\sigma(\Omega)$. By the interpolation and Proposition \ref{semigroup-L2}, we obtain that for any $|n|^\gamma\nu^{\frac{1}{2}}\geq 1$ and $|n|\leq\delta_0^{-1}\nu^{-\f34}$,
\begin{align*}
\|u^{(n)}\|_{L^2_xL^\infty_y}\leq& \|u^{(n)}\|_{L^2_{x,y}}^{\f12}\|\partial_y u^{(n)}\|_{L^2_{x,y}}^{\f12}\\
\leq&C\frac{|n|^{(1-\gamma)/2}}{\nu^{\f14}t^{\f14}}e^{\frac{|n|^\gamma t}{2\delta}}\|f\|_{L^2_{x,y}}+C\nu^{-\f14}|n|^{1-\f34\gamma}e^{\frac{|n|^\gamma t}{\delta}}\|{f}\|_{L^2_{x,y}},
\end{align*}
and for any $|n|\geq\delta_0^{-1}\nu^{-\f34}$,
\begin{align*}
\|u^{(n)}\|_{L^2_xL^\infty_y}\leq& \|u^{(n)}\|_{L^2_{x,y}}^{\f12}\|\partial_y u^{(n)}\|_{L^2_{x,y}}^{\f12}\\
\leq&C\frac{1}{\nu^{\f14}t^{\f14}}(1+|n|^{\f12}t^{\f12})e^{-\frac{1}{4}\nu
n^2t}\|f\|_{L^2_{x,y}}.
\end{align*}

Next we consider the case of  $|n|^{\gamma}\nu^{\f12}\leq 1$ and $\delta_0^{-1}\le |n|\leq\delta_0^{-1}\nu^{-\f34}$. Recall that 
\begin{align*}
v^{(n)}(t,X,Y):=(e^{-\tau\mathbb{L}_{\nu,n}}f_{\nu})(X,Y)=(e^{-t\mathbb{A}_{\nu,n}})f(x,y)=u^{(n)}(t,x,y),
\end{align*}
where $(\tau,X,Y)=(t/\sqrt{\nu},x/\sqrt{\nu},y/\sqrt{\nu})$, $f_\nu(X,Y)=f(\nu^{\f12}X,\nu^{\f12}Y)$. Hence, we directly have
\begin{align}\label{Rescalenorm}
\|u^{(n)}(t)\|_{L^2_xL^\infty_y(\Omega)}=\nu^{\f14}\|v^{(n)}(\tau)\|_{L^2_XL^\infty_Y(\Omega_\nu)},
\qquad \|f\|_{L^2_xL^2_y(\Omega)}=\nu^{\f12}\|f_{\nu}\|_{L^2_XL^2_Y(\Omega_{\nu})}.
\end{align}
On the other hand,  we get by Lemma \ref{lem:inter} that 
\begin{align*}
\|v^{(n)}(\tau)\|_{L^2_XL^\infty_Y}\leq& C\|\rho^{\f12}\omega^{(n)}(\tau)\|^{\f12}_{L^2_{X,Y}}\|v^{(n)}(\tau)\|^{\f12}_{L^2_{X,Y}}+C\|(1-\rho^{\f12})\omega^{(n)}(\tau)\|_{L^2_XL^1_Y}\\
&+C\nu^{\f14}|n|^{\f12}|\|v^{(n)}(\tau)\|_{L^2_{X,Y}},
\end{align*}
where $\omega^{(n)}:=\mathrm{curl}_{X,Y}v^{(n)}$.  For the second term on the right hand side, we notice that by taking $h_0=(|n|^{\gamma-2/3}/\delta)^{-\f32}$ and $h_1:=h_0|n|^{3(\gamma-1)/2}$
\begin{align*}
   \big\|(1-\rho^{\f12})\omega^{(n)}\big\|_{L^2_XL^1_Y}\leq & \|\omega^{(n)}\|_{L^2_XL^1_Y(h_1,h_2)}+\|\omega^{(n)}\|_{L^2_XL_Y^1(0,h_1)}\\
   \leq &\big\|\rho^{-\f12}\big\|_{L^2_Y(h_1,h_0)}\big\|\rho^{\f12}\omega^{(n)}\big\|_{L^2_{X,Y}}+ h_1^{\f12}\|\omega^{(n)}\|_{L^2_{X,Y}}\\
   \leq &C|n|^{1/2-3\gamma/4}\big(\log(|n|)\big)^{\f12}\big\|\rho^{\f12}\omega^{(n)}\big\|_{L^2_{X,Y}} +C|n|^{-\f14}\|\omega^{(n)}\|_{L^2_{X,Y}}.
\end{align*}
Hence, by the above two estimates,   we get
\begin{eqnarray}
\begin{split}
\|v^{(n)}(\tau)\|_{L^2_XL^\infty_Y}\leq&C\|\rho^{\f12}\omega^{(n)}(\tau)\|^{\f12}_{L^2_{X,Y}}\|v^{(n)}(\tau)\|^{\f12}_{L^2_{X,Y}}\\
&+C|n|^{1/2-3\gamma/4}\big(\log(|n|)\big)^{\f12}\big\|\rho^{\f12}\omega^{(n)}(\tau)\big\|_{L^2_{X,Y}}\\
&+C|n|^{-\f14}\|\omega^{(n)}(\tau)\|_{L^2_{X,Y}}+C\nu^{\f14}|n|^{\f12}|\|v^{(n)}(\tau)\|_{L^2_{X,Y}}.
\end{split}
\end{eqnarray}
Applying Proposition \ref{semigroup-L2} and after scaling, we deduce  that for $\delta_0^{-1}\leq |n|\leq \delta_0^{-1}\nu^{-3/4}$ and $|n|^\gamma\nu^{\frac{1}{2}}<1$,
\begin{align*}
&\|v^{(n)}(\tau)\|_{L^2_{X,Y}}\leq C |n|^{2(1-\gamma)}e^{\frac{|n|^\gamma}{\delta}\nu^{\f12}\tau}\|f_{\nu}\|_{L^2_{X,Y}},\\
&\|\omega^{(n)}(\tau)\|_{L^2_{X,Y}}\leq C\Big(\frac{1}{\nu^{\f14}\tau^{\f12}}+n^{\f54(1-\gamma)+\f12}e^{\frac{|n|^\gamma\nu^{\f12}\tau}{\delta}}\Big)\|f_\nu\|_{L^2_{X,Y}}.
\end{align*}
Hence, in order to obtain the estimate of $\|v^{(n)}(\tau)\|_{L^2_XL^\infty_Y}$, we are left with the control of $\|\rho^{\f12}\omega^{(n)}(\tau)\|^{\f12}_{L^2_{X,Y}}$. For this purpose, we are back to the formula
\begin{align*}
v^{(n)}(\tau)=e^{-\tau\mathbb{L}_{\nu,n}}f_\nu=\frac{1}{2\pi\mathrm{i}}\int_{\Gamma}e^{\tau\mu}(\mu+\mathbb{L}_{\nu,n})^{-1}f_\nu\mathrm{d}\mu,
\end{align*}
where $\Gamma$ is given in \reff{curve-Gamma}. Then by using \reff{mularge-nabla}, \reff{Immularge-na} and \reff{musmall-wegihted}, we infer that for any $\delta_0^{-1}\leq |n|\leq \delta_0^{-1}\nu^{-\f34}$,
\begin{align}
\|\rho^{\f12}\omega^{(n)}(\tau)\|_{L^2}\leq C\Big(\frac{1}{\nu^{\f14}\tau^{\f12}}+n^{(1-\gamma/2)}e^{\frac{|n|^\gamma\nu^{\f12}\tau}{\delta}}\Big)\|f_\nu\|_{L^2_{X,Y}},
\end{align}
by using a similar argument in the proof of derivative estimates in Proposition \ref{semigroup-L2}.

Collecting the above estimates, we obtain that for $\delta_0^{-1}\leq |n|\leq \delta_0^{-1}\nu^{-3/4}$ and $|n|^\gamma\nu^{\frac{1}{2}}<1$,
\begin{align*}
&\|v^{(n)}(\tau)\|_{L^2_XL^\infty_Y}\\
&\leq C|n|^{1-\gamma}e^{\frac{|n|^\gamma\nu^{\f12}\tau}{2\delta}}\big(\frac{1}{\nu^{1/8}\tau^{1/4}}+|n|^{\frac{1}{2}-\frac{\gamma}{4}}e^{\frac{|n|^\gamma\nu^{\f12}\tau}{2\delta}}\big)\|f_{\nu}\|_{L^2_{X,Y}}\\
&\quad+C|n|^{1/2-3\gamma/4}\big(\log(|n|)\big)^{\f12}\Big(\frac{1}{\nu^{\f14}\tau^{\f12}}+n^{(1-\gamma/2)}e^{\frac{|n|^\gamma\nu^{\f12}\tau}{\delta}}\Big)\|f_\nu\|_{L^2_{X,Y}}\\
&\quad+C|n|^{-\frac{1}{4}}\Big(\frac{1}{\nu^{\f14}\tau^{\f12}}+n^{\f54(1-\gamma)+\f12}e^{\frac{|n|^\gamma\nu^{\f12}\tau}{\delta}}\Big)\|f_\nu\|_{L^2_{X,Y}}\\
&\quad+C\nu^{\f14}|n|^{\f12}|n|^{2(1-\gamma)}e^{\frac{|n|^\gamma}{\delta}\nu^{\f12}\tau}\|f_{\nu}\|_{L^2_{X,Y}}\\
&\leq C|n|^{\f34(1-\gamma)}\log^\f12(|n|)\frac{1}{\nu^{\f14}|n|^{\f14}\tau^{\f12}}\|f_{\nu}\|_{L^2_{X,Y}}+C\frac{|n|^{1-\gamma}}{\nu^{\f18}\tau^{\f14}}e^{\frac{|n|^\gamma\nu^{\f12}\tau}{2\delta}}\|f_{\nu}\|_{L^2_{X,Y}}\\
&\quad+C\log^{\f12}(|n|)|n|^{\frac{3}{2}-\f54\gamma}e^{\frac{|n|^\gamma \nu^{\f12}\tau}{\delta}}\|f_{\nu}\|_{L^2_{X,Y}}.
\end{align*}
By scaling back to the original variables and \eqref{Rescalenorm}, we obtain that for $\delta_0^{-1}\leq |n|\leq \delta_0^{-1}\nu^{-3/4}$ and $|n|^\gamma\nu^{\frac{1}{2}}<1$,
\begin{align*}
&\|u^{(n)}\|_{L^2_xL^\infty_y}\\
&\leq C|n|^{\f34(1-\gamma)}\log^\f12(|n|)\frac{1}{\nu^{\f14}|n|^{\f14}t^{\f12}}\|f\|_{L^2}+C\frac{|n|^{1-\gamma}}{\nu^{\f14}t^{\f14}}e^{\frac{|n|^\gamma t}{2\delta}}\|f\|_{L^2}\\
&\qquad+C\nu^{-\f14}\log^{\f12}(|n|)|n|^{\frac{3}{2}-\f54\gamma}e^{\frac{|n|^\gamma t}{\delta}}\|f\|_{L^2}.
\end{align*}
When $|n|\le \delta_0^{-1}$, by (1) in Proposition \ref{semigroup-L2} and the interpolation, we get 
\beno
\|u^{(n)}\|_{L^2_xL^\infty_y}\le C\f 1 {\nu^\f14 t^\f14}e^{ct}\|f\|_{L^2}.
\eeno
This completes the proof of the first one in this proposition. 
\end{proof}

In Proposition \ref{semigoup-L2Linf}, all of the $L^\infty$ estimates of the semigroup $e^{-t\mathbb{A}_\nu}$ contain some singularity of $t$ at $t=0$, which means that these result could not apply to estimate the homogeneous part of the solution, as our goal is to obtain an $L^\infty$ stability up to $t=0$. Hence, we need to use the following $H^1-L^\infty$ semigroup estimates without the singularity at $t=0$. 

\begin{proposition}\label{semigoup-H^1Linf}
Under the same assumptions in Proposition \ref{semigroup-L2}, there holds for all $f\in \mathcal{P}_nL^2_\sigma(\Omega)\cap\mathcal{P}_n L^2_xH^1_y(\Omega)$ and $t>0$.
\begin{enumerate}
\item If $|n|^\gamma\nu^{\f12}\leq1$, then
\begin{align}
\|&e^{-t\mathbb{A}_{\nu,n}}f\|_{L^2_xL^\infty_y}\leq C\|f\|_{L^2_xH^1_y}+C\nu^{-\f14}t^{\f34}(1+|n|^{3-2\gamma})e^{\frac{|n|^\gamma t}{\delta}}\|f\|_{L^2}.\nonumber
\end{align}

\item If $|n|^\gamma\nu^{\frac{1}{2}}\geq 1$ and $|n|\leq\delta_0^{-1}\nu^{-\f34}$, then
\begin{align*}
&\|e^{-t\mathbb{A}_{\nu,n}}f\|_{L^2_xL^\infty_y}\leq C\|f\|_{L^2_xH^1_y}+C\nu^{-\f14}t^{\f34}|n|^{2-\gamma}e^{\frac{|n|^\gamma t}{\delta}}\|f\|_{L^2}.
\end{align*}

\item If $|n|\geq\delta_0^{-1}\nu^{-\f34}$, then
\begin{align*}
&\|e^{-t\mathbb{A}_{\nu,n}}f\|_{L^2_xL^\infty_y}\leq C\|f\|_{L^2_xH^1_y}+C\nu^{-\f14}t^{\f34}|n|e^{-\f14\nu n^2t}\|f\|_{L^2}.
\end{align*}
\end{enumerate}
\end{proposition}
The main idea of the proof is that we treat the linear part generated by shear follow as a perturbation, and  then the semigroup $e^{-t\mathbb{A}_n}$ is a perturbation of Stokes semigroup $e^{\nu t\mathbb{P}\Delta}$.

\begin{proof}
Let $n\in\mathbb{Z}$ and $f\in\mathcal{P}_nL^2_xL^\infty_y(\Omega)\cap\mathcal{P}_nL^2_\sigma(\Omega)$.
 We still denote $u^{(n)}=e^{-t\mathbb{A}_{\nu,n}}f$. By Duhamel's formula, $u^{(n)}$ can be written as
\begin{align}
u^{(n)}(t)=e^{\nu t\mathbb{P}\Delta}f-\int_0^t e^{\nu(t-s)\mathbb{P}\Delta}\mathbb{P}\big(U^{p}(Y)\partial_x u^{(n)}(s)+(y^{-1}u_2^{(n)}(s)Y\partial_YU^p(Y),0)\big)\mathrm{d}s.
\end{align}
According to Lemma \ref{lem:stokes-semigroup}, we have
\begin{align*}
&\|u^{(n)}(t)\|_{L^2_xL^\infty_y}\\
&\leq C\|f\|_{L^2_xH^1_y}+C\sup_{0\leq s\leq t}\|\mathbb{P}\big(U^{p}(Y)\partial_x u^{(n)}(s)+(y^{-1}u_2^{(n)}(s)Y\partial_YU^p(Y),0)\big)\|_{L^2}\int_0^t\frac{1}{(\nu(t-s))^{\f14}}\mathrm{d}s\\
&\leq C\|f\|_{L^2_xH^1_y}+C\nu^{-\f14}t^{\f34}\big(\|\partial_xu^{(n)}(s)\|_{L^2}+\|y^{-1}u_2^{(n)}(s)\|_{L^2}\big),
\end{align*}
which along with Hardy inequality and the divergence free condition implies
\begin{align}\label{heat-L2-Linfinity}
\|&u^{(n)}(t)\|_{L^2_xL^\infty_y}\leq C\|f\|_{L^2_xH^1_y}+C\nu^{-\f14}t^{\f34}|n|\sup_{0\leq s\leq t}\|u^{(n)}(s)\|_{L^2}.
\end{align}

Next we consider the case of $|n|^\gamma\nu^{\f12}\leq1$. In this case, by the results (1) and (2) in Proposition \ref{semigroup-L2}, we have 
\begin{align*}
\|u^{(n)}(t)\|_{L^2}\leq C\big(1+|n|^{2(1-\gamma)}\big)e^{\frac{|n|^\gamma t}{\delta}}\|f\|_{L^2},
\end{align*}
which along with \reff{heat-L2-Linfinity} gives
\begin{align}\nonumber
\|&u^{(n)}(t)\|_{L^2_xL^\infty_y}\leq C\|f\|_{L^2_xH^1_y}+C\nu^{-\f14}t^{\f34}\big(1+|n|^{3-2\gamma}\big)e^{\frac{|n|^\gamma t}{\delta}}\|f\|_{L^2}.
\end{align}
In a similar way, we deduce that for $1\leq |n|^\gamma\nu^{\f12}$ and $|n|\leq\delta_0^{-1}\nu^{-\f34}$,
\begin{align*}
\|&u^{(n)}(t)\|_{L^2_xL^\infty_y}\leq C\|f\|_{L^2_xH^1_y}+C\nu^{-\f14}t^{\f34}|n|^{2-\gamma}e^{\frac{|n|^\gamma t}{\delta}}\|f\|_{L^2},
\end{align*}
and for $|n|\geq \delta_0^{-1}\nu^{-\f34}$,
\begin{align*}
\|&u^{(n)}(t)\|_{L^2_xL^\infty_y}\leq C\|f\|_{L^2_xH^1_y}+C\nu^{-\f14}t^{\f34}|n|e^{-\f14\nu n^2t}\|f\|_{L^2}.
\end{align*}
This completes the proof.
\end{proof}

\section{Nonlinear stability in Gevrey class}
In this section,  we prove the nonlinear stability.  By the Duhamel formula, the solution $u(t)$ of  the system \reff{NSA} with initial data $a$ is given by
\begin{align}
u(t)=e^{-t\mathbb{A}_\nu}a-\int_0^t e^{-t\mathbb{A}_{\nu}(t-s)}\mathbb{P}(u\cdot\nabla u)ds,\qquad t>0.
\end{align}
With respect to each Fourier mode $n$ of variable $x$, we know that the solution $u(t)$ can be represented as
\begin{align}
\mathcal{P}_n u(t)=e^{-t\mathbb{A}_{\nu,n}}\mathcal{P}_n a-\int_0^te^{-(t-s)\mathbb{A}_{\nu,n}}\mathcal{P}_n\mathbb{P}(u\cdot\nabla u)(s)ds.
\end{align}
Since we have already obtained the estimates of the semigroup $e^{-t\mathbb{A}_{\nu,n}}$ in Proposition \ref{semigroup-L2}, \ref{semigoup-L2Linf} and \ref{semigoup-H^1Linf}, our next task is to obtain nonlinear estimates.\smallskip

Before we start our proof, we recall some function spaces which we shall work on. For $\gamma\in[0,1]$, $d\geq 0$ and $K>0$, the Banach space $X_{d,\gamma,K}$ is given by
\begin{align*}
&X_{d,\gamma,K}=\Big\{f\in L^2_\sigma(\Omega): \|f\|_{X_{d,\gamma,K}}=\sup_{n\in\mathbb{Z}}(1+|n|^d)e^{K\an^{\gamma}}\|\mathcal{P}_n f\|_{L^2(\Omega)}<+\infty\Big\},\\
&X^{(1)}_{d,\gamma,K}=\Big\{f\in L^2_\sigma(\Omega): \|f\|_{X_{d,\gamma,K}^{(1)}}=\sup_{n\in\mathbb{Z}}(1+|n|^d)e^{K\an^{\gamma}}\|\mathcal{P}_n f\|_{L^2_xH^1_y(\Omega)}<\infty\Big\},
\end{align*}
and the space $Y_{d,\gamma,K}$ is defined as
\begin{align*}
Y_{d,\gamma,K}=\Big\{f\in L^2_\sigma(\Omega): \|f\|_{Y_{d,\gamma,K}}=\sup_{n\in\mathbb{Z}}(1+|n|^d)e^{K\an^{\gamma}}\|\mathcal{P}_n f\|_{L^2_xL^\infty_y(\Omega)}<\infty\Big\}.
\end{align*}

\begin{proof}[Proof of Theorem \ref{main}]
Let $\gamma\in[\f23,1)$ and $u(t)$ be the solution to the system \reff{NSA} with initial data $a$. Set
\begin{align}
q:=d-3(1-\gamma)-1\in(1,d),\qquad K(t):=K-2\delta^{-1}t.
\end{align}
In what follows, we always assume that $T\le \delta K/2\le 1$ so that $K(t)\ge 0$ for $t\in [0,T]$.

We would like to establish a priori estimate of $u(t)$ in the space
\begin{eqnarray}
\begin{split}
Z_{\gamma,K,T}:=&\Big\{f\in(C([0,T];L^2_\sigma(\Omega)):\|f\|_{Z_{\gamma,K,T}}:=\sup_{0<t\leq T'}\big(\|f(t)\|_{X_{q,\gamma,K(t)}}\\
&\quad+\nu^{\f14}\|f(t)\|_{Y_{q,\gamma,K(t)}}+(\nu t)^{\f12}\|\nabla f(t)\|_{X_{q,\gamma,K(t)}}\big)<+\infty\Big\}.
\end{split}
\end{eqnarray}

We first show a basic estimate for nonlinear term $\mathcal{P}_n\mathbb{P}(u\cdot\nabla u)$, which will be used frequently later. We notice that
\begin{align*}
\|\mathcal{P}_n\mathbb{P}(u\cdot\nabla u)\|_{L^2(\Omega)}&\leq \|\mathcal{P}_n(u_1\partial_x u)\|_{L^2(\Omega)}+\|\mathcal{P}_n(u_2\partial_y u)\|_{L^2(\Omega)}\\
&\leq \big\|\sum_{j\in\mathbb{Z}}(e^{-\mathrm{i}jx}\mathcal{P}_j u_1)\cdot(e^{-\mathrm{i}(n-j)x}\partial_x\mathcal{P}_{n-j}u)\big\|_{L^2_y(\mathbb{R}_+)}\\
&\quad+\big\|\sum_{j\in\mathbb{Z}}(e^{-\mathrm{i}jx}\mathcal{P}_j u_2)\cdot(e^{-\mathrm{i}(n-j)x}\partial_y\mathcal{P}_{n-j}u)\big\|_{L^2_y(\mathbb{R}_+)}.
\end{align*}
From Gagliardo-Nirenberg inequality and divergence free condition, we obtain
\begin{align*}
\big\|\sum_{j\in\mathbb{Z}}&(e^{-\mathrm{i}jx}\mathcal{P}_j u_1)\cdot(e^{-\mathrm{i}(n-j)x}\partial_x\mathcal{P}_{n-j}u)\big\|_{L^2_y(\mathbb{R}_+)}\\
&\leq C\sum_{j\in\mathbb{Z}}\|\mathcal{P}_j u\|_{L^2(\Omega)}^{\f12}\|\nabla\mathcal{P}_j u\|_{L^2(\Omega)}^{\f12}|n-j|^{\f12}\|\mathcal{P}_{n-j} u\|_{L^2(\Omega)}^{\f12}\|\nabla\mathcal{P}_{n-j} u\|_{L^2(\Omega)}^{\f12},
\end{align*}
and
\begin{align*}
\big\|\sum_{j\in\mathbb{Z}}&(e^{-\mathrm{i}jx}\mathcal{P}_j u_2)\cdot(e^{-\mathrm{i}(n-j)x}\partial_y\mathcal{P}_{n-j}u)\big\|_{L^2_y(\mathbb{R}_+)}
\leq C\sum_{j\in\mathbb{Z}}|j|^{\f12}\|\mathcal{P}_j u\|_{L^2(\Omega)}\|\nabla\mathcal{P}_{n-j} u\|_{L^2(\Omega)}.
\end{align*}
Therefore, for $u\in Z_{\gamma,K,T}$, we have 
\begin{align}\label{mainproof-nonlinear}
\|\mathcal{P}_n\mathbb{P}(u\cdot\nabla u)(t)\|_{L^2(\Omega)}\leq C(\nu t)^{-\f12}\frac{e^{-K(t)\an^\gamma}}{1+|n|^{q-\frac{1}{2}}}\|u\|_{Z_{\gamma,K,T}}^2
\end{align}
due to $q>1$.\smallskip

Now we are ready to show the estimates of $\mathcal{P}_n u(t)$. We split our proof into three situations: $|n|^\gamma\nu^{\f12}<1$, $1\leq |n|^\gamma\nu^{\f12}\leq \delta_0^{-\gamma}\nu^{-\f34\gamma+\f12}$ and $|n|\geq \delta_0^{-1}\nu^{-\f34}$.\smallskip

\no\textbf{Case 1.} $|n|^{\gamma}\nu^{\f12}<1$. In this case, we infer from Proposition \ref{semigroup-L2}  that
\begin{align*}
\|\mathcal{P}_n u(t)\|_{L^2(\Omega)}\leq& C(1+|n|^{2(1-\gamma)})e^{\frac{|n|^\gamma t}{\delta}}\|\mathcal{P}_n a\|_{L^2(\Omega)}\\
&+C(1+|n|^{2(1-\gamma)})\int_0^t e^{\frac{|n|^\gamma(t-s)}{\delta}}\|\mathcal{P}_n\mathbb{P}(u\cdot\nabla u)(s)\|_{L^2(\Omega)}\mathrm{d}s\\
\leq&C(1+|n|^{2(1-\gamma)})e^{\frac{|n|^\gamma t}{\delta}}\|\mathcal{P}_n a\|_{L^2(\Omega)}\\
&+C(1+|n|^{2(1-\gamma)})\int_0^te^{\frac{|n|^\gamma (t-s)}{\delta}}(\nu s)^{-\f12}\frac{e^{-K(s)\an^\gamma}}{1+|n|^{q-\frac{1}{2}}}\mathrm{d}s\|u\|_{Z_{\gamma,K,T}}^2.
\end{align*}
On the other hand, we notice that
\begin{eqnarray}\label{mainproof-int}
\begin{split}
\int_0^{t}e^{\frac{|n|^\gamma(t-s)}{\delta}}e^{-K(s)\an^\gamma}s^{-\f12}\mathrm{d}s&\le e^{-K(t)\an^\gamma}\int_0^t e^{\frac{-\an^\gamma(t-s)}{\delta}}s^{-\f12}\mathrm{d}s\\
&\leq C e^{-K(t)\an^\gamma}\an^{-\frac{\gamma}{2}}.
\end{split}
\end{eqnarray}
Therefore, we obtain
\begin{align*}
\|\mathcal{P}_n u(t)\|_{L^2(\Omega)}\leq& C\frac{e^{-K(t)\an ^\gamma}}{1+|n|^{d-2(1-\gamma)}}\|a\|_{X_{d,\gamma,K}}\\
&+\frac{C\an^{\f12-\f \gamma 2+2(1-\gamma)}}{\nu^{\f12}(1+|n|^q)}e^{-K(t)\an^\gamma}\|u\|^2_{Z_{\gamma,K,T}}.
\end{align*}
This shows that for $\beta_0=\frac{5(1-\gamma)}{4\gamma}$,
\begin{align}\label{mainproof-case1-L^2}
&\sup_{0<t\leq T}\sup_{|n|^\gamma\nu^{\f12}<1}(1+|n|^q)e^{K(t)\an^\gamma}\|\mathcal{P}_n u(t)\|_{L^2(\Omega)}\nonumber\\
&\leq C\Big(\|a\|_{X_{d,\gamma,K}}+\nu^{-\f12}\|u\|^2_{Z_{\gamma,K,T}}\sup_{|n|^\gamma\nu^{\f12}<1}\an^{\frac 52(1-\gamma)}\Big)\nonumber\\
&\leq C\Big(\|a\|_{X_{d,\gamma,K}}+\nu^{-\f12-\beta_0}\|u\|^2_{Z_{\gamma,K,T}}\Big).
\end{align}

Now we turn to show the $L^2_xL^\infty_y$ estimates in this case. By Proposition \ref{semigoup-L2Linf} and Proposition \ref{semigoup-H^1Linf}, we have
\begin{align*}
\|\mathcal{P}_n u(t)\|_{L^2_xL^\infty_y(\Omega)}\leq& C\|\mathcal{P}_n a\|_{L^2_xH^1_y(\Omega)}+C\nu^{-\frac{1}{4}}t^{\f34}\an^{3-2\gamma}e^{\frac{|n|^\gamma t}{\delta}}\|\mathcal{P}_n a\|_{L^2(\Omega)}\\
&+\frac{C}{\nu^{\f14}\an^{\f14}}\an^{\f34(1-\gamma)}\log^\f12\an\int_0^t\frac{1}{(t-s)^{\f12}}\|\mathcal{P}_n\mathbb{P}(u\cdot\nabla u)(s)\|_{L^2(\Omega)}\mathrm{d}s\\
&+\frac{C\an^{1-\gamma}}{\nu^{\f14}}\int_0^t(t-s)^{-\f14}e^{\frac{\an^\gamma(t-s)}{2\delta}}\|\mathcal{P}_n\mathbb{P}(u\cdot\nabla u)(s)\|_{L^2(\Omega)}\mathrm{d}s\\
&+C\nu^{-\f14}\log^{\f12}\an\an^{\frac{3}{2}-\f54\gamma}\int_0^t e^{\frac{|n|^\gamma(t-s)}{\delta}}\|\mathcal{P}_n\mathbb{P}(u\cdot\nabla u)(s)\|_{L^2(\Omega)}\mathrm{d}s,
\end{align*}
which along with \reff{mainproof-nonlinear} implies
\begin{align*}
\|\mathcal{P}_n u(t)\|_{L^2_xL^\infty_y(\Omega)}\leq& C\|\mathcal{P}_n a\|_{L^2_xH^1_y(\Omega)}+C\nu^{-\frac{1}{4}}t^{\f34}\an^{3-2\gamma}e^{\frac{|n|^\gamma t}{\delta}}\|a\|_{L^2(\Omega)}\\
&+\frac{C\an^{\f34(1-\gamma)}\log^\f12\an)}{\nu^{\f34}\an^{\f14}(1+|n|^{q-\f12})}\|u\|_{Z_{\gamma,K,T}}^2\int_0^t(t-s)^{-\f12}e^{-K(s)\an^\gamma}s^{-\f12}\mathrm{d}s\\
&+\frac{C\an^{1-\gamma}}{\nu^{\f34}(1+|n|^{q-\f12})}\|u\|^2_{Z_{\gamma,K,T}}\int_0^t(t-s)^{-\f14}e^{\frac{|n|^\gamma(t-s)}{2\delta}}e^{-K(s)\an^\gamma}s^{-\f12}\mathrm{d}s\\
&+\frac{C\log^{\f12}\an\an^{\frac{3}{2}-\f54\gamma}}{\nu^{\f34}(1+|n|^{q-\f12})}\|u\|^2_{Z_{\gamma,K,T}}\int_0^t e^{\frac{|n|^\gamma(t-s)}{\delta}}e^{-K(s)\an^\gamma}s^{-\f12}\mathrm{d}s.
\end{align*}
On the other hand, we notice that
\begin{align*}
&\int_0^t(t-s)^{-\f12}e^{-K(s)\an^\gamma}s^{-\f12}\mathrm{d}s\leq C e^{-K(t)\an^\gamma},\\
&\int_0^t(t-s)^{-\f14}e^{\frac{|n|^\gamma(t-s)}{2\delta}}e^{-K(s)\an^\gamma}s^{-\f12}\mathrm{d}s\leq C\an^{-\f\gamma 4}e^{-K(t)\an^\gamma}.
\end{align*}
Collecting \reff{mainproof-int} and the above two inequalities, we obtain
\begin{align*}
\|\mathcal{P}_n u(t)\|_{L^2_xL^\infty_y(\Omega)}\leq& C\|\mathcal{P}_n a\|_{L^2_xH^1_y(\Omega)}+C\nu^{-\frac{1}{4}}t^{\f34}\an^{3-2\gamma}e^{\frac{|n|^\gamma t}{\delta}}\|\mathcal{P}_na\|_{L^2(\Omega)}\\
&+\frac{C\an^{\f34(1-\gamma)+\f14}\log^\f12\an}{\nu^{\f34}(1+|n|^{q})}e^{-K(t)\an^\gamma}\|u\|_{Z_{\gamma,K,T}}^2\\
&+\frac{C\an^{1-\gamma+\f12-\f \gamma 4}}{\nu^{\f34}(1+|n|^{q})}e^{-K(t)\an^\gamma}\|u\|^2_{Z_{\gamma,K,T}}\\
&+\frac{C\log^{\f12}\an\an^{\frac{3}{4}+\f54(1-\gamma)-\f \gamma 2}}{\nu^{\f34}(1+|n|^{q})}e^{-K(t)\an^\gamma}\|u\|^2_{Z_{\gamma,K,T}}.
\end{align*}
This shows that 
\begin{align*}
&\sup_{0<t\leq T}\sup_{|n|^\gamma\nu^{\f12}<1}\nu^{\f14}(1+|n|^q)e^{K(t)\an^\gamma}\|\mathcal{P}_n u(t)\|_{L^2_xL^\infty_y(\Omega)}
\leq C\nu^{\f14}\|a\|_{X_{d,\gamma,K}^{(1)}}+C\|a\|_{X_{d,\gamma,K}}\\
&\quad+C\nu^{-\f12}\|u\|^2_{Z_{\gamma,K,T}}\sup_{|n|^\gamma\nu^{\f12}<1}\an^{\frac{1}{4}+\f34(1-\gamma)}\log^\f12\an
+C\nu^{-\f12}\|u\|^2_{Z_{\gamma,K,T}}\sup_{\an^\gamma\nu^{\f12}<1}\an^{\f14+\f54(1-\gamma)}\\
&\qquad+C\nu^{-\f12}\|u\|^2_{Z_{\gamma,K,T}}\sup_{|n|^\gamma\nu^{\f12}<1}\log^{\f12}\an\an^{\f14+\f74(1-\gamma)}.
\end{align*}
Taking $\beta_1=\frac{7(1-\gamma)}{8\gamma}+\frac{1}{8\gamma}+$, we conclude that 
\begin{eqnarray}\label{mainproof-case1-Linfinity}
\begin{split}
&\sup_{0<t\leq T}\sup_{|n|^\gamma\nu^{\f12}<1}\nu^{\f14}(1+|n|^q)e^{K(t)\an^\gamma}\|\mathcal{P}_n u(t)\|_{L^2_xL^\infty_y(\Omega)}\\
&
\leq C\|a\|_{X^{(1)}_{d,\gamma,K}}+C\nu^{-\f12-\beta_1}\|u\|^2_{Z_{\gamma,K,T}}.
\end{split}
\end{eqnarray}

Next we prove the $H^1$ estimates in this case. It follows from Proposition \ref{semigroup-L2} that 
\begin{align*}
\|\nabla\mathcal{P}_n u(t)\|_{L^2(\Omega)}\leq& \frac{C}{\nu^{1/2}}\Big(t^{-1/2}+\an^{\frac{1}{2}+\frac{5}{4}(1-\gamma)}e^{\frac{|n|^\gamma}{\delta}t}\Big)\|\mathcal{P}_n a\|_{L^2}\\
&+\frac{C}{\nu^{\f12}}\int_0^t\big((t-s)^{-\f12}+\an^{\frac{1}{2}+\frac{5}{4}(1-\gamma)}e^{\frac{|n|^\gamma}{\delta}(t-s)}\big)\|\mathcal{P}_n(u\cdot\nabla u)(s)\|_{L^2(\Omega)}\mathrm{d}s,
\end{align*}
which along with \reff{mainproof-nonlinear} and the fact $e^{\frac{|n|^\gamma t}{\delta}}\leq C(\frac{|n|^\gamma t}{\delta})^{-\f12}e^{\frac{2|n|^\gamma t}{\delta}}$ gives
\begin{eqnarray}\nonumber
\begin{split}
\|\nabla\mathcal{P}_n& u(t)\|_{L^2(\Omega)}
\leq\frac{C}{(\nu t)^{\f12}}\an^{\f74(1-\gamma)}e^{\frac{2|n|^\gamma t}{\delta}}\|\mathcal{P}_n a\|_{L^2(\Omega)}\\
&+\frac{C}{\nu(1+|n|^{q-\f12})}\|u\|_{Z_{\gamma,K,T}}^2\int_0^t\big((t-s)^{-\f12}+\an^{\frac{1}{2}+\frac{5}{4}(1-\gamma)}e^{\frac{\an^\gamma}{\delta}(t-s)}\big)\frac{e^{-K(s)|n|^\gamma}}{s^{\f12}}\mathrm{d}s.
\end{split}
\end{eqnarray}
On the other hand, we notice that
\begin{align*}
&\int_0^t(t-s)^{-\f12}{e^{-K(s)\an^\gamma}}{s^{-\f12}}\mathrm{d}s\leq Ct^{-\f12}\an^{-\f \gamma 2}e^{-K(t)\an^\gamma},\\
&\int_0^te^{\frac{|n|^\gamma}{\delta}(t-s)}{e^{-K(s)\an^\gamma}}{s^{-\f12}}\mathrm{d}s\leq Ct^{-\f 12}\an^{-\gamma }e^{-K(t)\an^\gamma}.
\end{align*}
Hence, we obtain
\begin{align*}
&(\nu t)^{\f12}\|\nabla\mathcal{P}_n u(t)\|_{L^2(\Omega)}\\
&\leq \frac{Ce^{-K(t)|n|^\gamma}}{(1+|n|)^{d-\frac{7}{4}(1-\gamma)}}\|a\|_{X^{(1)}_{d,\gamma,K}}+\frac{C\an^{\f12(1-\gamma)}e^{-K(t)\an^\gamma}}{\nu^{\f12}(1+|n|^q)}\|u\|_{Z_{\gamma,K,T}}^2\\
&\qquad+\frac{C\an^{\frac{9}{4}(1-\gamma)}e^{-K(t)\an^\gamma}}{\nu^\f12(1+|n|^q)}\|u\|_{Z_{\gamma,K,T}}^2,
\end{align*}
which gives
\begin{eqnarray}\label{mainproof-case1-na}
\begin{split}
&\sup_{0<t\leq T}\sup_{|n|^\gamma\nu^{\f12}<1}(1+|n|^q)e^{K(t)\an^\gamma}(\nu t)^{\f12}\|\nabla\mathcal{P}_n u(t)\|_{L^2(\Omega)}\\
&\leq C\|a\|_{X^{(1)}_{d,\gamma,K}}+C\nu^{-\f12-\beta_0}\|u\|^{2}_{Z_{\gamma,K,T}}.
\end{split}
\end{eqnarray}

\no\textbf{Case 2.} $1\leq |n|^\gamma\nu^{\f12}\leq \delta_0^{-\gamma}\nu^{-\f34\gamma+\f12}$. The argument is similar to Case 1. According to  Proposition \ref{semigroup-L2} and \reff{mainproof-nonlinear}, we have 
\begin{align}\label{mainproof-case2-L2}
&\sup_{0<t\leq T}\sup_{1\leq |n|^\gamma\nu^{\f12}\leq \delta_0^{-\gamma}\nu^{-\f34\gamma+\f12}}(1+|n|^q)e^{K(t)\an^\gamma}\|\mathcal{P}_n u(t)\|_{L^2(\Omega)}\nonumber\\
&\quad\leq C\|a\|_{X_{d,\gamma,K}}+C\nu^{-\f12}\|u\|_{Z_{\gamma,K,T}}^2\sup_{1\leq |n|^\gamma\nu^{\f12}\leq \delta_0^{-\gamma}\nu^{-\f34\gamma+\f12}}(1+|n|)^{\f32(1-\gamma)}\nonumber\\
&\quad\leq C\|a\|_{X_{d,\gamma,K}}+C\nu^{-\f12-\f98(1-\gamma)}\|u\|_{Z_{\gamma,K,T}}^2.
\end{align}
For the $L^\infty$-estimate, we infer from Proposition \ref{semigoup-L2Linf} and Proposition \ref{semigoup-H^1Linf}  that 
\begin{align*}
\|\mathcal{P}_n u(t)\|_{L^2_xL^\infty_y(\Omega)}\leq& C\|\mathcal{P}_n a\|_{L^2_xH^1_y(\Omega)}+C\nu^{-\frac{1}{4}}t^{\f34}|n|^{2-\gamma}e^{\frac{|n|^\gamma t}{\delta}}\|\mathcal{P}_na\|_{L^2(\Omega)}\\
&+\frac{C|n|^{\f14+\f 34(1-\gamma)}}{\nu^{\f34}(1+|n|^{q})}e^{-K(t)\an^\gamma}\|u\|^2_{Z_{\gamma,K,T}}\\
&+C\frac{(1+|n|)^{\frac{1}{4}+\f54(1-\gamma)}}{\nu^{\f34}(1+|n|^{q})}e^{-K(t)\an^\gamma}\|u\|^2_{Z_{\gamma,K,T}},
\end{align*}
which gives
\begin{eqnarray}\label{mainproof-case2-Linfinity}
\begin{split}
\sup_{0<t\leq T}&\sup_{1\leq |n|^\gamma\nu^{\f12}\leq \delta_0^{-\gamma}\nu^{-\f34\gamma+\f12}}\nu^{\f14}(1+|n|^q)e^{K(t)\an^\gamma}\|\mathcal{P}_n u(t)\|_{L^2_xL^\infty_y(\Omega)}\\
\leq&C\|a\|_{X^{(1)}_{d,\gamma,K}}+C\nu^{-\f12-\f 3{16}-\f {15} {16}(1-\gamma)}\|u\|^2_{Z_{\gamma,K,T}}.
\end{split}
\end{eqnarray}

For the $H^1$ estimate,  we infer from Proposition\ref{semigroup-L2} and \reff{mainproof-nonlinear} that 
\begin{eqnarray}\label{mainproof-case2-na}
\begin{split}
\sup_{0<t\leq T}&\sup_{1\leq |n|^\gamma\nu^{\f12}\leq \delta_0^{-\gamma}\nu^{-\f34\gamma+\f12}}(\nu t)^{\f12}(1+|n|^q)e^{K(t)\an^\gamma}\|\mathcal{P}_n u(t)\|_{L^2(\Omega)}\\
\leq&C\|a\|_{X^{(1)}_{d,\gamma,K}}+C\nu^{-\f12}\|u\|_{Z_{\gamma,K,T}}^2\sup_{1\leq |n|^\gamma\nu^{\f12}\leq \delta_0^{-\gamma}\nu^{-\f34\gamma+\f12}}\an^{\f32(1-\gamma)}\\
\leq& C\|a\|_{X^{(1)}_{d,\gamma,K}}+C\nu^{-\f12-\frac{9}{8}(1-\gamma)}\|u\|_{Z_{\gamma,K,T}}^2.
\end{split}
\end{eqnarray}

\no\textbf{Case 3.} $|n|>\delta_0^{-1}\nu^{-\f34}$. By Proposition \ref{semigroup-L2} and \reff{mainproof-nonlinear}, we first have
\begin{align*}
\|\mathcal{P}_n u(t)\|_{L^2(\Omega)}\leq C\|\mathcal{P}_n a\|_{L^2(\Omega)}+\frac{C\an^{\f12}}{\nu^{\f12}(1+|n|^q)}\int_0^te^{-\frac{1}{4}\nu|n|^2(t-s)}e^{-K(s)\an^\gamma}s^{-\f12}\mathrm{d}s\|u\|_{Z_{\gamma,K,T}}^2.
\end{align*}
We claim that for all $|n|>\delta_0^{-1}\nu^{-\f34}$
\begin{align}\label{mainproof-Inu}
I_\nu:=\an^{\f12}\int_0^te^{-\frac{1}{4}\nu|n|^2(t-s)}e^{-K(s)\an^\gamma}s^{-\f12}\mathrm{d}s\leq \frac{Ce^{-K(t)\an^\gamma}}{\nu^{\f12(1-\gamma)}}.
\end{align}
In fact, when $\delta_0^{-1}\nu^{-\f34}\leq|n|\leq\nu^{-1}$, we have
\begin{align*}
I_\nu\leq C|n|^{\f12}\int_0^te^{-K(s)\an^\gamma}s^{-\f12}\mathrm{d}s\leq C|n|^\f12 e^{-K(t)\an^\gamma}|n|^{-\frac{\gamma}{2}}\leq C\nu^{-\f12(1-\gamma)}e^{-K(t)\an^\gamma},
\end{align*}
and when $|n|\geq\nu^{-1}$, we have
\begin{align*}
I_\nu\leq C\int_0^t\nu^{-\f14}(t-s)^{-\f14}e^{-K(s)\an^\gamma}s^{-\f12}\mathrm{d}s\leq C\nu^{-\f14}|n|^{-\frac{\gamma}{4}}e^{-K(t)\an^\gamma}\leq C\nu^{-\f12(1-\gamma)}e^{-K(t)\an^\gamma},
\end{align*}
which gives \reff{mainproof-Inu}. Then we obtain 
\begin{eqnarray}\label{mainproof-case3-L2}
\begin{split}
&\sup_{0<t\leq T}\sup_{|n|^\gamma\nu^{\f12}>\delta_0^{-1}\nu^{-\f14}}(1+|n|^q)e^{K(t)\an^\gamma}\|\mathcal{P}_n u(t)\|_{L^2(\Omega)}\\
&\leq C\|a\|_{X^{(1)}_{d,\gamma,K}}+C\nu^{-\frac{1}{2}-\frac{1}{2}(1-\gamma)}\|u\|^2_{Z_{\gamma,K,T}}.
\end{split}
\end{eqnarray}

For the $L^\infty_y$ estimate, we infer that from Proposition \ref{semigoup-L2Linf},  Proposition \ref{semigoup-H^1Linf} and \reff{mainproof-nonlinear} that 
\begin{eqnarray}
\begin{split}
\|\mathcal{P}_n u(t)&\|_{L^2_xL^\infty(\Omega)}\leq C\|\mathcal{P}_n  a\|_{L^2_xH^1_y(\Omega)}+C|n|\nu^{-\f14}\|\mathcal{P}_n a\|_{L^2(\Omega)}\\
&+\frac{C\an^{\f12}}{\nu^{\f34}(1+|n|^q)}\int_0^t\frac{1+|n|^{\f12}(t-s)^{\f12}}{(t-s)^{\f14}}e^{-\f14\nu |n|^2(t-s)}e^{-K(s)\an^\gamma}s^{-\f12}\mathrm{d}s\|u\|_{Z_{\gamma,K,T}}^2.
\end{split}
\end{eqnarray}
On the other hand, we notice that for all $|n|>\delta_0^{-1}\nu^{-\f34}$
\begin{align*}
II_{\nu}:=&\an^{\f12}\int_0^t\frac{1+|n|^{\f12}(t-s)^{\f12}}{(t-s)^{\f14}}e^{-\f14\nu |n|^2(t-s)}e^{-K(s)\an^\gamma}s^{-\f12}\mathrm{d}s\\
\leq &C\int_0^t\Big(\frac{1}{\nu^{\f14}(t-s)^{\f12}}+\frac{1}{\nu^{\f12}(t-s)^{\f14}}\Big)e^{-K(s)\an^\gamma}s^{-\f12}\mathrm{d}s\\
\leq&C\nu^{-\f14}+\nu^{-\f12}\an^{-\f \gamma 4}.
\end{align*}
Then we obtain
\begin{eqnarray}\label{mainproof-case3-Linfinity}
\begin{split}
&\sup_{0<t\leq T}\sup_{|n|>\delta_0^{-1}\nu^{-\f34}}\nu^{\f14}(1+|n|^q)e^{K(t)\an^\gamma}\|\mathcal{P}_n u(t)\|_{L^2_xL^\infty_y(\Omega)}\\
&\leq C\|a\|_{X^{(1)}_{d,\gamma,K}}+C\nu^{-1+\frac{3}{16}\gamma}\|u\|_{Z_{\gamma,K,T}}^2.
\end{split}
\end{eqnarray}

For the $H^1$ estimate, we have
\begin{eqnarray}\nonumber
\begin{split}
&\|\nabla\mathcal{P}_n u(t)\|_{L^2(\Omega)}\leq\frac{C}{(\nu t)^{\f12}}(1+|n|t)e^{-\f14\nu n^2 t}\|\mathcal{P}_n a\|_{L^2(\Omega)}\\
&\qquad+\frac{C\an^\f12}{\nu(1+|n|^q)}\int_0^t\Big(\frac{1}{(t-s)^{\f12}}+|n|(t-s)^{\f12}\Big)e^{-\nu n^2(t-s)}e^{-K(s)\an^\gamma}s^{-\f12}\mathrm{d}s\|u\|^2_{Z_{\gamma,K,T}}.
\end{split}
\end{eqnarray}
By a similar argument as in \reff{mainproof-Inu}, we obtain
\begin{eqnarray*}
\begin{split}
&\an^\f12\int_0^t\Big(\frac{1}{(t-s)^{\f12}}+|n|(t-s)^{\f12}\Big)e^{-\nu n^2(t-s)}e^{-K(s)\an^\gamma}s^{-\f12}\mathrm{d}s\\
&\leq \frac{Ce^{-K(t)\an^\gamma}}{\nu^{\f32(1-\gamma)}t^{\f12}}.
\end{split}
\end{eqnarray*}
Then we deduce that 
\begin{eqnarray}\label{mainproof-case3-na}
\begin{split}
&\sup_{0<t\leq T}\sup_{|n|>\delta^{-1}_0\nu^{-\f34}}(\nu t)^{\f12}(1+|n|^q)e^{K(t)|n|^\gamma}\|\mathcal{P}_n u(t)\|_{L^2(\Omega)}\\
&\leq C\|a\|_{X^{(1)}_{d,\gamma,K}}+C\nu^{-\f12-\f32(1-\gamma)}\|u\|^2_{Z_{\gamma,K,T}}.
\end{split}
\end{eqnarray}

Now we denote
\begin{align*}
\beta=&\max\Big\{\f {5(1-\gamma)} {4\gamma},\f {7(1-\gamma)} {8\gamma}+\f1 {8\gamma}+,\f 98(1-\gamma),\f{3} {16}+\frac{15(1-\gamma)}{16}\Big\}\\
=&\max\Big\{\f {7(1-\gamma)} {8\gamma}+\f1 {8\gamma}+,\f{3} {16}+\frac{15(1-\gamma)}{16}\Big\}.
\end{align*}
Summing up \reff{mainproof-case1-L^2},
\reff{mainproof-case1-Linfinity}, \reff{mainproof-case1-na}, \reff{mainproof-case2-L2}, \reff{mainproof-case2-Linfinity}, \reff{mainproof-case2-na}, \reff{mainproof-case3-L2}, \reff{mainproof-case3-Linfinity} and \reff{mainproof-case3-na}, we conclude that 
\begin{align*}
\|u\|_{Z_{\gamma,K,T}}\leq C\Big(\|a\|_{X^{(1)}_{d,\gamma,K}}+C\nu^{-\frac{1}{2}-\beta}\|u\|^2_{Z_{\gamma,K,T}}\Big).
\end{align*}
Thanks to our assumption $\|a\|_{X^{(1)}_{d,\gamma,K}}\leq \epsilon\nu^{\f12+\beta}$, if we take $\epsilon$ small enough so that $C\epsilon\le  \f12$, then we deduce that 
\begin{align}
\|u\|_{Z_{\gamma,K,T}}\leq C\|a\|_{X^{(1)}_{d,\gamma,K}}.\nonumber
\end{align}
This completes the proof of Theorem \ref{main}.
\end{proof}

\appendix

\section{Estimates of the Stokes semigroup}
In this appendix,  we present some $L^\infty$ type estimates of the Stokes semigroup $e^{\nu t\mathbb{P}\Delta}$ on the half plane. 

\begin{lemma}\label{lem:stokes-semigroup}
If $u_0\in L^2_x H^1_y(\mathbb{T}\times\mathbb{R}_+)$, then it holds that 
\begin{align*}
 &\|e^{\nu t\mathbb{P}\Delta} u_0\|_{L^2_xL^\infty_y}\leq C\|u_0\|_{L^2_x H^1_y},\\
& \|e^{\nu t\mathbb{P}\Delta} u_0\|_{L^2_xL^\infty_y}\leq \frac{C}{(\nu t)^{\f14}}\|u_0\|_{L^2},\\
& \|\na e^{\nu t\mathbb{P}\Delta} u_0\|_{L^2}\leq \frac{C}{(\nu t)^{\f12}}\|u_0\|_{L^2}.
\end{align*}
\end{lemma}
\begin{proof}
For given $u_0\in L^2_x H^1_y(\mathbb{T}\times\mathbb{R}_+)$, $u^{(S)}:=e^{\nu t\mathbb{P}\Delta}u_0$ is the solution to the Stokes equation on the half plane:
\begin{align*}
\left\{
\begin{aligned}
&\partial_t u^{(S)}(x,y)-\nu\Delta u^{(S)}(x,y)+\nabla p^{(S)}(x,y)=0,\quad (x,y)\in\mathbb{T}\times\mathbb{R}_+,\\
&\nabla\cdot u^{(S)}(x,y)=0,\\
&u|_{y=0}=0,\quad u(t=0)=u_0.
\end{aligned}
\right.
\end{align*}
It is easy to see that 
\begin{align*}
\frac{1}{2}\frac{d}{ dt}\|u^{(S)}\|^2_{L^2}+\nu\|\nabla u^{(S)}\|^2_{L^2}=0,
\end{align*}
which gives
\begin{align}\label{eq:stokes1}
\|u^{(S)}(t)\|_{L^2}^2+2\nu\int_0^t\|\nabla u^{(S)}(\tau)\|_{L^2}^2d\tau \leq \|u_0\|_{L^2}^2.
\end{align}
By taking the $L^2$ inner product with $\partial_t u^{(S)}$, we find that 
\begin{align*}
&\int_{\mathbb{T}\times\mathbb{R}_+}|\partial_t u^{(S)}|^2dxdy-\nu\int_{\mathbb{T}\times\mathbb{R}_+}\Delta u^{(S)}\partial_t u^{(S)}dxdy\\
&=\int_{\mathbb{T}\times\mathbb{R}_+}|\partial_t u^{(S)}|^2dxdy+\frac{\nu}{2}\frac{d}{dt}\|\nabla u^{(S)}\|^2_{L^2_{x,y}}=0,
\end{align*}
which implies that for any $t>0$
\begin{align}\label{eq:stokes2}
\|\partial_y u^{(S)}(t)\|_{L^2}\leq \|\partial_y u_0\|_{L^2}.
\end{align}
Therefore by collecting \eqref{eq:stokes1}, \eqref{eq:stokes2} and the interpolation, we deduce that for any $t>0$
\begin{align*}
\|u^{(S)}(t)\|_{L^2_xL^\infty_y}\leq C\|u^{(S)}(t)\|^{\f12}_{L^2}\|\partial_y u^{(S)}(t)\|_{L^2}^{\f12}\leq C\|u_0\|_{L^2_x H^1_y}.
\end{align*}
This proves the first inequality of this lemma.

To prove the other two inequalities, we make the following weighted estimate 
\begin{align*}
&t\int_{\mathbb{T}\times\mathbb{R}_+} (\partial_t u^{(S)})\pa_tu^{(S)}dxdy-\nu t\int_{\mathbb{T}\times\mathbb{R}_+}(\Delta u^{(S)})\pa_tu^{(S)}dxdy\\
&=t\|\pa_tu^{(S)}(t)\|^2_{L^2}+\f12\nu\f d {dt}\big(t\|\nabla u^{(S)}(t)\|^2_{L^2}\big)-\f12\nu\|\nabla u^{(S)}\|_{L^2}^2=0,
\end{align*}
which along with \eqref{eq:stokes1} gives 
\begin{align*}
\|\na u^{(S)}\|_{L^2}\leq \frac{1}{(2\nu t)^{\f12}}\|u_0\|_{L^2}.
\end{align*}
Again by the interpolation, we obtain 
\begin{align*}
\|u^{(S)}(t)\|_{L^2_xL^\infty_y}\leq \frac{C}{(\nu t)^{\f14}}\|u_0\|_{L^2}.
\end{align*}
The proof is completed.
\end{proof}

\section{Interpolation inequality}

\begin{lemma}\label{lem:inter}
  Let $\varphi\in H^2(\mathbb{R}_{+})$ solve $(\partial_Y^2-\alpha^2)\varphi=w$ with $\varphi(0)=0$. Then for any function $0\leq \rho\leq 1$, it holds that
  \begin{align*}
      \|(\partial_Y\varphi,\alpha \varphi)\|_{L^\infty}\leq& C\|\rho^{\f12}w\|_{L^2}^{\f12}\|(\partial_Y\varphi,\alpha \varphi)\|_{L^2}^{\f12}+ C\|(1-\rho^{\f12})w\|_{L^1} +C|\alpha|^{\f12}\|(\partial_Y\varphi,\alpha\varphi)\|_{L^2}.
  \end{align*}
\end{lemma}
\begin{proof} We have
    \begin{align*}
       & \|(\partial_Y\varphi,\alpha \varphi)\|_{L^\infty}^2 \leq C\||\partial_Y\varphi|^2+|\alpha \varphi|^2\|_{L^\infty}
       \leq C\big\|\partial_Y(|\partial_Y\varphi|^2+|\alpha \varphi|^2)\big\|_{L^1}.
    \end{align*}
Notice that
\begin{align*}
   & \big|\partial_Y(|\partial_Y\varphi|^2+|\alpha \varphi|^2)\big|\leq 2|\partial_Y^2\varphi||\partial_Y\varphi|+2 \alpha^2|\partial_Y\varphi||\varphi|\leq  2|w||\partial_Y\varphi|+4 \alpha^2|\partial_Y\varphi||\varphi|.
\end{align*}
Then we infer that 
\begin{align*}
    &\|(\partial_Y\varphi,\alpha \varphi)\|_{L^\infty}^2\leq C\big\||w||\partial_Y\varphi|\big\|_{L^1}+ C\alpha^2\big\||\partial_Y\varphi||\varphi|\big\|_{L^1}\\
    &\leq  C\big\||\rho^{\f12}w||\partial_Y\varphi|\big\|_{L^1}+ C\big\||(1-\rho^{\f12})w||\partial_Y\varphi|\big\|_{L^1}+ C\alpha^2\big\||\partial_Y\varphi||\varphi|\big\|_{L^1}\\
   &\leq C\|\rho^{\f12}w\|_{L^2}\|(\partial_Y\varphi,\alpha \varphi)\|_{L^2}+ C\|(1-\rho^{\f12})w\|_{L^1}\|\partial_Y\varphi\|_{L^\infty}+ C|\alpha|\|(\partial_Y\varphi,\alpha\varphi)\|_{L^2}^2.
\end{align*}
Then by Young's inequality, we get
 \begin{align*}
      \|(\partial_Y\varphi,\alpha \varphi)\|_{L^\infty}\leq& C\|\rho^{\f12}w\|_{L^2}^{\f12}\|(\partial_Y\varphi,\alpha \varphi)\|_{L^2}^{\f12}+ C\|(1-\rho^{\f12})w\|_{L^1} +C|\alpha|^{\f12}\|(\partial_Y\varphi,\alpha\varphi)\|_{L^2}.
  \end{align*}
  
The proof is completed. 
\end{proof}

\section{Some estimates of Airy Function}

 Let $Ai(y)$ be the Airy function, which is a nontrivial solution of $f''-yf=0$.  
 We denote
\begin{align*}
&A_0(z)=\int_{\mathrm{e}^{{\mathrm{i}}\pi/6}z}^{\infty}Ai(t)\mathrm{d}t =\mathrm{e}^{{i}\pi/6}\int_{z}^{\infty}Ai(\mathrm{e}^{{\mathrm{i}}\pi/6}t)\mathrm{d}t.
\end{align*}

The following lemma comes from \cite{CLWZ}.

\begin{lemma}\label{lem:Airy-p1}
There exists $c>0$ and $\delta_0>0$ so that for $\textbf{Im}(z)\le \delta_0$,
\begin{align}
&\left|\f{A_0'(z)}{A_0(z)}\right|\lesssim1+|z|^{\f12},\quad {\rm Re}\f{A_0'(z)}{A_0(z)}\leq\min(-1/3,-c(1+|z|^{\f12})).
\end{align}
Moreover, for ${{\bf Im }z}\le \delta_0$, we have
\beno
 \Big|\f{A_0''(z)}{A_0(z)}\Big|\le C(1+|z|).
 \eeno
\end{lemma}

We denote 
\beno
\tilde{A}(Y):=Ai(e^{\mathrm{i}\frac{\pi}{6}}\kappa(Y+\eta))/Ai(e^{\mathrm{i}\frac{\pi}{6}}\kappa\eta)
\eeno
with $\kappa>0$ and $\mathbf{Im}\eta<0$. We define $\tilde{\Phi}(Y)$ as the solution of 
\beno
(\partial_Y^2-\alpha^2)\tilde{\Phi}=\tilde{A},\quad \tilde{\Phi}(0)=0.
\eeno

To estimate $\tilde{A}(Y)$, we need the following lemmas.

\begin{lemma}\label{lem:Uineq}
  There exists a positive constant $C>0$, such that $\forall \ z\in\mathbb{C},\ t>0$, it holds
  \begin{align*}
     & \int_{0}^{t}|z-s|^{\f12}\mathrm{d}s\geq C^{-1}|z|^{\f12}t.
  \end{align*}
\end{lemma}
\begin{proof}
  Let $z_r=\mathbf{Re}(z),\ z_i=\mathbf{Im}(z)$. Let us first claim that
  \begin{align}\label{est:Uineq Re}
     & \int_{0}^{t}|z_r-s|^{\f12}\mathrm{d}s\geq C^{-1}|z_r|^{\f12}t.
  \end{align}
  Once \eqref{est:Uineq Re} holds,  we have
  \begin{align*}
     & \int_{0}^{t}|z-s|^{\f12}\mathrm{d}s\geq C^{-1} \big( \int_{0}^{t}|z_r-s|^{\f12}\mathrm{d}s+ \int_{0}^{t}|z_i|^{\f12}\mathrm{d}s\big)\geq C^{-1}\big(|z_r|^{\f12}t+|z_i|^{\f12}t\big)\geq C^{-1}|z|^{\f12}t.
  \end{align*}
  It remains to prove \eqref{est:Uineq Re}.\smallskip

  \textbf{Case 1}. $z_r\leq 0$.  In this case, we have
  $$\int_{0}^{t}|z_r-s|^{\f12}\mathrm{d}s\geq \int_{0}^{t}|z_r|^{\f12}\mathrm{d}s= |z_r|^{\f12}t.$$

  \textbf{Case 2}. $0\leq z_r\leq t/2$. In this case, we have
  \begin{align*}
     \int_{0}^{t}|z_r-s|^{\f12}\mathrm{d}s&\geq \int_{t/2}^{t}|z_r-s|^{\f12}\mathrm{d}s=  \int_{t/2}^{t}(s-z_r)^{\f12}\mathrm{d}s\geq \int_{t/2}^{t}(s-t/2)^{\f12}\mathrm{d}s\\
     &= \dfrac{2(t-t/2)^{\f32}}{3}=\dfrac{2(t/2)^{\f32}}{3}\geq \dfrac{z_r^{\f12}t}{3}=\dfrac{|z_r|^{\f12}t}{3}.
  \end{align*}

  \textbf{Case 3}. $z_r\geq t/2$. In this case, we have
  \begin{align*}
      \int_{0}^{t}|z_r-s|^{\f12}\mathrm{d}s&\geq  \int_{0}^{t/4}|z_r-s|^{\f12}\mathrm{d}s= \int_{0}^{t/4}|z_r-s|^{\f12}\mathrm{d}s\geq  \int_{0}^{t/4}(z_r-t/4)^{\f12}\mathrm{d}s\\
      &=\dfrac{(z_r-t/4)^{\f12}t}{4}\geq \dfrac{|z_r/2|^{\f12}t}{4},
  \end{align*}
  here we used $z_r-t/4\geq z_r/2=|z_r|/2$.

Combining three cases, we conclude our result.
\end{proof}

\begin{lemma}\label{lem:ham-bound}
  If $w\in L^2(\mathbb{R}_+)$, and $\phi\in H^2$ satisfies
  \begin{align*}
     & (\partial_Y^2-\alpha^2)\phi=w,
  \end{align*}
  then it holds that
  \begin{align*}
     &-\partial_Y\phi(Y)=\int_{Y}^{+\infty}w(Z)\mathrm{e}^{-\alpha( Z-Y)}\mathrm{d}Z.
  \end{align*}
Specially, we have
\begin{align*}
   & -\partial_Y\phi(0)=\int_{0}^{+\infty}w(Y)\mathrm{e}^{-\alpha Y}\mathrm{d}Y.
\end{align*}
\end{lemma}
\begin{proof}
  Direct calculation gives
  \begin{align*}
     & \int_{Y}^{+\infty}w(Z)\mathrm{e}^{-\alpha Z}\mathrm{d}Z= \int_{Y}^{+\infty}\big((\partial_Z^2-\alpha^2)\phi(Z)\big)\mathrm{e}^{-\alpha Z}\mathrm{d}Z\\
     &\quad=\int_{Y}^{+\infty}\phi\big((\partial_Z^2-\alpha^2)\mathrm{e}^{-\alpha Z}\big)\mathrm{d}Z+\big((\partial_Z\phi)\mathrm{e}^{-\alpha Z}\big)\big|_{Y}^{+\infty}=-\partial_Y\phi(Y)\mathrm{e}^{-\alpha Y}.
  \end{align*}
\end{proof}

\begin{lemma}\label{lem:Airy-w}
Let $\kappa>0$ and $\mathbf{Im}\eta<0$. Then we have
\begin{align*}
&|\tilde{A}(Y)|\leq C\mathrm{e}^{-c\kappa\big(1+|\kappa\eta|^{\f12}\big)},\qquad \|\tilde{A}\|_{L^2}\leq C\kappa^{-\f12}\big(1+|\kappa\eta|)^{-\f14},\\
&\|Y\tilde{A}\|_{L^2}\leq C\kappa^{-\f32}\big(1+|\kappa\eta|)^{-\f34},\qquad \|Y^2\tilde{A}\|_{L^2}\leq C\kappa^{-\f52}\big(1+|\kappa\eta|)^{-\f54},\\
&\|(\partial_Y\tilde{\Phi},\alpha\tilde{\Phi})\|_{L^2}\leq C\kappa^{-\f32}\big(1+|\kappa\eta|)^{-\f34}.
\end{align*}
Moreover,  there holds 
\begin{align*}
   & |\partial_Y\tilde{\Phi}(Y)|\leq C\kappa^{-1}\big(1+|\kappa\eta
   |)^{-\f12}\mathrm{e}^{-c\kappa Y\big(1+|\kappa\eta|^{\f12}\big)},\\
   &\|\tilde{\Phi}\|_{L^2}+\|Y\partial_Y\tilde{\Phi}\|_{L^2}\leq C\kappa^{-\f52}\big(1+|\kappa\eta|)^{-\f54},\\
   &\big\|Y\tilde{\Phi}\big\|_{L^2}+\big\|Y^2\partial_Y\tilde{\Phi}\big\|_{L^2}\leq C\kappa^{-\f72}\big(1+|\kappa\eta|)^{-\f74},\\
   &\big\|Y^2\tilde{\Phi}\big\|_{L^2}+\big\|Y^3\partial_Y\tilde{\Phi}\big\|_{L^2}\leq C\kappa^{-\f92}\big(1+|\kappa\eta|)^{-\f94}.
\end{align*}
\end{lemma}
\begin{proof}
 By Lemma \ref{lem:Airy-p1}, we have
  \begin{align*}
     &\left|\dfrac{A_0(t+B)}{A_0(B)}\right|= \left|\exp\big(\ln(A_0\big(t+B\big))-\ln(A_0(B))\big)\right| = \left|\exp\bigg(\int_{0}^{t}\dfrac{A_0'\big(s+B\big)}{A_0\big(s+B\big)}\mathrm{d}s\bigg)\right|\\
     &\leq  \exp\bigg(\int_{0}^{t}\mathbf{Re}\dfrac{A_0'\big(s+B\big)}{A_0 \big(s+B\big)}\mathrm{d}s\bigg) \leq \exp\bigg(-\int_{0}^{t}\max\big(1/3,c(1+|s+B|^{\f12})\big)\mathrm{d}s\bigg),
  \end{align*}
which along with Lemma \ref{lem:Uineq} implies
\begin{align}\label{A0tB}
\left|\dfrac{A_0(t+B)}{A_0(B)}\right|\leq \exp\left(-\max\big(t/3,c(1+|B|^{\f12})t\big)\right).
\end{align}
Thanks to $\mathbf{Re}\dfrac{A_0'(z)}{A_0(z)}\leq \min(-1/3,-c(1+|z|^{\f12})\big)<0$, $\left|\mathbf{Re}\dfrac{A_0'(z)}{A_0(z)}\right| \geq c(1+|z|^{\f12})$ and
  \begin{align}\label{est:A/A'}
     & \left|\dfrac{A_0(z)}{A_0'(z)}\right|=\left|\dfrac{A_0'(z)}{A_0(z)}\right|^{-1}\leq \left|\mathbf{Re}\dfrac{A_0'(z)}{A_0(z)}\right|^{-1} \leq c^{-1}(1+|z|^{\f12})^{-1}.
  \end{align}

Now we are ready to show the estimates about $\tilde{A}(Y)$.
 Lemma \ref{lem:Airy-p1} gives
    \begin{align*}
     |\tilde{A}(Y)|&=\left|\dfrac{A_0'\big(\kappa(Y+\eta)\big)}{A_0'(\kappa\eta)} \right|= \left|\dfrac{A_0(\kappa\eta)}{A_0'(\kappa\eta}\right| \left|\dfrac{A_0\big(\kappa(Y+\eta)\big)}{A_0(\kappa\eta)}\right| \left|\dfrac{A_0'\big(\kappa(Y+\eta)\big)}{A_0\big(\kappa(Y+\eta)\big)}\right|\\
   & \leq C(1+|\kappa\eta|)^{-\f12}\big(1+|\kappa\eta|+\kappa Y\big)^{\f12} \mathrm{e}^{-c\kappa Y\big(1+|\kappa\eta|^{\f12}\big)}\\
     &\leq C\mathrm{e}^{-c\kappa Y\big(1+|\kappa\eta|^{\f12}\big)}.
  \end{align*}
  Then we find that
  \begin{align*}
     & \|\tilde{A}\|_{L^2}\leq C\big\|\mathrm{e}^{-c\kappa Y\big(1+|\kappa\eta|^{\f12}\big)}\big\|_{L^2}\leq C\kappa^{-\f12}\big(1+|\kappa\eta|)^{-\f14},\\
     & \|Y\tilde{A}\|_{L^2}\leq C\big\|Y\mathrm{e}^{-c\kappa Y\big(1+|\kappa\eta|^{\f12}\big)}\big\|_{L^2}\leq C\kappa^{-\f32}\big(1+|\kappa\eta|)^{-\f34},\\
     &\|Y^2\tilde{A}\|_{L^2}\leq C\big\|Y^2 \mathrm{e}^{-c\kappa Y\big(1+|\kappa\eta|^{\f12}\big)}\big\|_{L^2}\leq C\kappa^{-\f52}\big(1+|\kappa\eta|)^{-\f54}.
  \end{align*}

Now we turn to deal with $\tilde{\Phi}(Y)$. By duality argument and Hardy's inequality, we obtain
  \begin{align*}
     & \|(\partial_Y\tilde{\Phi},\alpha \tilde{\Phi})\|_{L^2}^2=\big|\langle\tilde{A} ,\tilde{\Phi}\rangle\big| \leq \|Y\tilde{A}\|_{L^2}\big\|\tilde{\Phi}/Y\big\|_{L^2}\leq 2\|Y\tilde{A}\|_{L^2}\|\partial_Y\tilde{\Phi}\|_{L^2},
  \end{align*}
  which gives
  \begin{align*}
     &\|(\partial_Y\tilde{\Phi},\alpha \tilde{\Phi})\|_{L^2}\leq 2\|Y\tilde{A}\|_{L^2}\leq C\kappa^{-\f32}\big(1+|\kappa\eta|)^{-\f34}.
  \end{align*}
By Lemma \ref{lem:ham-bound} and $|\tilde{A}(Y)|\leq C\mathrm{e}^{-c\kappa\big(1+|\kappa\eta|^{\f12}\big)Y}$, we have
  \begin{align*}
     \left|-\mathrm{e}^{-\alpha Y}\partial_Y\tilde{\Phi}(Y)\right| &=\left|\int_{Y}^{+\infty}\tilde{A}(Z)\mathrm{e}^{-\alpha Z}\mathrm{d}Z\right| \leq \int_{Y}^{+\infty}C\mathrm{e}^{-c\kappa\big(1+|\kappa\eta|^{\f12}\big)Z}\mathrm{e}^{-\alpha Z}\mathrm{d}Z\\
     &\leq C\big(c\kappa\big(1+|\kappa\eta|^{\f12}\big)+\alpha\big)^{-1} \mathrm{e}^{-c\kappa\big(1+|\kappa\eta|^{\f12}\big)Y-\alpha Y}\\
     &\leq C\kappa^{-1}\big(1+|\kappa\eta|^{\f12}\big)^{-1} \mathrm{e}^{-c\kappa\big(1+|\kappa\eta|^{\f12}\big)Y-\alpha Y}.
  \end{align*}
  Therefore, we obtain
  \begin{align*}
     &  |\partial_Y\tilde{\Phi}(Y)|\leq  C \kappa^{-1}(1+|\kappa\eta|)^{-\f12} \mathrm{e}^{-c\big(1+|\kappa\eta|^{\f12}\big)\kappa Y},
  \end{align*}
 which implies 
  \begin{align*}
     & \big\|Y\partial_Y\tilde{\Phi}\big\|_{L^2}\leq C\kappa^{-1}(1+|\kappa\eta|)^{-\f12} \big\|Y\mathrm{e}^{-c\big(1+|\kappa\eta|^{\f12}\big)\kappa Y}\big\|_{L^2}\leq C\kappa^{-\f52}\big(1+|\kappa\eta|)^{-\f54},\\
     &\big\|Y^2\partial_Y\tilde{\Phi}\big\|_{L^2}\leq  C\kappa^{-1}(1+|\kappa\eta|)^{-\f12} \big\|Y^2\mathrm{e}^{-c\big(1+|\kappa\eta|^{\f12}\big)\kappa Y}\big\|_{L^2}\leq C\kappa^{-\f72}\big(1+|\kappa\eta|)^{-\f74},\\
     &\big\|Y^3\partial_Y\tilde{\Phi}\big\|_{L^2}\leq  C\kappa^{-1}(1+|\kappa\eta|)^{-\f12} \big\|Y^3\mathrm{e}^{-c\big(1+|\kappa\eta^{\f12}\big)\kappa Y}\big\|_{L^2} \leq C\kappa^{-\f92}\big(1+|\kappa\eta|)^{-\f94}.
  \end{align*}
 For $\beta=\{0,1,2\}$, we have
  \begin{align*}
     &(2\beta+1)\|Y^{\beta}\tilde{\Phi}\|_{L^2}^2 =\big\langle |\tilde{\Phi}|^2,\partial_Y\big(Y^{2\beta+1}\big)\big\rangle=-\big\langle\partial_Y(|\tilde{\Phi}|^2),Y^{2\beta+1} \big\rangle\\
     &=-2\mathbf{Re}\langle Y^{\beta}\tilde{\Phi}, Y^{\beta+1}\partial_Y\tilde{\Phi}\rangle\leq 2\big\|Y^{\beta}\tilde{\Phi}\big\|_{L^2}\big\|Y^{\beta+1}\partial_Y\tilde{\Phi}\big\|_{L^2}.
  \end{align*}
 Then we obtain
  \begin{align*}
     &\big\|Y^{\beta}\tilde{\Phi}\big\|_{L^2}\leq \dfrac{2}{2\beta+1}\big\|Y^{\beta+1}\partial_Y\tilde{\Phi}\big\|_{L^2}\leq C\kappa^{-\f{5+2\beta}{2}}\big(1+|\kappa\eta|)^{-\f{5+2\beta}{4}}.
  \end{align*}
  
This proves the lemma.
 \end{proof}

\begin{lemma}\label{lem:Airy-bound}
Let $\kappa>0$ and $\mathbf{Im}\eta<0$. Then it holds that
  \begin{align*}
     &|\partial_Y\tilde{\Phi}(0)|\geq C^{-1}(1+|\kappa\eta|)^{-\f12}(\kappa+3\alpha)^{-1}.
  \end{align*}
\end{lemma}
\begin{proof}
  By Lemma \ref{lem:ham-bound}, we have
  \begin{align*}
     -\partial_Y\tilde{\Phi}(0) &=\int_{0}^{+\infty}\tilde{A}(Y)\mathrm{e}^{-\alpha Y}\mathrm{d}Y = \int_{0}^{+\infty}\dfrac{Ai\big( \mathrm{e}^{\mathrm{i}\frac{\pi}{6}}\kappa(Y+\eta)\big)}{ Ai\big(\mathrm{e}^{\mathrm{i}\frac{\pi}{6}}\kappa\eta\big)}\mathrm{e}^{-\alpha Y}\mathrm{d}Y\\
     &= \int_{0}^{+\infty}\dfrac{A_0'\big(\kappa(Y+\eta)\big)}{ A_0'(\kappa\eta)}\mathrm{e}^{-\alpha Y}\mathrm{d}Y\\
     &=-\dfrac{A_0(\kappa\eta)}{\kappa A_0'(\kappa\eta)} -\int_{0}^{+\infty}\dfrac{A_0\big(\kappa(Y+\eta)\big)}{ \kappa A_0'(\kappa\eta)}\partial_Y\big(\mathrm{e}^{-\alpha Y}\big)\mathrm{d}Y,
  \end{align*}
  which along with Lemma \ref{lem:Airy-w} gives
  \begin{align*}
     \kappa|A_0'(\kappa\eta)||\partial_Y\tilde{\Phi}(0)|&\geq \big|A_0(\kappa\eta)\big|- \int_{0}^{+\infty}\big|A_0\big(\kappa(Y+\eta)\big)\big| \big|\partial_Y\big(\mathrm{e}^{-\alpha Y}\big)\big|\mathrm{d}Y\\
     &\geq \big|A_0(\kappa\eta)\big|- \int_{0}^{+\infty} \mathrm{e}^{-\kappa Y/3} \big|A_0(\kappa\eta)\big|\big| \partial_Y\big(\mathrm{e}^{-\alpha Y}\big)\big|\mathrm{d}Y\\
     &= \big|A_0(\kappa\eta)\big|+ \big|A_0(\kappa\eta)\big|\int_{0}^{+\infty} \mathrm{e}^{-\kappa Y/3} \partial_Y\big(\mathrm{e}^{-\alpha Y}\big)\mathrm{d}Y\\
     &= \dfrac{\kappa}{3}\big|A_0(\kappa\eta)\big|\int_{0}^{+\infty} \mathrm{e}^{-\kappa Y/3-\alpha Y}\mathrm{d}Y\\
     &=\big|A_0(\kappa\eta)\big|\dfrac{\kappa}{ \kappa+3\alpha},
  \end{align*}
   which along with Lemma \ref{lem:Airy-p1} gives
     \begin{align*}
     |\partial_Y\tilde{\Phi}(0)|&\geq \dfrac{\big|A_0(\kappa\eta)\big|}{(\kappa+3\alpha)|A_0'(\kappa\eta)|}\geq C^{-1}(1+|\kappa\eta|)^{-\f12}(\kappa+3\alpha)^{-1},
  \end{align*}
  which gives our result.
  \end{proof}

\end{CJK*}
\end{document}